\documentclass[final,12pt]{elsarticle}

\usepackage{amssymb}
\usepackage{amsmath}

\usepackage{amsfonts}
\usepackage{bm}
\usepackage{mathrsfs}

\usepackage{verbatim} 
\usepackage{booktabs}
\usepackage{multirow}
\usepackage{subfig}
\usepackage{placeins}

\usepackage{xcolor}
\usepackage[normalem]{ulem}

\usepackage{algorithm}
\usepackage{algpseudocode}

\usepackage{epstopdf}
\usepackage{cases}
\usepackage{lineno}

\usepackage{amsthm}

\newtheorem{theorem}{Theorem}
\newtheorem{lemma}[theorem]{Lemma}

\newtheorem{remark}[theorem]{Remark}     
\usepackage[colorlinks,linkcolor=blue,citecolor=blue,urlcolor=blue]{hyperref}

\biboptions{sort&compress}

\begin{document}

\begin{frontmatter}



\title{A Single-Loop Bilevel Deep Learning Method  for Optimal Control of Obstacle Problems}



\affiliation[inst4]{organization={Pengcheng Laboratory},
            city={Shenzhen},
            postcode={518000}, 
            state={Guangdong},
            country={China}}

\affiliation[inst1]{organization={Division of Mathematical Sciences, School of Physical and Mathematical Sciences, Nanyang Technological University},
            addressline={21
Nanyang Link}, 
            postcode={637371}, 
            country={Singapore}}

\affiliation[inst2]{organization={National Center for Applied Mathematics~Shenzhen},
            addressline={Shenzhen,~518000, Guangdong},
            country={China}}

\affiliation[inst3]{organization={Department of Mathematics, Southern University of Science and Technology},
            city={Shenzhen},
            postcode={518055}, 
            state={Guangdong},
            country={China}}

\author[inst1]{Yongcun Song\fnref{equal}}
\author[inst2,inst3]{Shangzhi Zeng}
\author[inst2,inst3]{Jin Zhang}
\author[inst4,inst3]{Lvgang Zhang\corref{cor1}}
\ead{12443001@mail.sustech.edu.cn}
\cortext[cor1]{Corresponding author.}

\fntext[equal]{These authors contributed equally and are listed alphabetically.}
\begin{abstract}
Optimal control of obstacle problems arises in a wide range of applications and are computationally challenging due to their nonsmoothness, nonlinearity, and bilevel structure. Classical numerical methods typically rely on mesh-based discretizations and require solving a sequence of computationally expensive subproblems, while deep learning approaches for this class of problems remain largely unexplored.
In this work, we propose a single-loop bilevel deep learning framework. The proposed method is mesh-free, scalable to high-dimensional settings and complex domains, and avoids the repeated solution of discretized subproblems. It employs constraint-embedding neural networks to approximate both the state and the control while preserving the intrinsic bilevel structure of the problem. To enable efficient training, we develop a Single-Loop Stochastic First-Order Bilevel Algorithm (S2-FOBA), which eliminates nested optimization and does not rely on restrictive lower-level uniqueness assumptions. We further analyze the convergence behavior of S2-FOBA under mild conditions. Numerical experiments on benchmark problems, including two-dimensional examples with regular and irregular obstacles, a five-dimensional example, and a problem posed on a complex domain, demonstrate the effectiveness and efficiency of the proposed approach.
\end{abstract}


\end{frontmatter}


\section{Introduction}
As a fundamental class of nonsmooth nonlinear problems, obstacle problems arise in diverse applications where some obstacle constraints must be satisfied. Examples include elastic-plastic torsion problems \cite{glowinski2015variational}, elastic membrane deformation \cite{glowinski2008lectures}, lubrication phenomena \cite{cryer1971method}, porous media filtration \cite{kinderlehrer2000introduction}, wake problems \cite{bourgat1977numerical}, irrotational flows of perfect fluids \cite{brezis1976hodograph},  and American option pricing  \cite{jaillet1990variational}.  
In addition to the numerical simulation of obstacle problems, it is often of interest to control them in order to achieve specific objectives. Consequently, optimal control of obstacle problems arises in various fields, where a control is introduced to steer the system toward a prescribed target \cite{barbu1984optimal}. Such optimal control problems capture important applications in a wide range of areas and have been intensively studied in \cite{barbu1984optimal,bergounioux1997use,bergounioux2004optimal,hintermuller2008active,hintermuller2014dual,hintermuller2011smooth,hintermuller2015ell,hintermuller2016bundle,ito2007optimal,meyer2015adaptive,meyer2013priori,mignot1976controle,mignot1984optimal,schiela2013convergence,surowiec2018numerical}.

Optimal control of obstacle problems presents significant challenges from both theoretical and algorithmic perspectives due to the inherent nonconvexity and nondifferentiability. For example, the nondifferentiability of the control-to-state operator complicates the full gradient computation, and some tools from variational analysis and approximation theory are required to address this difficulty \cite{mignot1976controle,surowiec2018numerical}. 
In addition, the obstacle constraint is nonsmooth and gives rise to a free boundary across which the solution typically lacks regularity. Moreover, this free boundary may not align with mesh grids in finite difference (FDM) or finite element (FEM) methods. Consequently, numerical solutions may suffer from significant errors, particularly in the vicinity of the free boundary; see, e.g., \cite{de2018optimal,friedman2010variational}.

From a computational perspective, these challenges highlight the need to develop scalable algorithms that avoid repeatedly solving large, ill-conditioned algebraic systems arising from discretization and that remain robust across different problem settings. In general, the development of efficient algorithms requires careful consideration of the specific structure and features of the problem under investigation.

\subsection{Model}
We consider the optimal control of obstacle problems that can be modeled as
\begin{equation} \label{general control}
\min _{y \in Y, u \in U} J(y, u) \quad \text { s.t. } \quad y=\arg \min _{y' \in Y_{a d}} \mathscr{E}(y', u), \quad u \in U_{a d}.
\end{equation}
Above, $Y$ and $U$ are Hilbert spaces, $J: Y \times U \rightarrow \mathbb{R}$ is the objective functional,  and $y \in Y$ and $u \in U$ are the state and the control, respectively. The obstacle constraint $y \in Y_{ad}$ and the control constraint $u \in U_{ad}$ impose realistic restrictions on $y$ and $u$, with \( Y_{\text{ad}} \subset Y \) and \( U_{\text{ad}} \subset U \) being nonempty, convex, closed admissible sets. 
The state \( y \) is governed by an obstacle problem, defined as the unique minimizer of the energy functional \( \mathscr{E}: Y \times U \rightarrow \mathbb{R} \) over \( Y_{\text{ad}} \) for a given control \( u \in U_{\text{ad}} \) \cite{mignot1984optimal}. This structure leads to a bilevel problem:
\begin{itemize}
	\item \textbf{Lower-level}: For a fixed \( u \), find \( y \in Y_{\text{ad}} \) that minimizes \( \mathscr{E}(y, u) \).
	\item \textbf{Upper-level}: Find an optimal control \( u^* \in U_{\text{ad}} \) such that the pair \((y^*, u^*)\) minimizes \( J(y, u) \), where \( y^* \) is the solution to the obstacle problem with \( u^* \).
\end{itemize}
From a physical perspective, the lower-level problem models the equilibrium of the system under a given 
control, while the upper-level problem seeks a control that steers this 
equilibrium toward a desired target.

Note that the bilevel formulation in \eqref{general control} is essential rather than optional. 
Indeed, for each admissible control $u \in U_{\mathrm{ad}}$, the associated 
state $y$ is implicitly defined as the solution of an obstacle problem, 
i.e., as the minimizer \( \mathscr{E}(y, u) \) over $Y_{ad}$. 
The solution operator induced by this obstacle problem is, in general, nonsmooth, which precludes direct elimination of the lower-level problem.

\subsection{Classical Numerical Methods}

A variety of numerical methods have been proposed in the literature to address the problem \eqref{general control}. These methods can generally be classified into two categories. The first category addresses obstacle problems by smoothing them via penalty terms or relaxation techniques, see e.g., \cite{hintermuller2009mathematical,hintermuller2011smooth,hintermuller2015ell,schiela2013convergence}. While these approaches effectively alleviate the nondifferentiability of the solution operator, their numerical implementation  necessitates solving a sequence of approximate subproblems, which significantly increases the overall complexity.
Furthermore, these methods cannot guarantee the strict enforcement of the obstacle constraint \( y \in Y_{ad} \).  The second category tackles the problem \eqref{general control} with non-smooth optimization tools, such as subgradient and generalized-differentiation methods \cite{hintermuller2008active,hintermuller2016bundle,surowiec2018numerical}. These methods involve solving an obstacle problem at each iteration, typically using semismooth Newton or multigrid techniques. Detailed discussions on the numerical methods for solving \eqref{general control} can be found in \cite{surowiec2018numerical}.

Both categories of methods rely on mesh-based discretization techniques, such as FDMs or FEMs, which give rise to large-scale algebraic systems that are often ill-conditioned. Solving such systems usually necessitates the use of advanced numerical solvers combined with carefully designed preconditioning strategies, which can lead to significant computational and memory costs, particularly for fine meshes, complex geometries, or high-dimensional problems. Hence, these methods are typically limited to low-dimensional problems and struggle to scale to high-dimensional settings or complex domains.
Although adaptive finite element methods have been proposed to reduce the 
degrees of freedom \cite{brett2015mesh,gaevskaya2014adaptive,hintermuller2014dual,meyer2015adaptive}, their implementation remains challenging in 
practice due to the evolving free boundary between active and inactive 
regions, which complicates mesh refinement and coarsening strategies.

\subsection{Deep Learning Methods}

In recent years, deep learning techniques have achieved significant success in computational mathematics, particularly in solving partial differential equations (PDEs). These methods leverage the universal approximation capabilities of deep neural networks (NNs)      \cite{cybenko1989approximation,hornik1991approximation,hornik1989multilayer}, allowing for mesh-free solutions that are efficient in high-dimensional settings and complex domains. Notable approaches include physics-informed neural networks (PINNs)  \cite{raissi2019physics}, the Deep Ritz method  \cite{yu2018deep}, and the Deep Galerkin method  \cite{sirignano2018dgm}.

Building upon these advances, recent efforts have sought to leverage deep learning to solve obstacle problems, see e.g., \cite{alphonse2024neural,cheng2023deep,combettes2020deep,darehmiraki2022deep,el2025physics,gao2025prox,zhao2022two}. Note that these approaches are specifically designed for obstacle problems, but they do not address the 
optimal control problem~\eqref{general control}, where the obstacle problem appears as a constraint. Meanwhile, deep learning methods for optimal control of PDEs have been explored in \cite{barry2022physics,cao2025adversarial,hao2022bi,lai2025hard,lu2021physics,mowlavi2023optimal,song2024admm}. While deep learning has demonstrated significant success in both  obstacle problems and optimal control of PDEs, direct applications of existing deep learning methods to optimal control of obstacle problems are impractical or inefficient, as to be demonstrated in section \ref{se:existing}. 


\subsection{Methodology}

We propose a bilevel deep learning framework with a single-loop training structure for solving the problem~\eqref{general control}. 
 First, we design constraint-embedding NNs to approximate $u$ and $y$. Such NNs strictly enforce the constraints $y\in Y_{ad}$ and $u\in U_{ad}$, eliminating the need for penalty terms or active-set identification. Based on these NN approximations, we formulate a finite-dimensional stochastic bilevel optimization problem that approximates \eqref{general control} in a mesh-free manner.

To solve the resulting stochastic bilevel optimization problem, we develop a Single-Loop Stochastic First-Order Bilevel Algorithm (S2-FOBA), which enables efficient training of the NNs and offers several significant advantages. First, it does not require the lower-level singleton assumption, which is often violated because the NN-parameterized lower-level problem is highly nonconvex. Second, S2-FOBA is inherently stochastic; it approximates the integrals in $\mathcal{J}$ and $\mathcal{E}$ via Monte Carlo sampling using mini-batches at each iteration, which accelerates training and mitigates overfitting to a fixed set of collocation points. Additionally, S2-FOBA operates with a single-loop structure and only requires first-order gradient information, making it computationally and memory-efficient.  Finally, the convergence for S2-FOBA is analyzed under mild assumptions, which guarantees the training process is numerically stable. In contrast, the method in~\cite{hao2022bi} does not include a convergence analysis and provides only an error estimate in a deterministic setting, which does not apply to practical stochastic settings. Moreover, its validity requires the strong convexity of the lower-level problem, a restrictive assumption,
and it requires solving a large linear system at each control update, resulting in higher computational cost.

The proposed bilevel deep learning method preserves the intrinsic bilevel structure of \eqref{general control}, rather than formulating the problem as a weighted combination of the upper- and lower-level objectives. Crucially, such weighted formulations neglect the strict hierarchical dependence of the state $y$ on the control $u$. Hence, the computed solutions are often physically inconsistent and highly sensitive to the choice of weighting parameters, see section \ref{se:experiments} for the numerical demonstrations.
Compared with classical numerical methods, the proposed method does not require repeatedly solving discretized PDE-related subproblems. Moreover, it is mesh-free and therefore applicable to problems modeled in high-dimensional or complex domains. Owing to these properties, the method applies to multiple optimal control settings and extends naturally to optimal control of general elliptic variational inequalities (EVIs), as demonstrated in section \ref{se:extension}. 

Extensive numerical experiments are presented to evaluate the proposed method. In addition to standard two-dimensional benchmark examples, the numerical study includes a five-dimensional example and a problem posed on a complex domain, demonstrating the scalability and geometric flexibility of the framework. Comparisons with the widely used discretization-based active-set method \cite{hintermuller2008active} are also included. The results show that the proposed method achieves satisfactory accuracy, robust constraint enforcement, and favorable performance relative to existing deep learning approaches.


\subsection{{Contributions}}

The main contributions of this work can be summarized as follows.
\begin{itemize}
    \item We propose a bilevel deep learning framework for optimal control of obstacle problems that preserves the intrinsic hierarchical structure of the model, rather than replacing it by a weighted single-level formulation.
    
    \item We design constraint-embedding neural networks for the state and the control, so that the admissibility conditions are enforced by construction and no penalty tuning or active-set identification is needed during training.
    
    \item We develop the single-loop stochastic first-order algorithm S2-FOBA, based on a Moreau-envelope reformulation of the lower-level problem. The algorithm uses only first-order information, avoids nested solves, and does not rely on lower-level uniqueness.
    
    \item We establish convergence guarantees for S2-FOBA under mild assumptions, showing descent of a merit function and decay of a first-order stationarity measure for the penalized problem despite inexact lower-level updates.
    
    \item We validate the proposed framework on several benchmark problems, including two-dimensional cases with regular and irregular obstacles, a five-dimensional example, and a problem posed on a complex domain, and compare it with existing deep learning approaches and the discretization-based active-set method.
\end{itemize}

\subsection{Organization} The remainder of this paper is organized as follows. In section \ref{se:preliminary},  for later convenience, we specify the general model \eqref{general control} as a concrete distributed optimal control problem and then  summarize some existing results. Then, we introduce the NN approximation, develop the S2-FOBA training algorithm, and hence present the proposed bilevel deep learning method in section \ref{se: algorithm}. The convergence behavior of S2-FOBA is analyzed in section \ref{se:convergence}. Numerical experiments on several benchmark examples are presented in section \ref{se:experiments}.  In section \ref{se:extension}, we showcase some extensions of the bilevel deep learning method. Finally, we present some concluding remarks and discuss possible directions for future work in section \ref{se:conclusion}.

\section{Preliminaries}\label{se:preliminary}
This section presents some preliminary concepts that will be useful throughout this work. To fixed ideas, we first specialize the general model \eqref{general control} to the distributed optimal control of an obstacle problem. Subsequently, we summarize some relevant existing theoretical results and discuss the application of current deep learning methods to this specific problem class.

\subsection{The Distributed Optimal Control of an Obstacle Problem}

Let \(\Omega\subset\mathbb{R}^d\) be a bounded domain with Lipschitz continuous boundary $\partial\Omega$.   We consider the following distributed optimal control problem:
\begin{equation}\label{eq:distribution}
\left\{
\begin{aligned}
\min_{y \in H_0^1(\Omega),\,u \in L^2(\Omega)}&J(y, u)
:=\frac{1}{2}\|y - y_d\|_{L^2(\Omega)}^2 + \frac{\sigma}{2}\|u\|_{L^2(\Omega)}^2 \\
\text{s.t.}\quad 
&y =\underset{y' \in Y_{ad}}{\arg\min}~\mathscr{E}(y^\prime, u)
:=\int_{\Omega}\Bigl(\tfrac12|\nabla y'|^2 - (f+u)\,y'\Bigr)\,\mathrm{d}x,
\quad u \in U_{ad}.
\end{aligned}
\right.
\end{equation}
Above, $y_d \in L^2(\Omega)$ denotes the target state and
$f \in L^2(\Omega)$ is a given external force.
The parameter $\sigma > 0$ serves as a regularization parameter. 
The admissible sets $Y_{ad}$ and $U_{ad}$ for the state $y$ and the control $u$, respectively, are defined as
\begin{equation}\label{eq:def_constraint}
\begin{aligned}
	&Y_{ad}
	=\{\,y\in H_0^1(\Omega)\mid y(x) \ge \psi(x)\text{ a.e.\ in }\Omega\},\\
	&U_{ad}
	=\{\,u\in L^2(\Omega)\mid u_a\le u(x)\le u_b\text{ a.e.\ in }\Omega\},
\end{aligned}
\end{equation}
where $\psi \in H_0^1(\Omega)$ is an obstacle and $u_a$, $u_b$ are assumed to be constants for simplicity. 

Problem \eqref{eq:distribution} seeks a control \( u \in U_{ad}\) such that the corresponding state \( y:=y(u) \), which minimizes the energy functional $\mathscr{E}$ under the total external force \( f + u \), matches the target state \( y_d \in L^2(\Omega)\).

The existence of solution of \eqref{eq:distribution} has been established in e.g., \cite{meyer2013priori,mignot1984optimal}. Note that \eqref{eq:distribution} is nonconvex and its solution is generally  non-unique. Moreover, it has been shown in \cite{mignot1976controle} that $y$ is in general not G{\^a}teaux-differentiable with respect to $u$ unless the biactive set $\{x\in\Omega\mid -\Delta y(x)=u(x)+f(x), y(x)=\psi(x)\}$ has measure zero. 

\subsection{Applications of Existing Deep Learning Approaches to Problem (\ref{eq:distribution}) }\label{se:existing}

As noted in the introduction, some deep learning methodologies developed for optimal control of PDEs can, in principle, be adapted to solve~\eqref{eq:distribution}, provided suitable NN approximations of $u$ and $y$ are employed. We review two representative approaches below.

\subsubsection{The Stationarity Condition-Based Approach} 
One natural approach is to characterize optimal solutions through 
stationarity conditions involving adjoint variables and Lagrange multipliers. To this end, we introduce the adjoint variable  \( p \in H_0^1(\Omega) \), the multipliers \( \xi \in L^2(\Omega) \), \( \lambda \in H^{-1}(\Omega) \), and \( \phi \in H_0^1(\Omega) \). Under some regularity assumptions (cf.~\cite{surowiec2018numerical}), the C-stationarity conditions of~\eqref{eq:distribution} read as
\begin{equation}\label{eq:system}
\left\{
\begin{aligned}
&(\sigma u - p, u-v)_{L^2(\Omega)} \geq 0, \forall v\in U_{ad}, \\
&y - \lambda - \Delta p = y_d, \quad
-\Delta y - u - \xi = f,\\
&(y-\psi, \xi) = 0, \quad y \geq \psi \text{ a.e.}, \quad \xi \geq 0 \text{ a.e.}, \\
&\langle \lambda, p \rangle \leq 0, \quad
p = 0 \text{ a.e. in } \{\xi > 0\},  \\
&\langle \lambda, \phi \rangle = 0 \quad \forall \phi \in H_0^1(\Omega),  \quad
\phi = 0 \text{ a.e. in } \mathbb{A},
\end{aligned}
\right.
\end{equation}
where $\mathbb{A}= \{\,x \in \Omega \mid y(x) = \psi(x)\,\}$ is called the active set.  More discussions on various types of stationarity conditions of  (\ref{eq:distribution}) can be found in \cite{surowiec2018numerical,wachsmuth2014strong,wachsmuth2016towards}.

Following \cite{barry2022physics}, PINNs \cite{raissi2019physics} can be conceptually applied to solve the system of equations \eqref{eq:system} and, consequently, the problem \eqref{eq:distribution}. This approach requires constructing six NNs to approximate the variables \(y, u, p\) and the multipliers \(\lambda, \xi, \phi\).  These NNs are trained simultaneously by minimizing a composite loss function formed from the residuals of \eqref{eq:system}.

In practice, however, the simultaneous training of multiple strongly coupled NNs associated with \eqref{eq:system} is numerically unstable and exhibits poor scalability with the problem size. More importantly, the complementarity conditions imposed on $\lambda$ are difficult to enforce strictly within the PINN framework, which can significantly degrade the numerical accuracy. Due to these limitations, this approach is not practical for solving \eqref{eq:distribution}.

\subsubsection{Objective Combination Approach}


Another commonly used strategy in deep learning methods for optimal control of PDEs  is to collapse the bilevel structure into a single-level optimization problem, see e.g., \cite{mowlavi2023optimal}. This is achieved by forming a weighted combination of the upper- and lower-level objectives. 

Specifically,  let $\hat{y}(x; \theta_y) $ and $\hat{u}(x; \theta_u)$ denote NN approximations of the state $y$ and the control $u$, respectively. This approach trains the NNs by minimizing a single-level objective function constructed by summing the upper- and lower-level objectives with a preset weight  \(w \in \mathbb{R}^+\):
\begin{equation} \label{eq: single reformulation}
	\min_{\theta_y,\theta_u}\quad L_{w}(\theta_y,\theta_u)
	:=J\bigl(\hat y(\cdot;\theta_y),\,\hat u(\cdot;\theta_u)\bigr)+w\mathscr{E}\bigl(\hat y(\cdot;\theta_y), \hat u(\cdot;\theta_u)\bigr).
\end{equation}
Then, a deep learning method can be derived for \eqref{eq:distribution}, as listed in Algorithm~\ref{alg:Single-level Deep Learning Method}.

\begin{algorithm}[H]
	\caption{The Single-Level Deep Learning Method for Problem~\eqref{eq:distribution}} \label{alg:Single-level Deep Learning Method}
	\begin{algorithmic}[1]
		\Require A fixed weight $w$ for problem~\eqref{eq: single reformulation}.
		\State Initialize the NNs $\hat{y}(x; \theta_y)$ and $\hat{u}(x; \theta_u)$ with $\theta^0_y$ and $\theta^0_u$.
		\State Train the NNs $\hat{y}(x; \theta_y)$ and $\hat{u}(x; \theta_u)$ to find the optimal parameters $\theta^*_y$ and $\theta^*_u$ by solving \eqref{eq: single reformulation}.
	\end{algorithmic}
\end{algorithm}

This approach is easy to implement. However, it suffers from several major drawbacks. First, the performance is highly sensitive to the hyperparameter $w$, which requires careful tuning and lacks principled tuning strategies. More importantly, even with substantial tuning of $w$, solving the single-level problem \eqref{eq: single reformulation} generally fails to yield a solution to the original problem \eqref{eq:distribution}. This formulation ignores the constraint that $y = \arg\min_{y' \in Y_{ad}}\mathscr{E}(y^\prime, u)$, thereby breaking the bilevel hierarchical structure of \eqref{eq:distribution}, specifically, the hierarchical dependence of $y$ on $u$. In particular, note that $y$ solves the lower-level problem  if and only if  $\mathscr{E}(y,u)
\le \mathscr{E}(y',u), \forall y' \in Y_{\mathrm{ad}}.$ Hence, the problem \eqref{eq:distribution} can be reformulated as
	{\small
\begin{equation}\label{eq:eq_oc}
		\min_{y \in H_0^1(\Omega),\,u \in U_{ad}} ~ J(y,u)\quad
		\text{s.t.} \quad
		\mathscr{E}(y,u)
		\le \mathscr{E}(y',u),
	~ \forall y' \in {Y}_{\mathrm{ad}}~ \Big(\text{or}~ \mathscr{E}(y,u)\leq \min_{y'\in Y_{ad}}\mathscr{E}(y',u)\Big).
\end{equation}
}

This observation highlights a fundamental flaw in the direct objective combination \eqref{eq: single reformulation}: it merely penalizes the value of the lower-level energy functional rather than the optimality violation. Therefore, \eqref{eq: single reformulation} cannot be viewed as a valid penalized formulation of the original bilevel problem \eqref{eq:distribution}.
Consequently, Algorithm~\ref{alg:Single-level Deep Learning Method} does not reliably enforce optimality of the lower-level state and typically fails to produce high-quality controls. These issues are clearly illustrated in our numerical experiments; see Example 1 in section~\ref{se:experiments}. In summary, Algorithm \ref{alg:Single-level Deep Learning Method} does not reliably produce feasible or meaningful solutions for optimal control of obstacle problems.

\section{The Bilevel Deep Learning Method for Solving Problem~(\ref{eq:distribution})}\label{se: algorithm}

This section presents the proposed bilevel deep learning method for solving the problem~\eqref{eq:distribution}. We first approximate the state $y$ and the control $u$ using NNs that directly embed the problem constraints. The resulting approximation of \eqref{eq:distribution} leads to a bilevel optimization problem formulated in terms of the NN parameters. To efficiently solve this bilevel problem, and thereby train the NNs, we develop the S2-FOBA, a stochastic algorithm specifically designed for this setting. Integrating these components yields our complete method to solve \eqref{eq:distribution}.

\subsection{Neural Network Approximations with Constraints Embedding}\label{subsection:resnets enbedding}
 In this section, we propose NNs  $ \hat{y}(x; \theta_y) $ and $ \hat{u}(x; \theta_u) $ with constraints embedding, which directly incorporates obstacle and possible control constraints into the neural network design, to approximate $y$ and $u$, respectively.

Let $\mathcal{N}(x;\theta_y)$ and \(\mathcal{N}(x; \theta_u)\) be the raw outputs of NNs with smooth activation functions (e.g., Swish, $\tanh$, or Softplus).  To approximate the state \(y\), we define
\begin{equation} \label{eq:state_embedding}
  \hat{y}(x; \theta_y) = m(x)(\mathcal{N}(x;\theta_y))^2 + \psi(x), 
\end{equation}
where \( m \in C^\infty(\overline{\Omega}) \) is chosen such that
   $m(x) = 0$ for $x \in \partial \Omega$, and 
    $m(x) > 0$ for $x \in \Omega.$

When $U_{ad}=L^2(\Omega)$,  we set
\begin{equation} \label{eq:control_embedding}
  \hat{u}(x; \theta_u) = \mathcal{N}(x;\theta_u).
\end{equation}
When $U_{ad}\subset L^2(\Omega)$ as defined in \eqref{eq:def_constraint}, we let
\begin{equation}\label{eq:control-embedding-relu}
\hat{u}(x; \theta_u) = -\operatorname{ReLU}\biggl(u_b - \Bigl[\operatorname{ReLU}\bigl(\mathcal{N}(x; \theta_u) - u_a\bigr) + u_a\Bigr]\biggr) + u_b,
\end{equation}
which strictly enforces the control constraints $u_a \leq \hat u(x;\theta_u) \leq u_b$. 
Similarly, for the state $y$, one can also consider
\begin{equation}\label{eq:state-embedding-relu}
\hat{y}(x; \theta_y) 
= \operatorname{ReLU}\bigl(\mathcal{N}(x; \theta_y)m(x) - \psi(x)\bigr) + \psi(x). 
\end{equation}


After approximating the state \( y \) and the control \( u \) respectively by the NNs given in \eqref{eq:state_embedding}-\eqref{eq:control_embedding} or \eqref{eq:control-embedding-relu}-\eqref{eq:state-embedding-relu}, it follows from \cite{yu2018deep}, the original problem \eqref{eq:distribution} is approximated by the following bilevel optimization problem in terms of $\theta_y$ and $\theta_u$:
\begin{equation*}
\left\{
  \begin{aligned}
	\min_{\theta_u,\theta_y}
	&~
	\frac{1}{2}\int_{\Omega}|\hat{y}\left(x ; \theta_y\right)  - y_d(x)|^2  \,\mathrm{d}x + \frac{\sigma}{2}\int_\Omega|\hat{u}\left(x ; \theta_u\right)|^2\,\mathrm{d}x, \\
	\text {s.t.}&~
	\theta_y
	\in \underset{\theta_y}{\arg\min}\,
	\int_{\Omega}\Bigl(\tfrac12|\nabla \hat{y}\left(x ; \theta_y\right)  |^2 - (f(x)+\hat{u}\left(x ; \theta_u\right) )\,\hat{y}\left(x ; \theta_y\right) \Bigr)\,\mathrm{d}x,
\end{aligned}
\right.
\end{equation*}
where, with a slight abuse of notation, we denote by $\hat{y}\left(x ; \theta_y\right)$ the NN used to approximate $y'$ in \eqref{eq:distribution}. 
Let 
$\mathcal{D}$ be the uniform distribution on $\Omega$ and introduce
$$
\left\{
\begin{aligned}
	L(x; \theta_y, \theta_u) &:= \frac{1}{2}|\hat{y}\left(x ; \theta_y\right)  - y_d(x)|^2  + \frac{\sigma}{2}|\hat{u}\left(x ; \theta_u\right)|^2,\\
	\ell(x; \theta_y, \theta_u) &:=  \tfrac12|\nabla \hat{y}\left(x ; \theta_y\right)  |^2 - (f(x)+\hat{u}\left(x ; \theta_u\right) )\,\hat{y}\left(x ; \theta_y\right).
\end{aligned}
\right.
$$
We thus obtain the following stochastic bilevel problem, which serves as an approximation of \eqref{eq:distribution}.
\begin{equation} \label{eq: stochastic bi}
\left\{
  \begin{aligned}
    \min_{\theta_y,\theta_u}~&j\bigl(\theta_y, \theta_u\bigr)= \mathbb{E}_{x \sim \mathcal{D}} \left[ L(x; \theta_y, \theta_u) \right] \\
    \text{s.t.}~ &\theta_y \in \underset{\theta_y}{\arg\min}\, {e}\bigl(\theta_y, \theta_u\bigr)=\mathbb{E}_{x \sim \mathcal{D}}  \left[ \ell(x; \theta_y, \theta_u)\right].
  \end{aligned}
  \right.
\end{equation}
It is worth noting that the lower-level problem is, in general, nonconvex, and the corresponding solution is not necessarily unique.



\subsection{The S2-FOBA for Problem (\ref{eq: stochastic bi})} \label{subsec:TSP}

In this section, we develop the S2-FOBA to solve the problem \eqref{eq: stochastic bi} and, consequently, to train the neural networks $\hat{u}(x; \theta_u)$ and $\hat{y}(x; \theta_y)$. The S2-FOBA is based on the Moreau envelope-based reformulation of \eqref{eq: stochastic bi} \cite{gao2023moreau}. For clarity, we illustrate the main ideas under the assumption that $j(\theta_y, \theta_u)$ and $e(\theta_y, \theta_u)$ are continuously differentiable with respect to  $\theta_y$ and $\theta_u$.

\subsubsection{Moreau Envelope-Based Reformulation}
The Moreau envelope-based reformulation \cite{gao2023moreau} recasts the bilevel problem \eqref{eq: stochastic bi} as the following single-level constrained optimization problem:
\begin{equation} \label{eq:constrained_reform}
		\min_{\theta_y, \theta_u} \,\, j(\theta_y, \theta_u) \qquad
		\text{s.t.} \quad  e(\theta_y, \theta_u) \leq e_\gamma(\theta_y, \theta_u),
\end{equation}
where $e_{\gamma}(\theta_y,\theta_u)$ is the Moreau envelope of $e(\theta_y,\theta_u)$ defined as:
\begin{equation} \label{moreau-lower}
	e_\gamma(\theta_y,\theta_u) := \min _z  \left\{ e(z,\theta_u)+\frac{1}{2 \gamma}\|z-\theta_y\|^2\right\}
\end{equation}
with the proximal parameter $\gamma>0$. This formulation differs from the classical value function approach \eqref{eq:eq_oc} by replacing the value function, $v(\theta_u):= \min _z \{ e(z,\theta_u)\}$, with $e_\gamma(\theta_y,\theta_u)$. A distinct advantage of this substitution is smoothness: whereas $v(\theta_u)$ is generally nonsmooth, $e_\gamma(\theta_y,\theta_u)$ is continuously differentiable, as discussed below.

As shown in \cite[Theorem A.1]{liu2024moreau}, when $ e(\theta_y, \theta_u) $ is $\rho$-weakly convex with respect to $\theta_y$ (i.e., $e(\theta_y, \theta_u) + \frac{\rho}{2}\|\theta_y\|^2$ with $\rho>0$ is convex with respect to $\theta_y$), the problem \eqref{eq:constrained_reform} is equivalent to a stationarity-based relaxation of \eqref{eq: stochastic bi} where the lower-level solution set is replaced by the set of stationary points:
\begin{equation*}
	\min_{\theta_y, \theta_u} \,\, j(\theta_y, \theta_u) \qquad
	\text{s.t.} \quad  0 \in \nabla_{\theta_y} e(\theta_y, \theta_u).
\end{equation*}
Consequently, the problem \eqref{eq:constrained_reform} is equivalent to \eqref{eq: stochastic bi} if the set of stationary points coincides with the solution set, which holds, for instance, when $ e(\theta_y, \theta_u) $ is convex with respect to $\theta_y$ \cite[Theorem 1]{gao2023moreau}.

A key property of the Moreau envelope $e_\gamma(\theta_y, \theta_u)$ is its smoothness \cite[Lemma A.5]{liu2024moreau}. Specifically, if $ e(\theta_y, \theta_u) $ is $\rho$-weakly convex with respect to $\theta_y$ and $\gamma$ is chosen such that $ \gamma  < 1/\rho$, the minimization problem defining $	e_\gamma(\theta_y,\theta_u)$ in \eqref{moreau-lower} becomes strongly convex. Its unique solution is denoted by:
\begin{equation}\label{proximalLL}
	z^*_{\gamma}(\theta_y, \theta_u) := \arg \min_z \{ e(z, \theta_u) + \frac{1}{2 \gamma} \|z - \theta_y\|^2\}.
\end{equation}
Under these conditions, $e_\gamma(\theta_y, \theta_u)$ is continuously differentiable, with gradient:
\begin{equation} \label{eq:gradient}
	\nabla e_\gamma(\theta_y, \theta_u) = \left( \frac{1}{\gamma} \left(   \theta_y - z^*_{\gamma}(\theta_y, \theta_u) \right), \nabla_{\theta_u} e(z^*_{\gamma}(\theta_y, \theta_u), \theta_u) \right).
\end{equation}

\subsubsection{Penalized Formulation and Stochastic Algorithm}
To train $\theta_y$ and $\theta_u$, we apply a penalty strategy to  \eqref{eq:constrained_reform}, yielding the single-level objective
\begin{equation} \label{penalty}
\psi_{c}(\theta_y, \theta_u) := j(\theta_y, \theta_u) + c \left( e(\theta_y, \theta_u) - e_\gamma(\theta_y, \theta_u) \right),
\end{equation}
where $c > 0$ is a penalty parameter. 
Although the penalized formulation \eqref{penalty} superficially resembles the naive single-level formulation \eqref{eq: single reformulation}, the inclusion of the term $-e_\gamma(\theta_y, \theta_u)$ forms a key difference.
	This term arises directly from penalizing the constraint $e(\theta_y, \theta_u) \le e_\gamma(\theta_y, \theta_u)$ in the Moreau envelope-based reformulation \eqref{eq:constrained_reform}, and is therefore essential.  By contrast, the naive reformulation \eqref{eq: single reformulation} can be interpreted as penalizing the stronger constraint $e(\theta_y, \theta_u) \le \min_{\theta_u, \theta_y} e(\theta_y, \theta_u)$. This forces $(\theta_u, \theta_y)$ toward the joint global minimizer(s) of $e$, which is not equivalent to the lower-level feasibility requirement in the original bilevel problem \eqref{eq: stochastic bi}. Hence, the inclusion of $-e_\gamma(\theta_y, \theta_u)$ in \eqref{penalty} preserves the intended bilevel hierarchical structure, which is lost in the simple combined objective of \eqref{eq: single reformulation}.

We now present the iterative scheme of the S2-FOBA training algorithm, which is inspired by the MEHA method \cite{liu2024moreau}. 
Note that a direct stochastic gradient descent (SGD) on \eqref{penalty} is challenging because computing the gradient $\nabla e_\gamma(\theta_y, \theta_u)$ via \eqref{eq:gradient} requires solving the proximal subproblem \eqref{proximalLL} to find $z^*_{\gamma}(\theta_y, \theta_u)$ at each step. To avoid this costly inner minimization, S2-FOBA introduces an auxiliary variable sequence $\{z^k\}$ that approximates $z^*_{\gamma}$. This auxiliary variable is updated concurrently with the primary parameters $\theta_y$ and $\theta_u$, resulting in a single-loop training process.

At each iteration $k$, we draw an i.i.d.\ mini-batch of $m$ samples, $\mathcal{T}_k=\{x_{i,k}\}_{i=1}^m\subset\Omega$ following the uniform distribution $\mathcal D$. 
The stochastic oracles for the gradients of $j$ and $e$ are constructed by averaging over this mini-batch, yielding $ \frac1m\sum_{i=1}^m  \nabla L(x_{i,k}; \cdot)$ and $\frac1m\sum_{i=1}^m  \nabla \ell (x_{i,k}; \cdot)$, respectively.
Using these stochastic oracles, we first update the auxiliary variable $z^k$ and the NN parameter $\theta_y^k$ by applying one stochastic gradient step to the proximal lower-level problem \eqref{proximalLL} and to the penalized objective in \eqref{penalty}. The gradient of $e_\gamma$ (per expression \eqref{eq:gradient}) is approximated using the newly computed $z^{k+1}$ in place of $z^*_{\gamma}$.
\begin{equation*} 
	\begin{aligned}
			z^{k+1} &= z^k - \eta_k \left(  \frac1m\sum_{i=1}^m  \nabla_{\theta_y} \ell (x_{i,k}; z^k, \theta_u^k)+ \frac{1}{\gamma}(z^k - \theta_y^k) \right),\\
			\theta_y^{k+1} & = \theta_y^k - \frac{\alpha_k}{m} \Big(  \frac{1}{c_k}\sum_{i=1}^m  \nabla_{\theta_y} L (x_{i, k} ; \theta_y^k, \theta_u^k) + \sum_{i=1}^m  \nabla_{\theta_y} \ell (x_{i, k} ; \theta_y^k, \theta_u^k) - \frac{m}{\gamma} (\theta_y^k - z^{k+1})\Big), 
	\end{aligned}
\end{equation*}
where \( \alpha_k, \eta_k > 0\) are the step sizes and $\{c_k\}$ is the sequence of penalty parameters.

Next, we update the NN parameter $\theta_u$ by applying a stochastic gradient step to the penalized objective in \eqref{penalty}.  After sampling another training set $\mathcal{T}_{k + \frac{1}{2}}:=\{x_{i, k+ \frac{1}{2}} \}_{i=1}^m \subset \Omega$, independently of $\mathcal{T}_{k}$, according to the uniform distribution $\mathcal D$, the update is
\begin{equation*} 
	{\scriptsize
	\begin{aligned}
		\theta_u^{k+1} =& \theta_u^k - \frac{\beta_k}{m} \Big(  \frac{1}{c_k}\sum_{i=1}^m  \nabla_{\theta_u} L (x_{i, k+ \frac{1}{2}} ; \theta_y^{k+1}, \theta_u^k) + \sum_{i=1}^m  \nabla_{\theta_u} \ell (x_{i, k+ \frac{1}{2}} ; \theta_y^{k+1}, \theta_u^k) - \sum_{i=1}^m  \nabla_{\theta_u} \ell (x_{i, k+ \frac{1}{2}} ; z^{k+1}, \theta_u^k)\Big),
	\end{aligned}
}
\end{equation*}
where \( \beta_k>0 \) is the step size.

The complete S2-FOBA procedure is detailed in Algorithm~\ref{alg:MEHA}. A significant advantage of this method is its computational efficiency. S2-FOBA is a single-loop, gradient-based algorithm that avoids computationally expensive inner loops. Each iteration only requires the computation of stochastic gradients for a sequential update of $z^k$, $\theta_y^k$, and $\theta_u^k$. All the nice features make S2-FOBA highly scalable and well-suited for large-scale training problems.

\begin{algorithm}[htpb]
	\caption{The S2-FOBA for Problem~\eqref{eq: stochastic bi}}
	\label{alg:MEHA}
	\begin{algorithmic}[1]
		\Require Initial parameters $\theta_y^0,\theta_u^0$, auxiliary $z^0$, proximal parameter $\gamma>0$, step sizes $\{\eta_k,\alpha_k,\beta_k\}$, penalty parameters $ \{c_k\}$, mini-batch size $m$, maximum iterations $K$.
		\For{$k=0$ to $K$}
		\State {Sample a training set $\mathcal{T}_k:=\left\{x_{i, k}\right\}_{i=1}^m \subset \Omega$ from the uniform distribution $\mathcal D$.}
		\State Compute  stochastic oracles of gradients: 
		$$
		{\small
		\begin{aligned}
				h_{j_y}^k := \frac{1}{m}\sum_{i=1}^m  \nabla_{\theta_y} L (x_{i, k}; \theta_y^k, \theta_u^k), \qquad h_{e_y}^k := \frac{1}{m}\sum_{i=1}^m  \nabla_{\theta_y} \ell(x_{i, k}; \theta_y^k, \theta_u^k),\\ 	h_{e_{y},z}^k := \frac1m\sum_{i=1}^m  \nabla_{\theta_y} \ell (x_{i,k}; z^k, \theta_u^k),\hspace{70pt}
		\end{aligned}
	}
		$$
		 and update $z^{k+1}$ and $\theta_y^{k+1}$ as 
		\[
		{\footnotesize
		\begin{aligned}
					z^{k+1} = z^k - \eta_k \left(  h_{e_{y},z}^k + \frac{1}{\gamma}(z^k - \theta_y^k) \right),
						\theta_y^{k+1}  = \theta_y^k - \alpha_k \left(  \frac{1}{c_k} h_{j_y}^k + h_{e_y}^k - \frac{1}{\gamma} (\theta_y^k - z^{k+1})\right).
		\end{aligned}
	}
		\]
		\State  {Sample a training set $\mathcal{T}_{k + \frac{1}{2}}:=\{x_{i, k+ \frac{1}{2}} \}_{i=1}^m \subset \Omega$, independently of $\mathcal{T}_{k}$, from the uniform distribution $\mathcal D$.}
		\State Compute  stochastic oracles of gradients:   
		\[{\small
		\begin{aligned}
			h_{j_u}^k := \frac{1}{m}\sum_{i=1}^m  \nabla_{\theta_u} L (x_{i, k+ \frac{1}{2}}; \theta_y^{k+1}, \theta_u^k), \qquad  h_{e_u}^k := \frac{1}{m}\sum_{i=1}^m  \nabla_{\theta_u} \ell(x_{i, k+ \frac{1}{2}}; \theta_y^{k+1}, \theta_u^k) \\
			h_{e_{u},z}^k := \frac1m\sum_{i=1}^m  \nabla_{\theta_u} \ell (x_{i, k+ \frac{1}{2}}; z^{k+1}, \theta_u^k), \hspace{70pt}
		\end{aligned}}
		\]
		and update  $\theta_u^{k+1}$ as
		\[{\small
		\begin{aligned}
			\theta_u^{k+1}  = \theta_u^k - \beta_k \left(  \frac{1}{c_k}h_{j_u}^k + h_{e_u}^k  -  h_{e_{u},z}^k  \right).
		\end{aligned}}
		\]
		\EndFor
	\end{algorithmic}
\end{algorithm}

%
%

\subsection{Lower-Level Optimality Refinement} \label{Subsection: framework}

The state $y$ is determined by the minimization of an energy functional $\mathscr{E}(y, u)$ for a given control $u$. This implies that the NN parameters should satisfy the lower-level optimality condition, $\theta_y \in \arg\min_{\theta_y}\,{e}(\theta_y, \theta_u)$. This condition is precisely the feasibility constraint in the reformulated problem \eqref{eq:constrained_reform}.

Since Algorithm \ref{alg:MEHA} is based on the approximation \eqref{penalty} of the constrained problem \eqref{eq:constrained_reform}, the parameters $(\theta_y, \theta_u)$ obtained from training are not guaranteed to be exactly feasible for \eqref{eq:constrained_reform}. This means the resulting state parameter $\theta_y$ may not represent a true minimizer of the lower-level objective ${e}(\theta_y, \theta_u)$ for the computed $\theta_u$.

To mitigate this potential infeasibility, we adopt a two-stage strategy inspired by \cite{mowlavi2023optimal}, which introduces a post-training feasibility refinement procedure.

\paragraph{\textbf{Stage 1: Bilevel Training}} 
Use Algorithm~\ref{alg:MEHA} to compute the network parameters $\hat \theta_y$ and $\tilde \theta_u$, yielding an approximate state \( \hat y\) and control \(\tilde u\).

\paragraph{\textbf{Stage 2: Lower-Level Optimality Improvement}}
With the control parameter $\tilde \theta_u$ fixed, initialize the lower-level solver with the state parameter $\hat \theta_y$ obtained in Stage~1 and solve the lower-level problem $\min_{\theta_y}\,{e}(\theta_y, \theta_u)$ to obtain a refined state. This refinement step mitigates the deviations from feasibility induced by the bilevel training.

Combined with constraint-embedding NN approximations, the proposed two-stage approach provides an effective and practical method for solving \eqref{eq:distribution}. Note that Stage~2 consists of a single forward solve of the obstacle problem using the control from Stage~1; it enforces physical consistency of the state without modifying the control. Hence, unless stated otherwise, numerical results are reported using the refined state from Stage~2 and the control from Stage~1.

\subsection{A Bilevel Deep Learning Method for Problem (\ref{eq:distribution})}
Based on the previous discussions, a bilevel deep learning method is proposed for solving the problem ~\eqref{eq:distribution}, as listed in Algorithm \ref{alg:Bilevel Deep Learning Method}. 
\begin{algorithm}[htpb]
	\caption{A Bilevel Deep Learning Method for Problem~\eqref{eq:distribution}} \label{alg:Bilevel Deep Learning Method}
	\begin{algorithmic}[1]
		\Require Parameters for Algorithm~\ref{alg:MEHA}.
		\State Initialize $\hat{y}(x; \theta_y)$ and $\hat{u}(x; \theta_u)$ as described in section \ref{subsection:resnets enbedding} with $\theta^0_y$ and $\theta^0_u$.
		\State \textbf{Stage 1:} Solve the bilevel optimization problem~\eqref{eq: stochastic bi} using Algorithm~\ref{alg:MEHA} to obtain parameters \( (\hat \theta_y, \tilde \theta_u) \) and hence \(\hat{y}(x; \hat \theta_y)\) and \( \tilde{u}(x; \tilde \theta_u)\).
		\State \textbf{Stage 2:} Fix \( \tilde \theta_u \) and compute the state \(\tilde{y}(x; \tilde \theta_y)\) by solving the lower-level problem of \eqref{eq: stochastic bi} via the Adam optimizer with the initial value \( \hat \theta_y \) obtained in \textbf{Stage 1}. 
		\State \textbf{Output:} the control \( \tilde{u}(x; \tilde \theta_u)\) from \textbf{Stage 1} and the state \(\tilde{y}(x; \tilde \theta_y)\)  from \textbf{Stage 2}.

	\end{algorithmic}
\end{algorithm}

\section{Convergence Analysis for Algorithm~\ref{alg:MEHA}}\label{se:convergence}

In this section, we analyze the convergence behavior of the iterates generated by Algorithm~\ref{alg:MEHA}. Since S2-FOBA is a single-loop scheme, the exact proximal point of the lower-level subproblem is not computed at each iteration; instead, it is tracked by the auxiliary variable $z^k$. Accordingly, the analysis proceeds through three ingredients: a merit function combining the penalized objective and the lower-level tracking error, a contraction estimate for the stochastic lower-level update, and a descent estimate for the coupled single-loop iteration. Together, these ingredients show that the inexact lower-level update remains controlled and that the overall scheme still yields descent of the merit function, leading to the main convergence result in Theorem~\ref{thm:convergence} and providing the theoretical justification for the proposed S2-FOBA algorithm.

\subsection{Main Results}
Our convergence analysis is based on the following assumptions.
\begin{itemize}
	\item[{\bf (A1)}] The upper-level objective function $j(\theta_y, \theta_u)$ is bounded from below, i.e., $\underline{j} := \inf_{\theta_y, \theta_u} j(\theta_y, \theta_u) > -\infty$. The lower-level objective function $e(\theta_y, \theta_u)$ is $\rho$-weakly convex, i.e., the function $(\theta_y, \theta_u) \mapsto e(\theta_y, \theta_u) + \frac{\rho}{2} \| (\theta_y, \theta_u)\|^2$ is convex.
	
	\item[{\bf (A2)}] The functions $j(\theta_y,\theta_u)$ and $e(\theta_y,\theta_u)$ are $L_j$-smooth and $L_e$-smooth, respectively; that is, their gradients $\nabla j(\theta_y,\theta_u)$ and $\nabla e(\theta_y,\theta_u)$ exist and are Lipschitz continuous with constants $L_j$ and $L_e$.
\end{itemize}
Note that $L_e$-smoothness of $e(\theta_y,\theta_u)$ implies $\rho$-weak convexity for any $\rho \ge L_e$. Moreover, in the optimal control problems considered in this work, the upper-level objective $j(\theta_y,\theta_u)$ is always nonnegative and thus bounded below by $0$.

To formalize the stochasticity, let $\mathcal{F}_k$ and $\mathcal{F}_{k+\frac{1}{2}}$ denote the $\sigma$-algebras generated by the samples up to steps $k$ and $k+\frac{1}{2}$, respectively:
\[
\mathcal{F}_k = \sigma \{ \mathcal{T}_0, \mathcal{T}_{\frac{1}{2}}, \ldots, \mathcal{T}_k \}, \quad \mathcal{F}_{k+\frac{1}{2}} = \sigma \{ \mathcal{T}_0, \mathcal{T}_{\frac{1}{2}}, \ldots, \mathcal{T}_k, \mathcal{T}_{k+\frac{1}{2}} \}.
\]
We next impose assumptions on the stochastic oracles employed in Algorithm~\ref{alg:MEHA}.
\begin{itemize}
	\item[{\bf (A3)}] For $w\in\{y,u\}$, define
	$
	(\theta_y^{k,w},\theta_u^{k,w},z^{k,w},\mathcal{F}_k^w)
	:=
	\begin{cases}
		(\theta_y^{k},  \theta_u^{k},  z^{k},  \mathcal{F}_k), & w=y,\\[2pt]
		(\theta_y^{k+1},\theta_u^{k},  z^{k+1},\mathcal{F}_{k+\frac12}), & w=u.
	\end{cases}
	$
	Let $h_{j_w}^k$, $h_{e_w}^k$ and $h_{e_{w},z}^k$ be the stochastic gradient
	oracles for $\nabla_{\theta_w} j(\theta_y^{k,w},\theta_u^{k,w})$,
	$\nabla_{\theta_w} e(\theta_y^{k,w},\theta_u^{k,w})$, and
	$\nabla_{\theta_w} e(z^{k,w},\theta_u^{k,w})$, respectively.
	We assume that, conditionally on $\mathcal{F}_k^w$, they are unbiased
	with uniformly bounded variance:
	\[
	\begin{aligned}
		\mathbb{E}\!\left[ h_{j_w}^k \mid \mathcal{F}_k^w \right]
		&= \nabla_{\theta_w} j(\theta_y^{k,w},\theta_u^{k,w}), &
		\mathbb{E}\!\left[ \| h_{j_w}^k - \nabla_{\theta_w} j(\theta_y^{k,w},\theta_u^{k,w}) \|^2
		\mid \mathcal{F}_k^w \right] &\le \sigma_j^2, \\
		\mathbb{E}\!\left[ h_{e_w}^k \mid \mathcal{F}_k^w \right]
		&= \nabla_{\theta_w} e(\theta_y^{k,w},\theta_u^{k,w}), &
		\mathbb{E}\!\left[ \| h_{e_w}^k - \nabla_{\theta_w} e(\theta_y^{k,w},\theta_u^{k,w}) \|^2
		\mid \mathcal{F}_k^w \right] &\le \sigma_e^2, \\
		\mathbb{E}\!\left[ h_{e_{w},z}^k \mid \mathcal{F}_k^w \right]
		&= \nabla_{\theta_w} e(z^{k,w},\theta_u^{k,w}), &
		\mathbb{E}\!\left[ \| h_{e_{w},z}^k - \nabla_{\theta_w} e(z^{k,w},\theta_u^{k,w}) \|^2
		\mid \mathcal{F}_k^w \right] &\le \sigma_e^2.
	\end{aligned}
	\]
	Moreover, the sample set $\mathcal{T}_{k+\frac12}$ is conditionally
	independent given $\mathcal{F}_k$.
\end{itemize}

These assumptions are satisfied in our setting since the samples in $\mathcal{T}_k$ and $\mathcal{T}_{k+\frac{1}{2}}$ are generated independently from the uniform distribution $\mathcal{D}$ over $\Omega$ at each iteration.

For notational simplicity, we define $\theta := (\theta_y,\theta_u)$, $\theta^k := (\theta_y^k,\theta_u^k)$, and $\theta^{k+\frac{1}{2}} := (\theta_y^{k+1},\theta_u^k)$. To analyze the convergence behavior of the iterates generated by Algorithm~\ref{alg:MEHA}, we define the following merit function for the iterates $\theta^k $:
\[
V_k := \phi_{c_k}(\theta^k ) + C_z \bigl\|z^k - z^*_{\gamma}(\theta^k )\bigr\|^2,
\]
where $C_z := {6(1+L_e^2)}/{(\gamma - \gamma^2 \rho)} > 0$ and
$
\phi_{c_k}(\theta) := \frac{1}{c_k} ( j(\theta)-\underline{j}) + \bigl(e(\theta) - e_\gamma(\theta)\bigr).
$

Here, $e(\theta)-e_\gamma(\theta)$ acts as a residual of the lower-level condition in the Moreau-envelope reformulation, while $\|z^k-z^*_{\gamma}(\theta^k)\|^2$ measures the error caused by replacing the exact proximal point with the single stochastic update $z^k$. Thus $V_k$ couples outer progress and inner accuracy in one quantity, which is why it is the natural merit function for the single-loop analysis. Then, we formally state the main convergence results in the following theorem.

\begin{theorem}\label{thm:convergence}
	Suppose that $\gamma \in (0, \frac{1}{2\rho})$ and the step sizes are chosen as
	\[
	\alpha_k = \alpha_0 (k+1)^{-p}, \quad \beta_k = \beta_0 (k+1)^{-p}, \quad \eta_k = \eta_0(k+1)^{-q},
	\]
	with $p \in ((q+1)/2, 1)$ and $q \in (1/2,1)$. Assume further that the penalty parameter $c_k$ is nondecreasing. If $\eta_0 \in (0, \frac{2}{L_{e} + 2/\gamma - \rho} )$, and $\alpha_0, \beta_0$ are sufficiently small, then the sequence of iterates $\{\theta^{k}\}$ generated by Algorithm~\ref{alg:MEHA} satisfies:
	\begin{equation}\label{thmeq1}
			{\small
		\begin{aligned}
			&\mathbb{E} \left[V_{k + 1} \right] +  \frac{\alpha_k}{2} \mathbb{E} \left[ \| \nabla_{\theta_y} \phi_{c_k}(\theta^{k})\|^2 \right] +  \frac{\beta_k}{2} \mathbb{E} \left[ \| \nabla_{\theta_u} \phi_{c_k}(  \theta^{k+\frac{1}{2}} )\|^2 \right] \\
			\le \, & \mathbb{E} \left[V_{k} \right] +  \left(  \frac{\alpha_k^2}{\eta_k} +\frac{\beta_k^2}{\eta_k} + \eta_k^2  \right)  C_\sigma \left( \sigma_{j}^2 + \sigma_{e}^2 \right),
		\end{aligned}
	}
	\end{equation}
	for some $C_\sigma > 0$. Consequently, there exists $M_V > 0$ such that, for any $K > 0$,
	\begin{equation*}
		\mathbb{E} \left[V_{K + 1} \right] + \sum_{k=0}^K \frac{\alpha_k}{2} \mathbb{E} \left[ \| \nabla_{\theta_y} \phi_{c_k}(\theta^{k})\|^2 \right] + \sum_{k=0}^K \frac{\beta_k}{2} \mathbb{E} \left[ \| \nabla_{\theta_u} \phi_{c_k}( \theta^{k+\frac{1}{2}})\|^2 \right] \le V_{0} + M_V.
	\end{equation*}
In particular, $
	\min_{0\le k \le K} \left\{ \mathbb{E} \left[ \|  \nabla_{\theta_y} \phi_{c_k}(\theta^{k})\| + \| \nabla_{\theta_u} \phi_{c_k}( \theta^{k+\frac{1}{2}}) \| \right] \right\} = \mathcal{O}\left( \frac{1}{K^{(1-p)/2}}\right).
	$
\end{theorem}

Theorem~\ref{thm:convergence} shows that the merit function decreases in expectation up to a summable stochastic error, and that the first-order stationarity measure of the penalized objective is controlled. In particular, the effect of replacing the exact proximal point by a single stochastic update is quantified through the second term in $V_k$, rather than through an inner solve carried out to high accuracy.

\subsection{Some Useful Lemmas}
Before proving Theorem~\ref{thm:convergence}, we present the three ingredients of the argument: regularity of the Moreau envelope, control of the stochastic lower-level update, and a descent estimate for the penalized objective.

The three lemmas below address the three ingredients that must be controlled in a single-loop bilevel method. Lemma~\ref{lem:convex-lgamma} shows regularity of the Moreau envelope. Lemma~\ref{lem:contraction} shows that one stochastic update of the auxiliary variable is enough to keep track of the exact proximal state. Lemma~\ref{lem:descent} then transfers these properties to a descent estimate for the penalized upper-level objective. Together, they explain why the lower-level problem need not be solved to high accuracy within each outer iteration.

We  define the update directions as:
\[
d_y^k := \frac{1}{c_k} h_{j_y}^k + h_{e_y}^k - \frac{1}{\gamma} (\theta_y^k - z^{k+1}), \quad
d_u^k := \frac{1}{c_k}h_{j_u}^k + h_{e_u}^k - h_{e_{u},z}^k.
\]

We first recall several useful properties of the Moreau envelope $e_\gamma(\theta)$, established in \cite[Lemmas A.4 and A.8]{liu2024moreau}, which will be repeatedly used in the analysis.

\begin{lemma} \label{lem:convex-lgamma}
	Suppose  that $\gamma \in (0, \frac{1}{2\rho})$. Then, for any $\rho_\gamma \ge 1/ \gamma$, the function $\theta \mapsto e_\gamma(\theta) + \tfrac{\rho_\gamma}{2}\|\theta\|^2$ is convex. Furthermore, the mapping $z^*_{\gamma}(\theta)$ is Lipschitz continuous; i.e., there exists a constant $L_{z_\gamma^*} > 0$ such that for any $\theta, \theta'$,
	\[
	\left\|z_\gamma^*(\theta) - z_\gamma^*(\theta')\right\| \leq L_{z_\gamma^*} \left\| \theta - \theta' \right\|.
	\]
\end{lemma}

We next analyze the stochastic update of the auxiliary lower-level iterate $z^k$. We observe that the objective function $z \mapsto e(z, \theta_u^k) + \frac{1}{2 \gamma} \|z - \theta_y^k\|^2$ in the lower-level proximal problem \eqref{proximalLL} is $(1/\gamma-\rho)$-strongly convex and $(L_{e_y}+1/\gamma)$-smooth. By applying standard convergence analysis techniques for stochastic gradient methods (see, e.g., \cite[section 2.1]{nemirovski2009robust} or \cite[Lemma 3]{chen2021closing}), we can have that the stochastic gradient update for $z^k$ enjoys a contraction property up to a variance term.

\begin{lemma}\label{lem:contraction}
	Suppose that $\gamma \in \bigl(0,\frac{1}{2\rho}\bigr)$ and $\eta_k \in \bigl(0,\frac{2}{L_e+2/\gamma-\rho}\bigr)$. Then the iterates $(\theta_y^k,\theta_u^k,z^k)$ generated by Algorithm~\ref{alg:MEHA} satisfy
	\begin{equation}
		\mathbb{E}\!\left[
		\left.\|z^{k+1}-z^*_{\gamma}(\theta_y^k,\theta_u^k)\|^2
		\,\right|\,\mathcal{F}_k
		\right]
		\le
		\varrho_k^2 \|z^k-z^*_{\gamma}(\theta_y^k,\theta_u^k)\|^2
		+ \eta_k^2 \sigma_e^2,
	\end{equation}
	where $\varrho_k := 1-\eta_k(1/\gamma-\rho)\in(0,1)$.
\end{lemma}

In practical terms, Lemma~\ref{lem:contraction} says that the lower-level iterate may be inexact, but its inexactness is damped from one iteration to the next and remains proportional to the sampling noise. We next establish a descent estimate for the penalized objective along the iterates.

\begin{lemma}\label{lem:descent}
	Suppose that $\gamma \in \bigl(0,\frac{1}{2\rho}\bigr)$ and that the penalty parameter sequence $\{c_k\}$ is nondecreasing. Then the iterates $(\theta_y^k,\theta_u^k,z^k)$ generated by Algorithm~\ref{alg:MEHA} satisfy
	\begin{equation}\label{lem3eq1}
			{\small
		\begin{aligned}
				&\mathbb{E} \left[ \left.	\phi_{c_k}(\theta^{k+1}) \right| \mathcal{F}_k \right]   \le \,	\phi_{c_k}(\theta^k) - \frac{\alpha_k}{2} \|  \nabla_{\theta_y} \phi_{c_k}( \theta^{k} )\|^2 - \frac{\beta_k}{2} \mathbb{E} \left[ \left. \|  \nabla_{\theta_u} \phi_{c_k}( \theta^{k+\frac{1}{2}} ) \|^2 \right| \mathcal{F}_{k} \right] \\
			 &\qquad\quad- \left( \frac{1}{2\alpha_k} - \frac{L_{\phi}}{2} \right) \|  \mathbb{E} \left[ \left.  \theta_y^{k+1} - \theta_y^k \right| \mathcal{F}_k \right]  \|^2 - \left( \frac{1}{2\beta_k} - \frac{L_{\phi}}{2} \right)  \|  \mathbb{E} \left[ \left. \theta_u^{k+1} - \theta_u^k \right| \mathcal{F}_k \right]  \|^2\\
			 &\qquad\quad+ \frac{\alpha_k}{2\gamma^2} \mathbb{E} \left[ \left.   \left\| z^{k+1} - z^*_{\gamma}(\theta^{k}) \right\|^2 \right| \mathcal{F}_k \right] + \frac{\beta_kL_e ^2}{2}  \mathbb{E} \left[ \left.   \left\| z^{k+1} - z^*_{\gamma}(\theta^{k+\frac{1}{2}}) \right\|^2 \right| \mathcal{F}_{k}  \right] \\
			 & \qquad\quad +	(\alpha_k^2+ 2 \beta_k^2 ) L_{\phi} \left( \sigma_{j}^2/c_0^2 +   \sigma_{e}^2 \right),
		\end{aligned}
	}
	\end{equation}
	where $L_{\phi} := L_j/c_0 + L_e + 1/\gamma$.
\end{lemma}

\begin{proof}
	Since $j$ and $e$ are $L_j$- and $L_e$-smooth, respectively, and noting that $c_k \ge c_0$ for all $k$, we invoke the convexity of $ e_\gamma(\theta) + \tfrac{1}{2\gamma}\|\theta\|^2$ (Lemma \ref{lem:convex-lgamma}) to obtain that
	\begin{equation}\label{lem6_eq1}
		\phi_{c_k}( \theta'' ) \le \phi_{c_k}( \theta') + \langle \nabla \phi_{c_k}( \theta'), \theta'' - \theta' \rangle + \frac{L_{\phi}}{2} \left\| \theta'' - \theta' \right \|^2, \quad  \forall \theta', \theta'',
	\end{equation}
	with $L_{\phi} := L_j/c_0 + L_e + 1/\gamma$. Setting $\theta'' = \theta^{k+\frac{1}{2}} $ and $\theta' = \theta^{k} $ in \eqref{lem6_eq1} and taking the conditional expectation with respect to $\mathcal{F}_k$ yields
	\begin{equation}\label{lem6_eq2}
			{\small
	\begin{aligned}
		\mathbb{E} \left[ \left. \phi_{c_k}( \theta^{k+\frac{1}{2}} ) \right| \mathcal{F}_k \right] \le \phi_{c_k}( \theta^{k} ) + \mathbb{E} \left[ \left. \langle \nabla_{\theta_y} \phi_{c_k}( \theta^{k}), \theta_y^{k+1} - \theta_y^{k} \rangle \right| \mathcal{F}_k \right] + \frac{L_{\phi}}{2} \mathbb{E} \left[ \left. \| \theta_y^{k+1} - \theta_y^k \|^2 \right| \mathcal{F}_k \right].
	\end{aligned}
}
	\end{equation}
	Using the update rules $\theta_y^{k+1} - \theta_y^k = -\alpha_k d_y^k$, we expand the inner product term as 
	\begin{equation}\label{lem6_eq3}
			{\small
		\begin{aligned}
			&2\mathbb{E} \left[ \left. \langle \nabla_{\theta_y} \phi_{c_k}( \theta^{k}), \theta_y^{k+1} - \theta_y^{k} \rangle \right| \mathcal{F}_{k} \right] \\
			=\, & - 2\alpha_k \langle \nabla_{\theta_y} \phi_{c_k}( \theta^{k} ), \mathbb{E} \left[ \left. d_y^k \right| \mathcal{F}_{k} \right] \rangle \\
			=\, &\alpha_k \left\| \nabla_{\theta_y} \phi_{c_k}( \theta^{k} ) - \mathbb{E} \left[ \left. d_y^k \right| \mathcal{F}_{k} \right] \right\|^2 - \alpha_k \left\| \nabla_{\theta_y} \phi_{c_k}( \theta^{k} ) \right\|^2 - \alpha_k \left\| \mathbb{E} \left[ \left. d_y^k \right| \mathcal{F}_{k} \right] \right\|^2.
		\end{aligned}
	}
	\end{equation}
	Next, utilizing the expression for $\nabla e_\gamma({x},{y})$ from \eqref{eq:gradient}, the definition of $d_y^k$, and the unbiasedness of the stochastic estimators $h_{j_y}^k$ and $h_{e_y}^k$, we have
	\[
	\mathbb{E} \left[ \left. d_y^k \right| \mathcal{F}_{k} \right] = \nabla_{\theta_y} \phi_{c_k}( \theta^{k} ) + \frac{1}{\gamma} \mathbb{E} \left[ \left. z^{k+1} - z^*_{\gamma}(\theta^{k}) \right| \mathcal{F}_{k} \right].
	\]
Applying Jensen's inequality provides the bound
	\begin{equation}\label{lem6_eq4}
		\left\| \nabla_{\theta_y} \phi_{c_k}( \theta^{k} ) - \mathbb{E} \left[ \left. d_y^k \right| \mathcal{F}_{k} \right] \right\|^2 \le \frac{1}{\gamma^2} \mathbb{E} \left[ \left. \left\| z^{k+1} - z^*_{\gamma}(\theta^{k}) \right\|^2 \right| \mathcal{F}_{k} \right].
	\end{equation}
	Using the variance decomposition $\mathbb{E}[\|X\|^2] = \|\mathbb{E}[X]\|^2 + \text{Var}(X)$ alongside the bounded variance assumptions for $h_{j_y}^k$ and $ h_{e_y}^k$, we have
	\begin{equation}\label{lem6_eq6}
		\mathbb{E} \left[ \left. \| \theta_y^{k+1} - \theta_y^k  \|^2 \right| \mathcal{F}_k \right] \le \| \mathbb{E} \left[ \left. \theta_y^{k+1} - \theta_y^k  \right| \mathcal{F}_k \right] \|^2 + 2\alpha_k^2 \left( \sigma_{j}^2/c_0^2 +   \sigma_{e}^2 \right).
	\end{equation}
	Combining inequalities \eqref{lem6_eq3}, \eqref{lem6_eq4} and \eqref{lem6_eq6} with \eqref{lem6_eq2}, and recalling that $\theta_y^{k+1} - \theta_y^k = -\alpha_k d_y^k$, we obtain 
	\begin{equation}\label{lem6_eq11}
			{\small
		\begin{aligned}
				\mathbb{E} \left[ \left. \phi_{c_k}( \theta^{k+\frac{1}{2}} ) \right| \mathcal{F}_k \right] \le \,  &\phi_{c_k}( \theta^{k} ) - \frac{\alpha_k}{2} \|  \nabla_{\theta_y} \phi_{c_k}( \theta^{k} )\|^2  - \left( \frac{1}{2\alpha_k} - \frac{L_{\phi}}{2} \right) \|  \mathbb{E} \left[ \left.  \theta_y^{k+1} - \theta_y^k \right| \mathcal{F}_k \right]  \|^2 \\
			&+  \frac{\alpha_k}{2\gamma^2}  \mathbb{E} \left[ \left.   \left\| z^{k+1} - z^*_{\gamma}(\theta^{k}) \right\|^2 \right| \mathcal{F}_k \right] +	\alpha_k^2L_{\phi} \left( \sigma_{j}^2/c_0^2 +   \sigma_{e}^2 \right),
		\end{aligned}
	}
	\end{equation}
	Next, setting $\theta'' = \theta^{k + 1} $ and $\theta' =  \theta^{k+\frac{1}{2}} $ in \eqref{lem6_eq1} and taking the conditional expectation with respect to $\mathcal{F}_{k}$, we have
	\begin{equation}\label{lem6_eq7}
			{\small
	\begin{aligned}
		\mathbb{E} \left[ \left. \phi_{c_k}( \theta^{k+1} ) \right| \mathcal{F}_{k} \right] \le \, & \mathbb{E} \left[ \left.  \phi_{c_k}( \theta^{k+\frac{1}{2}} )  \right| \mathcal{F}_{k} \right]  + \mathbb{E} \left[ \left. \langle \nabla_{\theta_u} \phi_{c_k}( \theta^{k+\frac{1}{2}}), \theta_u^{k+1} - \theta_u^{k} \rangle \right| \mathcal{F}_{k} \right] \\ &+ \frac{L_{\phi}}{2} \mathbb{E} \left[ \left. \| \theta_u^{k+1} - \theta_u^k \|^2 \right| \mathcal{F}_{k} \right].
	\end{aligned}
}
    \end{equation}	
Replacing $\theta_u^{k+1} - \theta_u^k = -\beta_k d_u^k$ in the inner product term gives
	\begin{equation}\label{lem6_eq8}
			{\small
		\begin{aligned}
			&2\mathbb{E} \left[ \left. \langle \nabla_{\theta_u} \phi_{c_k}( \theta^{k+\frac{1}{2}}), \theta_u^{k+1} - \theta_u^{k} \rangle \right| \mathcal{F}_{k+\frac{1}{2}} \right] \\
			=\, & \beta_k \left\| \nabla_{\theta_u} \phi_{c_k}( \theta^{k+\frac{1}{2}} ) - \mathbb{E} \left[ \left. d_u^k \right| \mathcal{F}_{k+\frac{1}{2}} \right] \right\|^2 - \beta_k \left\| \nabla_{\theta_u} \phi_{c_k}( \theta^{k+\frac{1}{2}} ) \right\|^2 - \beta_k \left\| \mathbb{E} \left[ \left. d_u^k \right| \mathcal{F}_{k+\frac{1}{2}} \right] \right\|^2.
		\end{aligned}
	}
	\end{equation}
	From the definition of $d_u^k$, the unbiasedness of the stochastic estimators $h_{j_u}^k$, $h_{e_u}^k$, and $h_{e_{u},z}^k$, and the $L_e$-Lipschitz continuity of $\nabla_{\theta_u} e$, we obtain
	\begin{equation}\label{lem6_eq5}
			{\small
		\begin{aligned}
			\left\| \nabla_{\theta_u} \phi_{c_k}( \theta^{k+\frac{1}{2}} ) - \mathbb{E} \left[ \left. d_u^k \right| \mathcal{F}_{k+\frac{1}{2}} \right] \right\|^2
			\le {L_e^2} \| z^{k+1} - z^*_{\gamma}(\theta^{k+\frac{1}{2}})\|^2.
		\end{aligned}
	}
	\end{equation}
	Taking the conditional expectation of \eqref{lem6_eq8} with respect to $\mathcal{F}_{k}$, applying \eqref{lem6_eq5} and Jensen's inequality ($\| \mathbb{E} [d_u^k | \mathcal{F}_{k} ] \|^2 = \| \mathbb{E} [ \mathbb{E} [d_u^k | \mathcal{F}_{k+\frac{1}{2}} ] | \mathcal{F}_{k}  ] \|^2 \le \mathbb{E} [  \| \mathbb{E} [d_u^k | \mathcal{F}_{k+\frac{1}{2}} ] \|^2 |  \mathcal{F}_{k}  ]$ ), we have	
	\begin{equation}\label{lem6_eq9}
			{\small
		\begin{aligned}
			& 2\mathbb{E} \left[ \left. \langle \nabla_{\theta_u} \phi_{c_k}( \theta^{k+\frac{1}{2}}), \theta_u^{k+1} - \theta_u^{k} \rangle \right| \mathcal{F}_{k} \right] 
		 = 	2\mathbb{E} \left[ \left.  \mathbb{E} \left[ \left. \langle \nabla_{\theta_u} \phi_{c_k}( \theta^{k+\frac{1}{2}}), \theta_u^{k+1} - \theta_u^{k} \rangle \right| \mathcal{F}_{k+\frac{1}{2}} \right] \right| \mathcal{F}_{k} \right]  \\
		\le \, &{\beta_k L_e^2}\mathbb{E} \left[ \left. \| z^{k+1} - z^*_{\gamma}(\theta^{k+\frac{1}{2}})\|^2 \right| \mathcal{F}_{k} \right] - \beta_k  \mathbb{E} \left[ \left. \|  \nabla_{\theta_u} \phi_{c_k}( \theta^{k+\frac{1}{2}} ) \|^2 \right| \mathcal{F}_{k} \right] - \beta_k \| \mathbb{E} [d_u^k | \mathcal{F}_{k} ] \|^2.
		\end{aligned}
	}
	\end{equation}
Substituting \eqref{lem6_eq9} into \eqref{lem6_eq7}, using $\theta_u^{k+1} - \theta_u^k = -\beta_k d_u^k$, and again employing the variance decomposition together with the bounded variance assumptions for $h_{j_u}^k$, $ h_{e_u}^k$ and $h_{e_{u},z}^k$, we have
	\begin{equation}\label{lem6_eq10}
		{\small
		\begin{split}
	\mathbb{E} \left[ \left. \phi_{c_k}( \theta^{k+1} ) \right| \mathcal{F}_{k} \right] \le \, & \mathbb{E} \left[ \left.  \phi_{c_k}( \theta^{k+\frac{1}{2}} )  \right| \mathcal{F}_{k} \right]  - \frac{\beta_k}{2} \mathbb{E} \left[ \left. \|  \nabla_{\theta_u} \phi_{c_k}( \theta^{k+\frac{1}{2}} ) \|^2 \right| \mathcal{F}_{k} \right]  \\
& - \left( \frac{1}{2\beta_k} - \frac{L_{\phi}}{2} \right)  \|  \mathbb{E} \left[ \left. \theta_u^{k+1} - \theta_u^k \right| \mathcal{F}_{k}  \right]  \|^2\\
&+ \frac{\beta_kL_e ^2}{2}  \mathbb{E} \left[ \left.   \left\| z^{k+1} - z^*_{\gamma}(\theta^{k+\frac{1}{2}}) \right\|^2 \right| \mathcal{F}_{k}  \right] + 2 \beta_k^2 L_{\phi} \left( \sigma_{j}^2/c_0^2 +   \sigma_{e}^2 \right).
\end{split}
}
	\end{equation}
The desired inequality \eqref{lem3eq1} follows by combining \eqref{lem6_eq11} and \eqref{lem6_eq10}.
\end{proof}

\subsection{Proof of Theorem \ref{thm:convergence}}

With these auxiliary results in place, we are now ready to present the proof of Theorem~\ref{thm:convergence}.

\begin{proof}[Proof of Theorem \ref{thm:convergence}]
	Since $c_{k+1} \ge c_k$, we have that $( j( \theta^{k+1}) - \underline{j} )/c_{k+1} \le ( j( \theta^{k+1}) - \underline{j} ) /c_{k}$. Combining this property with \eqref{lem3eq1} established in Lemma \ref{lem:descent} yields	
	\begin{equation}\label{lem9_eq4}
		{\small
		\begin{aligned}
			&\mathbb{E} \left[ \left.	V_{k+1} \right| \mathcal{F}_k \right]  - V_k 
			\le \,   - \frac{\alpha_k}{2} \|  \nabla_{\theta_y} \phi_{c_k}( \theta^{k} )\|^2 - \frac{\beta_k}{2} \mathbb{E} \left[ \left. \|  \nabla_{\theta_u} \phi_{c_k}( \theta^{k+\frac{1}{2}} ) \|^2 \right| \mathcal{F}_{k} \right] \\  
			&\qquad- \left( \frac{1}{2\alpha_k} - \frac{L_{\phi}}{2} \right) \|  \mathbb{E} \left[ \left.  \theta_y^{k+1} - \theta_y^k \right| \mathcal{F}_k \right]  \|^2
			- \left( \frac{1}{2\beta_k} - \frac{L_{\phi}}{2} \right)  \|  \mathbb{E} \left[ \left. \theta_u^{k+1} - \theta_u^k \right| \mathcal{F}_k \right]  \|^2\\
			 &\qquad+C_z \mathbb{E} \left[ \left.	 \left\| z^{k+1} - z^*_{\gamma}(\theta^{k+1}) \right\|^2 \right| \mathcal{F}_k \right] + \frac{\alpha_k}{2\gamma^2} \mathbb{E} \left[ \left.   \left\| z^{k+1} - z^*_{\gamma}(\theta^{k}) \right\|^2 \right| \mathcal{F}_k \right]  \\ 
			 &\qquad + \frac{\beta_kL_e ^2}{2}  \mathbb{E} \left[ \left.   \left\| z^{k+1} - z^*_{\gamma}(\theta^{k+\frac{1}{2}}) \right\|^2 \right| \mathcal{F}_{k}  \right]  - C_z  \left\| z^{k} - z^*_{\gamma}(\theta^{k}) \right\|^2 + 	(\alpha_k^2+ 2\beta_k^2 ) L_{\phi} \left( \sigma_{j}^2/c_0^2 +   \sigma_{e}^2 \right).
		\end{aligned}
	}
	\end{equation}
	 By Lemma~\ref{lem:convex-lgamma}, $z^*_{\gamma}(\cdot)$ is $L_{z_\gamma^*}$-Lipschitz continuous. Applying the inequality $\|a+b\|^2 \le (1+\epsilon)\|a\|^2 + (1+\epsilon^{-1})\|b\|^2$ for any $\epsilon_k > 0$ and $c_z > 0$, we have
	\begin{equation*}
		{\small
		\begin{split}
			& \left\| z^{k+1} - z^*_{\gamma}(\theta^{k+1}) \right\|^2 + c_z \left\| z^{k+1} - z^*_{\gamma}(\theta^{k}) \right\|^2 + c_z \left\| z^{k+1} - z^*_{\gamma}(\theta^{k+\frac{1}{2}}) \right\|^2 - \left\| z^{k} - z^*_{\gamma}(\theta^{k}) \right\|^2 \\
		\le \, & (1+\epsilon_{k} )\left\| z^{k+1} - z^*_{\gamma}(\theta^{k}) \right\|^2 + \left(1+ \frac{1}{\epsilon_{ k} }\right) \|  z^*_{\gamma}(\theta^{k+1}) - z^*_{\gamma}(\theta^{k}) \|^2 + c_z \left\| z^{k+1} - z^*_{\gamma}(\theta^{k}) \right\|^2 \\
		& + 2c_z \left\| z^{k+1} - z^*_{\gamma}(\theta^{k}) \right\|^2 + 2c_z \|  z^*_{\gamma}(\theta^{k+\frac{1}{2}}) - z^*_{\gamma}(\theta^{k}) \|^2  - \left\| z^{k} - z^*_{\gamma}(\theta^{k}) \right\|^2 \\
			\le \, & (1 + \epsilon_{k} + 3c_z )\left\| z^{k+1} - z^*_{\gamma}(\theta^{k}) \right\|^2 + \left(1 + 2c_z + \frac{1}{ \epsilon_{k} } \right)L_{z_\gamma^*}^2 \| \theta^{k+1} - \theta^{k} \|^2 - \left\| z^{k} - z^*_{\gamma}(\theta^{k}) \right\|^2.
		\end{split}
	}
	\end{equation*}
	Taking the conditional expectation and applying the contraction result of Lemma~\ref{lem:contraction} to $\mathbb{E}[\| z^{k+1} - z^*_{\gamma}(\theta^{k}) \|^2 \mid \mathcal{F}_k]$ (valid for $\eta_k \in (0, \frac{2}{L_e+2/\gamma-\rho})$), we obtain
	\begin{equation*}
		{\small
		\begin{aligned}
			& \mathbb{E} \left[ \left. \left\| z^{k+1} - z^*_{\gamma}(\theta^{k+1}) \right\|^2 \right| \mathcal{F}_k \right]+ c_z \mathbb{E} \left[ \left. \left\| z^{k+1} - z^*_{\gamma}(\theta^{k}) \right\|^2 \right| \mathcal{F}_k \right] \\
			&\qquad\qquad\qquad\qquad+ c_z \mathbb{E} \left[ \left.   \left\| z^{k+1} - z^*_{\gamma}(\theta^{k+\frac{1}{2}}) \right\|^2 \right| \mathcal{F}_{k}  \right]  - \left\| z^{k} - z^*_{\gamma}(\theta^{k}) \right\|^2 \\
			\le \, & (1+\epsilon_k+ 3c_z ) \varrho_k^2 \| z^k -z^*_{\gamma}(\theta^k) \|^2 - \left\| z^{k} - z^*_{\gamma}(\theta^{k}) \right\|^2 \\
			& + \left(1+ 2c_z +  \frac{1}{\epsilon_k}\right)L_{z_\gamma^*}^2\mathbb{E} \left[ \left. \| \theta^{k+1} - \theta^{k} \|^2 \right| \mathcal{F}_k \right] + (1+\epsilon_k+ 3c_z )\eta_k^2 \sigma_{e}^2,
		\end{aligned}
	}
	\end{equation*}
	where $\varrho_k := 1-\eta_k\left(1/{\gamma}-\rho\right) < 1$. 
	Setting $\epsilon_k = \eta_k(1/{\gamma}-\rho)/2$ and restricting $c_z \le \eta_k(1/{\gamma}-\rho)/6$ simplifies the inequality to
	\begin{equation}\label{lem9_eq5}
		{\small
		\begin{split}
			& \mathbb{E} \left[ \left. \left\| z^{k+1} - z^*_{\gamma}(\theta^{k+1}) \right\|^2 \right| \mathcal{F}_k \right]+ c_z \mathbb{E} \left[ \left. \left\| z^{k+1} - z^*_{\gamma}(\theta^{k}) \right\|^2 \right| \mathcal{F}_k \right] \\
			&\qquad\qquad\qquad\qquad\qquad+ c_z \mathbb{E} \left[ \left.   \left\| z^{k+1} - z^*_{\gamma}(\theta^{k+\frac{1}{2}}) \right\|^2 \right| \mathcal{F}_{k}  \right]  - \left\| z^{k} - z^*_{\gamma}(\theta^{k}) \right\|^2 \\
			\le \, & -\eta_k\left(1/{\gamma}-\rho\right) \left\| z^{k} - z^*_{\gamma}(\theta^{k}) \right\|^2 + \left(1 +\eta_k(1/{\gamma}-\rho) + \frac{2}{\eta_k(1/{\gamma}-\rho) }\right)L_{z_\gamma^*}^2\mathbb{E} \left[ \left. \| \theta^{k+1} - \theta^{k} \|^2 \right| \mathcal{F}_k \right] \\
			& + \left(1+\eta_k(1/{\gamma}-\rho) \right)\eta_k^2 \sigma_{e}^2.
		\end{split}
	}
	\end{equation}
	If the step sizes satisfy $\alpha_k \le \frac{(\gamma - \gamma^2 \rho)C_z}{6} \eta_k$ and $\beta_k \le \frac{(1/\gamma -\rho)C_z}{6L_e^2} \eta_k$, then $\frac{1}{C_z}\left( \frac{\alpha_k}{2\gamma^2} + \frac{\beta_kL_e ^2}{2} \right) \le \eta_k(1/{\gamma}-\rho) /6$. Setting $c_z =\frac{1}{C_z}\left( \frac{\alpha_k}{2\gamma^2} + \frac{\beta_kL_e ^2}{2} \right)$ and combining \eqref{lem9_eq5} with the variance bounds \eqref{lem6_eq6} and the descent inequality \eqref{lem9_eq4} yields
	\begin{equation}\label{lem9_eq6}
{\small
		\begin{split}
			&\mathbb{E} \left[ V_{k+1} \mid \mathcal{F}_k \right] - V_k\\
			\le \, & - \frac{\alpha_k}{2} \| \nabla_{\theta_y} \phi_{c_k}( \theta^{k} )\|^2 - \frac{\beta_k}{2} \mathbb{E} \left[ \left. \|  \nabla_{\theta_u} \phi_{c_k}( \theta^{k+\frac{1}{2}} ) \|^2 \right| \mathcal{F}_{k} \right] \\ 
			& - \left( \frac{1}{2\alpha_k} - \frac{L_{\phi}}{2} - C_z(1 +\eta_k(1/{\gamma}-\rho))L_{z_\gamma^*}^2 - \frac{2C_zL_{z_\gamma^*}^2}{\eta_k(1/{\gamma}-\rho)} \right) \| \mathbb{E} \left[ \theta_y^{k+1} - \theta_y^k \mid \mathcal{F}_k \right] \|^2 \\
			& - \left( \frac{1}{2\beta_k} - \frac{L_{\phi}}{2} - C_z(1 +\eta_k(1/{\gamma}-\rho) )L_{z_\gamma^*}^2 - \frac{2C_zL_{z_\gamma^*}^2}{\eta_k(1/{\gamma}-\rho)} \right) \| \mathbb{E} \left[ \theta_u^{k+1} - \theta_u^k \mid \mathcal{F}_k \right] \|^2 \\
			& -\eta_k\left(1/{\gamma}-\rho\right) C_z \left\| z^{k} - z^*_{\gamma}(\theta^{k}) \right\|^2 + \left(1+\eta_k(1/{\gamma}-\rho) \right)\eta_k^2 C_z\sigma_{\ell_{y}}^2 \\
			& + (\alpha_k^2+ 2 \beta_k^2 ) \left( L_{\phi} + C_z(1 +\eta_k(1/{\gamma}-\rho) )L_{z_\gamma^*}^2 + \frac{2C_zL_{z_\gamma^*}^2}{\eta_k(1/{\gamma}-\rho)} \right)\left( \sigma_{j}^2/c_0^2 + \sigma_{e}^2 \right).
		\end{split}
	}
	\end{equation}
	
	The step size strategy $\alpha_k = \alpha_0 (k+1)^{-p}$, $\beta_k = \beta_0 (k+1)^{-p}$ and $\eta_k = \eta_0(k+1)^{-q}$, with $p \in ((q+1)/2, 1)$ and $q \in (1/2,1)$, ensures the required bounds hold. Specifically, if $\eta_0 \in (0, \frac{2}{L_{e} + 2/\gamma - \rho} )$, then $\eta_k$ remains within this interval for all $k$. Furthermore, by choosing $\alpha_0$ and $\beta_0$ sufficiently small relative to $\eta_0$, we ensure that the terms $\frac{1}{2\alpha_k}$ and $\frac{1}{2\beta_k}$ dominate the negative components in the coefficients of the squared differences  $\| \mathbb{E} \left[ \theta_y^{k+1} - \theta_y^k \mid \mathcal{F}_k \right] \|^2 $ and $\| \mathbb{E} \left[ \theta_u^{k+1} - \theta_u^k \mid \mathcal{F}_k \right] \|^2 $, guaranteeing their non-negativity for all $k$. Taking the total expectation of \eqref{lem9_eq6}  implies \eqref{thmeq1}.
	
	Summing the total expectation of \eqref{lem9_eq6} from $k=0$ to $K$, we utilize the fact that the series $\sum \alpha_k^2, \sum \beta_k^2, \sum \eta_k^2, \sum \alpha_k^2/\eta_k$, and $\sum \beta_k^2/\eta_k$ are all convergent (finite). Thus, there exists a constant $M_V > 0$ such that for any $K > 0$,
	\begin{equation*}
		\mathbb{E} \left[V_{K + 1} \right] + \sum_{k=0}^K \frac{\alpha_k}{2} \mathbb{E} \left[ \| \nabla_{\theta_y} \phi_{c_k}( \theta^{k} )\|^2 \right] + \sum_{k=0}^K \frac{\beta_k}{2} \mathbb{E} \left[ \| \nabla_{\theta_u} \phi_{c_k}( \theta^{k+\frac{1}{2}} )\|^2 \right] \le V_{0} + M_V.
	\end{equation*}
	Taking the limit as $K \rightarrow \infty$ implies
	\[
	\sum_{k=0}^\infty \alpha_k \mathbb{E} \left[ \| \nabla_{\theta_y} \phi_{c_k}( \theta^{k} )\|^2 \right]  + \sum_{k=0}^\infty \beta_k \mathbb{E} \left[ \| \nabla_{\theta_u} \phi_{c_k}( \theta^{k+\frac{1}{2}} )\|^2 \right] < \infty.
	\]
Finally, given the forms of $\alpha_k$ and $\beta_k$, the minimum gradient norm converges at the rate $\min_{0\le k \le K} \left\{ \mathbb{E} \left[ \|\nabla_{\theta_y} \phi_{c_k}( \theta^{k} )\|^2 +   \| \nabla_{\theta_u} \phi_{c_k}( \theta^{k+\frac{1}{2}} )\|^2 \right] \right\} = \mathcal{O}\left( \frac{1}{K^{(1-p)}}\right)$.
\end{proof}


\section{Numerical Results} \label{se:experiments}
In this section, we present some numerical results of Algorithm~\ref{alg:Bilevel Deep Learning Method} for solving various  problems in the form of \eqref{eq:distribution} to validate its effectiveness and efficiency. All codes were written in Python and PyTorch. The numerical experiments were
conducted on  a server
provisioned with dual Intel Xeon Gold 5218R CPUs (a total of 40 cores/80 threads, with 2.1-4.0
GHz) and an NVIDIA H100 GPU.

The state \(y\) and the control \(u\) are approximated by ResNets with three residual blocks and a linear output layer, augmented with the constraint-embedding layers from section~\ref{subsection:resnets enbedding}. Each residual block contains two fully connected layers with 16 neurons each. All networks employ the Swish activation function and Xavier uniform initialization with zero biases. Unless stated otherwise, the learning rates are set to \(\alpha=\beta=\eta=10^{-3}\) and decayed by a factor of 0.8 every 1{,}000 epochs. The mini-batch size is set as \(m=512\).

\medskip

\noindent \textbf{Example 1 (Bi-active set).} Let \(\Omega=(0,1)^2\), \(\sigma=1\), \(Y_{ad}=\{y\in H_0^1(\Omega)\mid y\ge0\text{ a.e.\ in }\Omega\}\),
and

\[
\begin{aligned}
y^{\dagger}(x_1,x_2)&=
\begin{cases}
  160\,C\bigl(x_1^3 - x_1^2 + \tfrac14 x_1\bigr)
     \bigl(x_2^3 - x_2^2 + \tfrac14 x_2\bigr)
  &0<x_1,x_2<0.5,\\[6pt]
  0&\text{else}.
\end{cases} \\
\xi^{\dagger}(x_1,x_2)&=\max\bigl(0,\,-2|x_1-0.8|-2|x_1x_2-0.3|+0.5\bigr).\\
  f&=-\Delta y^{\dagger}-y^{\dagger}-\xi^{\dagger},
\qquad
y_d=y^{\dagger}+\xi^{\dagger}-\Delta y^{\dagger}.
\end{aligned}
\]

This example features a nontrivial bi-active set (where \( y = \xi = 0 \)), which makes accurate detection of the active set challenging for numerical methods.  Compared to  \cite{hintermuller2011smooth},
we add a scaling constant \(C=20\) in \(y^\dagger\) for normalization. The analytical solution is \(u^*=y^*=y^\dagger\), which is used to evaluate the numerical accuracy. For the implementation of Algorithm~\ref{alg:Bilevel Deep Learning Method}, we use the NNs $\hat{y}$ and $\hat{u}$ given by \eqref{eq:state_embedding}-\eqref{eq:control_embedding}  and set
$
T = 20,000,
\gamma = 20,$ and $c_k = 5 k^{0.3}.
$

We first set  \(U_{ad}=L^2(\Omega)\) and the results of Algorithm \ref{alg:Bilevel Deep Learning Method} are shown in Figures \ref{fig:training_ex1}-\ref{fig:meha_2}. We show the training trajectories and then compare the computed state \(y\) and control \(u\) against their exact counterparts and plot the pointwise errors.

\begin{figure}[htpb]
	\centering
	\subfloat[Stage 1: Bilevel Training Process]{%
		\includegraphics[width=0.5\textwidth]{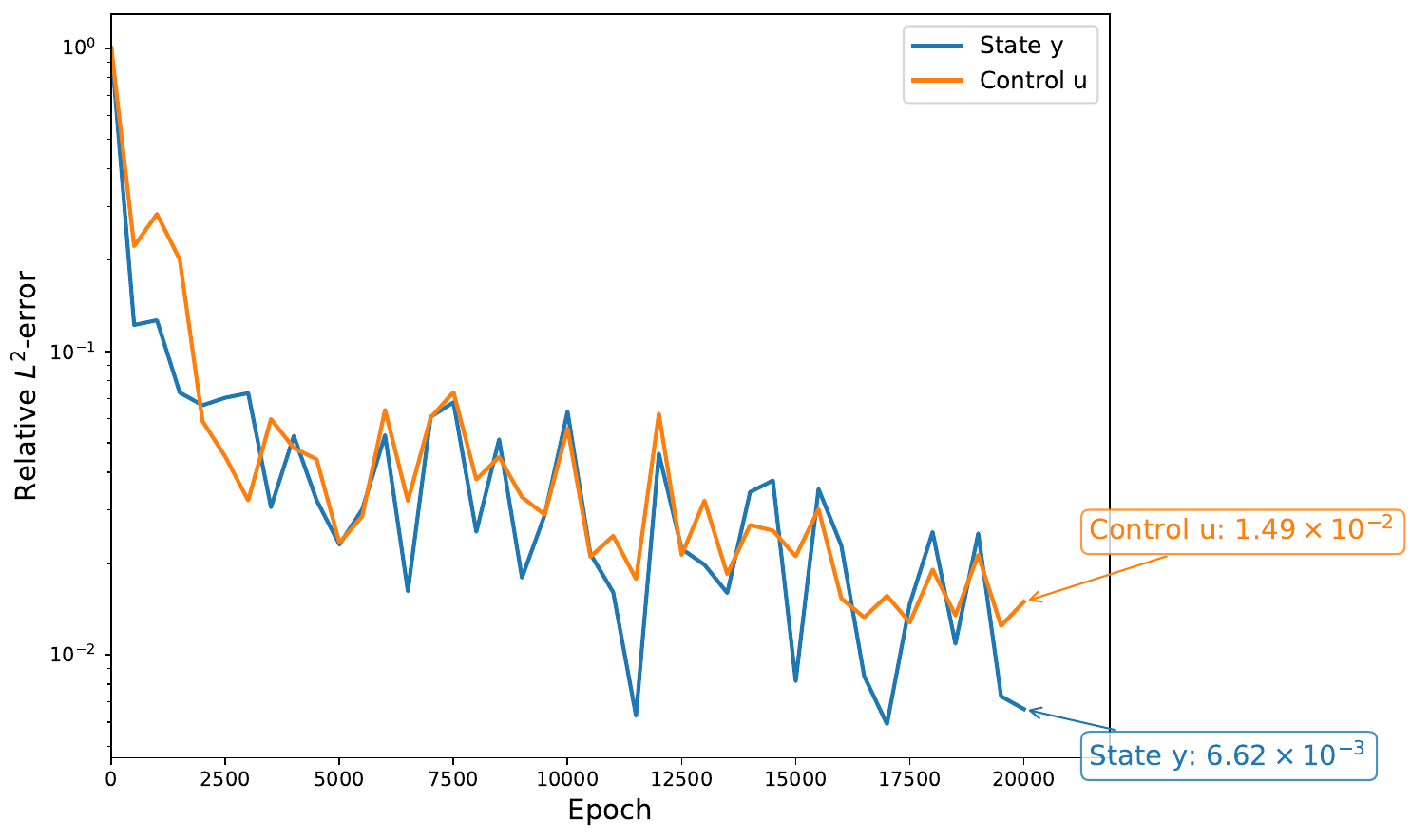}%
		
	}
	\hfill
	\subfloat[Stage 2:  Lower-Level Optimality Improvement Procedure]{%
		\includegraphics[width=0.5\textwidth]{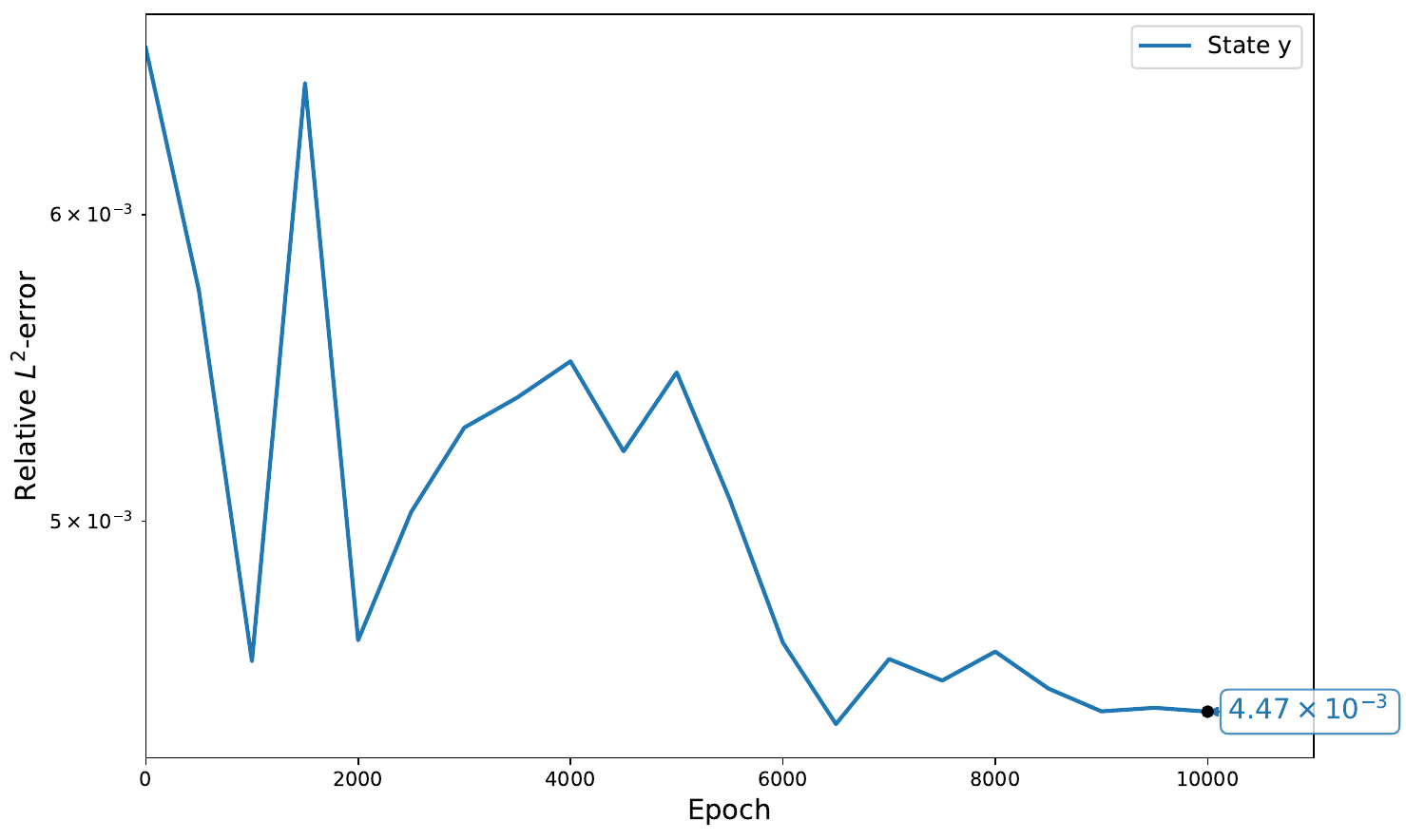}%
		
	}
	
	\caption{\small \em Training trajectories of Algorithm~\ref{alg:Bilevel Deep Learning Method} for {Example 1}}
	\label{fig:training_ex1}
\end{figure}

\begin{figure}[htpb]
  \centering
  
  \subfloat[Reference state $y^*$]{%
    \includegraphics[width=0.30\textwidth]{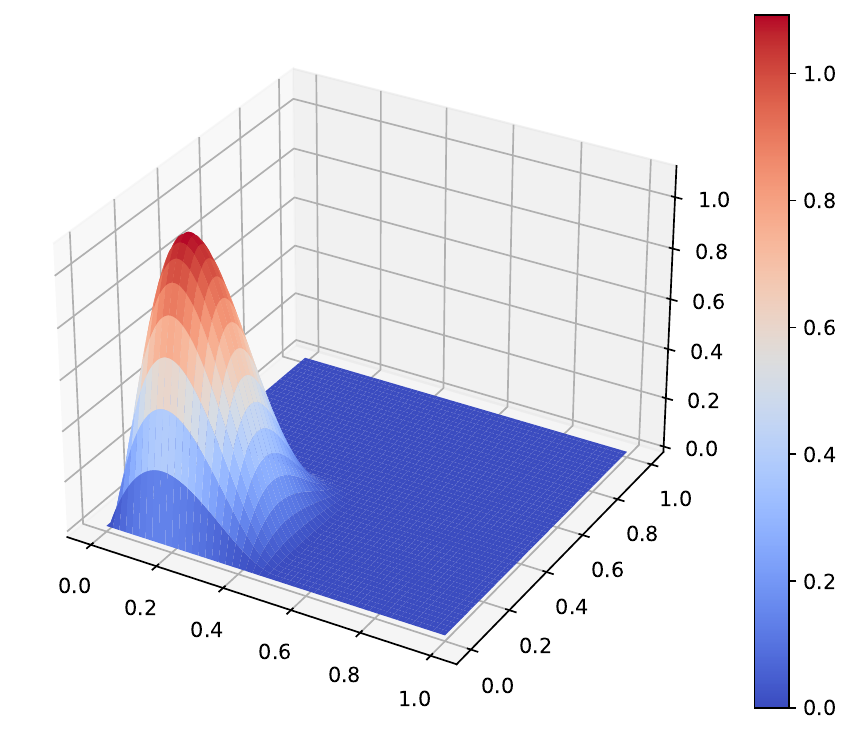}%

  }
  \hfill
  \subfloat[Computed state $\hat y$]{%
    \includegraphics[width=0.30\textwidth]{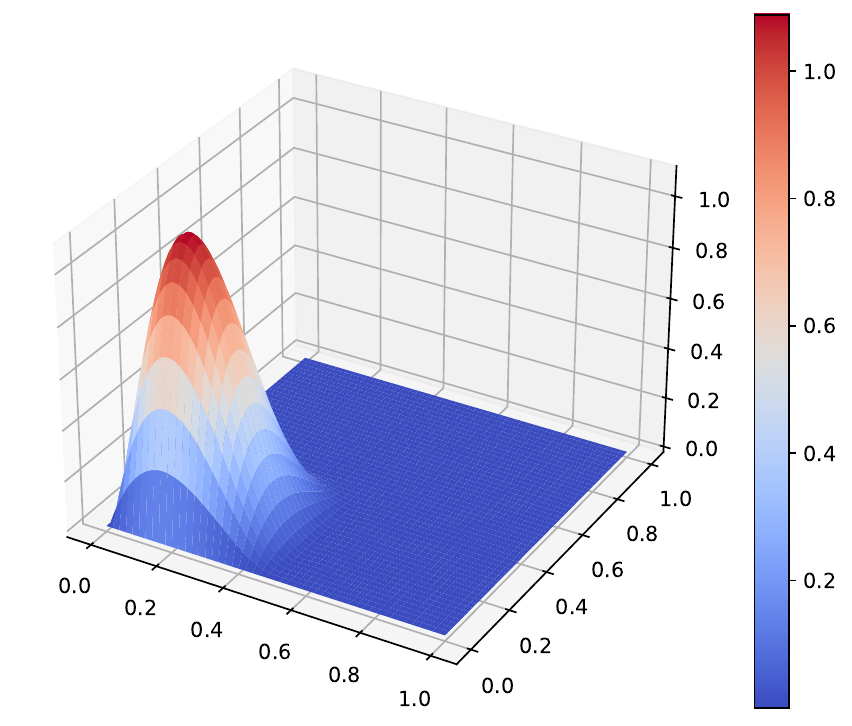}%

  }
  \hfill
  \subfloat[Pointwise error of state]{%
    \includegraphics[width=0.30\textwidth]{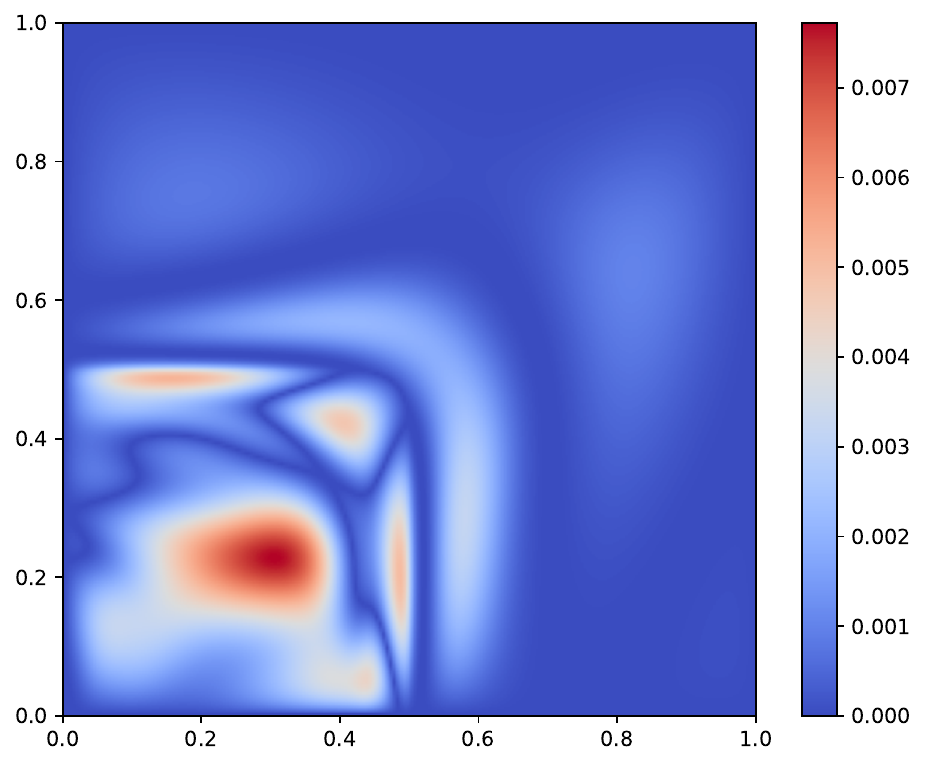}%

  }

  \vspace{1em}

  \subfloat[Reference control $u^*$]{%
    \includegraphics[width=0.30\textwidth]{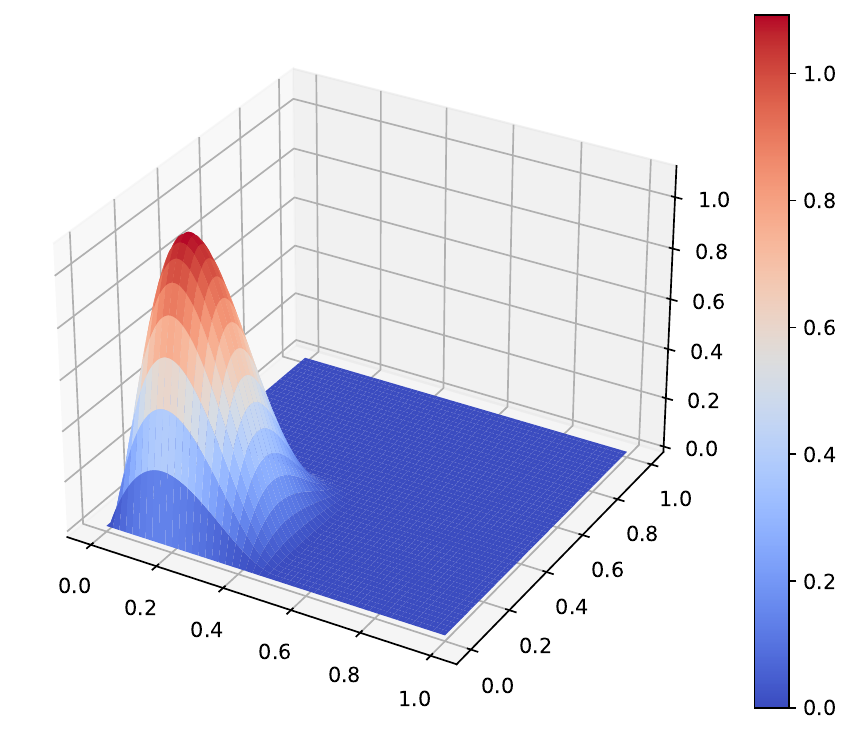}%

  }
  \hfill
  \subfloat[Computed control $\hat u$]{%
    \includegraphics[width=0.30\textwidth]{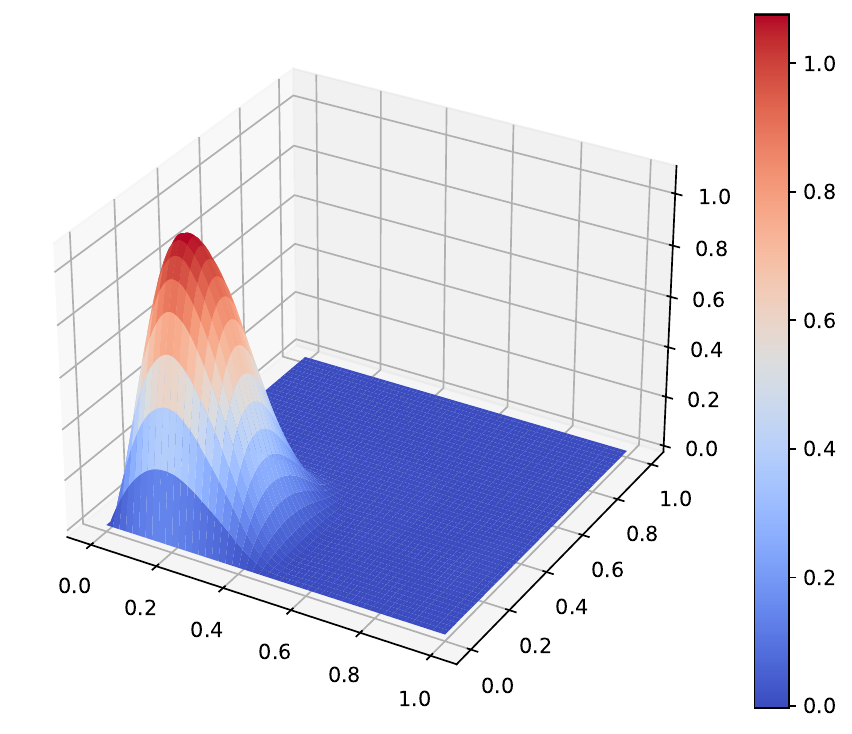}%

  }
  \hfill
  \subfloat[Pointwise error of control]{%
    \includegraphics[width=0.30\textwidth]{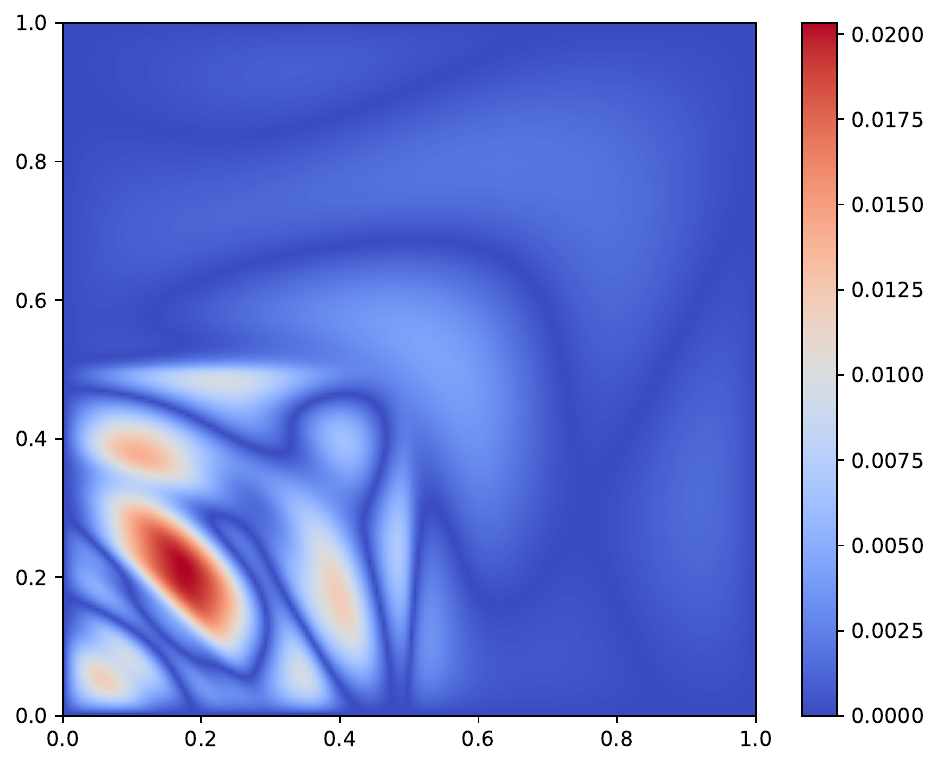}%

  }

  \caption{\small \em Numerical results of Algorithm~\ref{alg:Bilevel Deep Learning Method} for {Example 1}}
  \label{fig:meha_2}
\end{figure}

Furthermore, the control constraint $u\in U_{ad}$ often presents challenges and typically requires additional techniques, such as active set detection. Within Algorithm \ref{alg:Bilevel Deep Learning Method}, this issue can be effectively addressed by embedding $u\in U_{ad}$ via \eqref{eq:control-embedding-relu}. To demonstrate the effect of this modification, we set \( U_{ad} = \{ u \in L^2(\Omega) \mid 0 \leq u(x) \leq 0.7 \text{ a.e.\ in } \Omega \} \). The corresponding results are shown in Figure~\ref{fig:meha_2_constraint}, which validate the effectiveness of Algorithm \ref{alg:Bilevel Deep Learning Method} for handling the control constraints. 

\begin{figure}[h!]
	\centering
	\subfloat[Computed state \(\hat y\)]{%
		\includegraphics[width=0.45\textwidth]{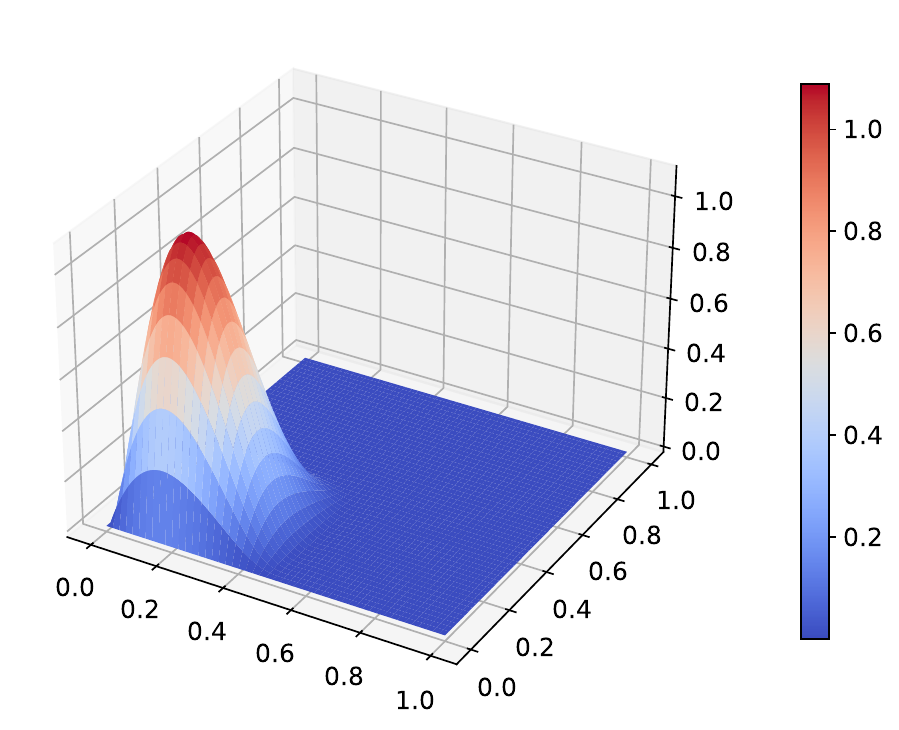}%
	}\hfill
	\subfloat[Computed control \(\hat u\) with constraint]{%
		\includegraphics[width=0.45\textwidth]{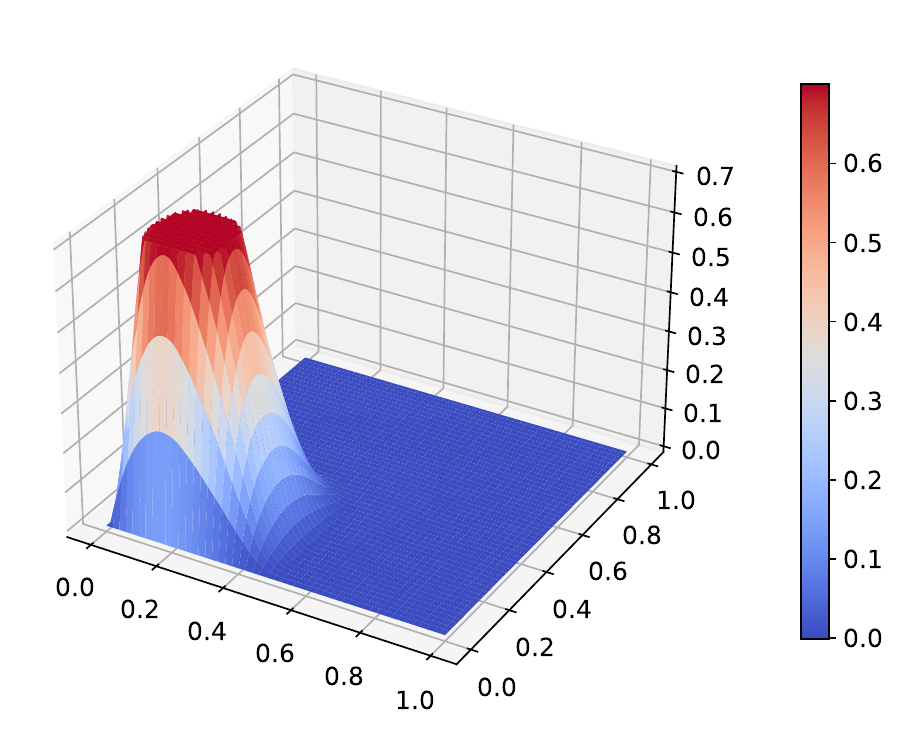}%
	}
	\caption{\small \em Numerical results of Algorithm~\ref{alg:Bilevel Deep Learning Method} for Example 1 with control constraints}
	\label{fig:meha_2_constraint}
\end{figure}
 
Then, we compare Algorithm \ref{alg:Bilevel Deep Learning Method} with Algorithm \ref{alg:Single-level Deep Learning Method}. For this purpose, we test Algorithm~\ref{alg:Single-level Deep Learning Method} using various weights \( w \) and the computed relative errors are displayed in Figure \ref{fig:example2 vanilla} (a) and some comparison results are presented in Table~\ref{tab:rel_l2_errors_example1}.  Moreover, to assess the quality of the obtained solutions, we employ the primal–dual active set strategy  \cite{ito2003semi} to solve the minimization of energy functional $\mathscr{E}$ in terms of $y$, with the control $u$ computed by Algorithm~\ref{alg:Single-level Deep Learning Method} and Algorithm \ref{alg:Bilevel Deep Learning Method}. This allows us to restore the lower-level minimization constraint and compute the corresponding recovered objective functional \( J \). Two representative results are shown in Figure~\ref{fig:example2 vanilla} (b) and (c). From these results, we conclude that, with a fixed weight $w$, Algorithm~\ref{alg:Single-level Deep Learning Method} fails to accurately recover the lower-level energy and the upper-level objective,  and the computed control and state lack sufficient accuracy.

\begin{figure}[htpb]
	\centering
	\subfloat[Numerical errors of Algorithm \ref{alg:Single-level Deep Learning Method} with different $w$.]{%
		\includegraphics[width=0.52\textwidth]{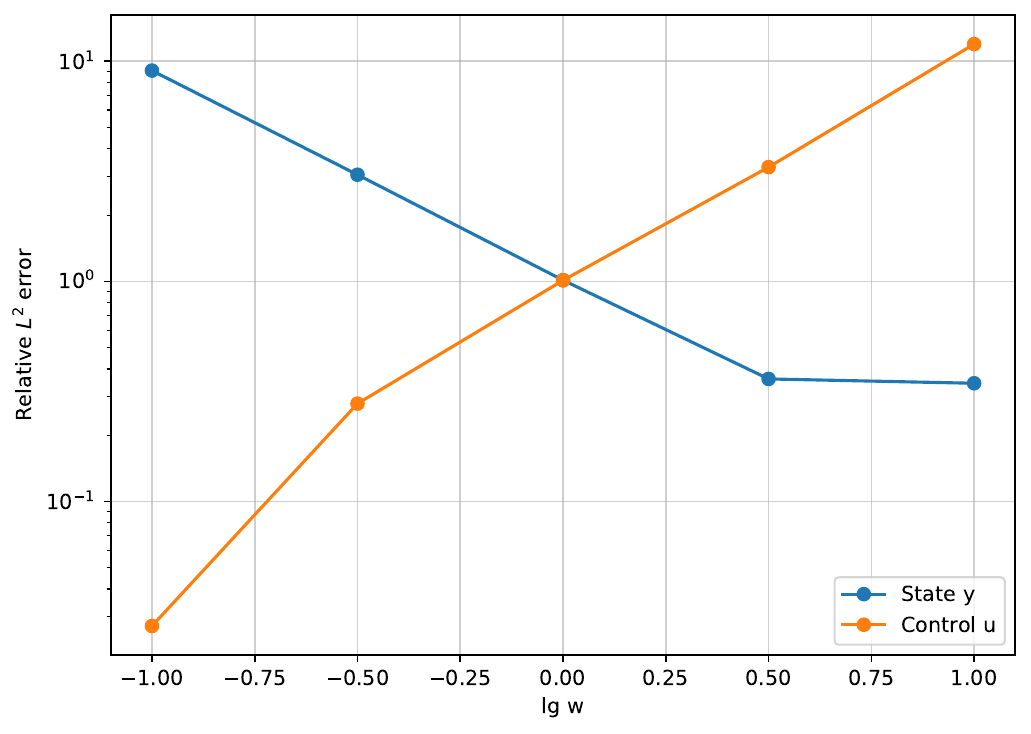}%
	}

	\vspace{0.8em}

	\subfloat[Comparisons of the computed lower-level energy]{%
		\includegraphics[width=0.48\textwidth]{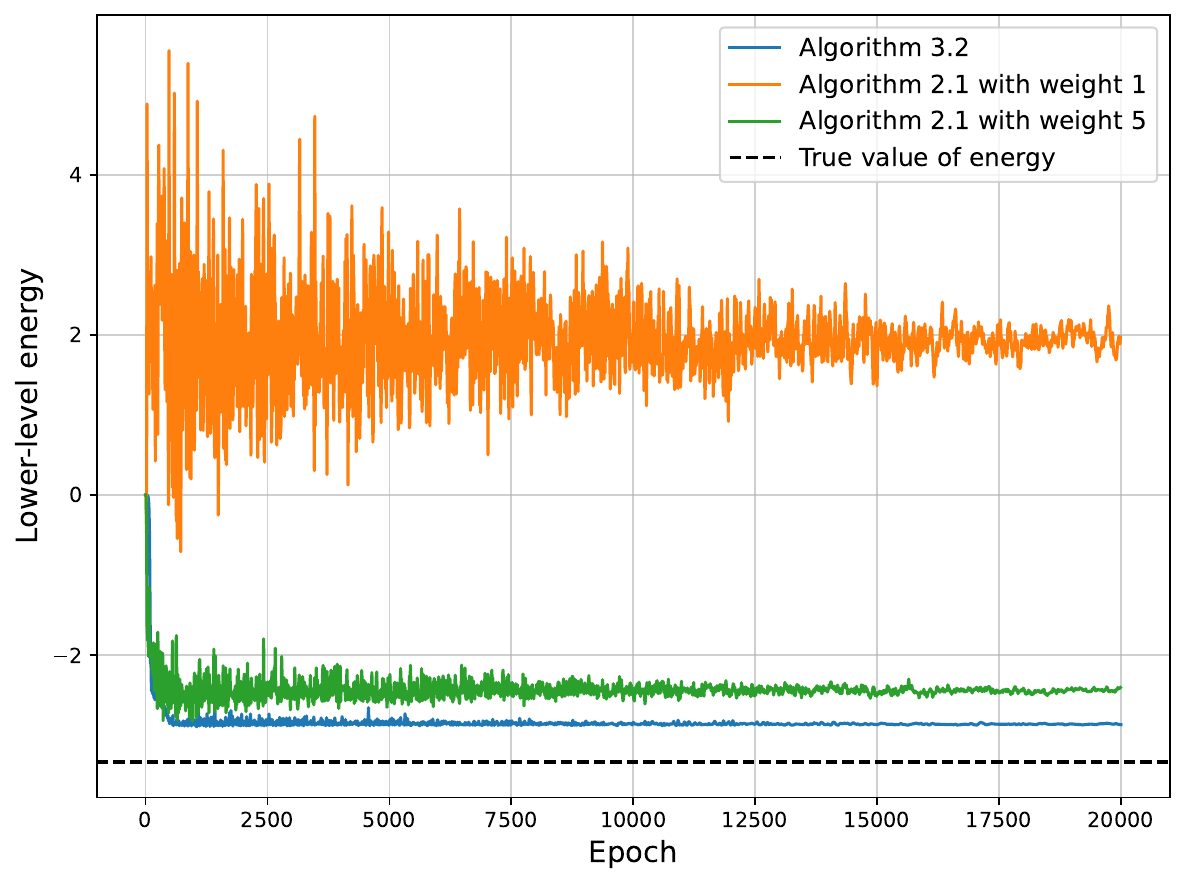}%
	}\hfill
	\subfloat[Comparisons of the computed upper-level objective functional value]{%
		\includegraphics[width=0.48\textwidth]{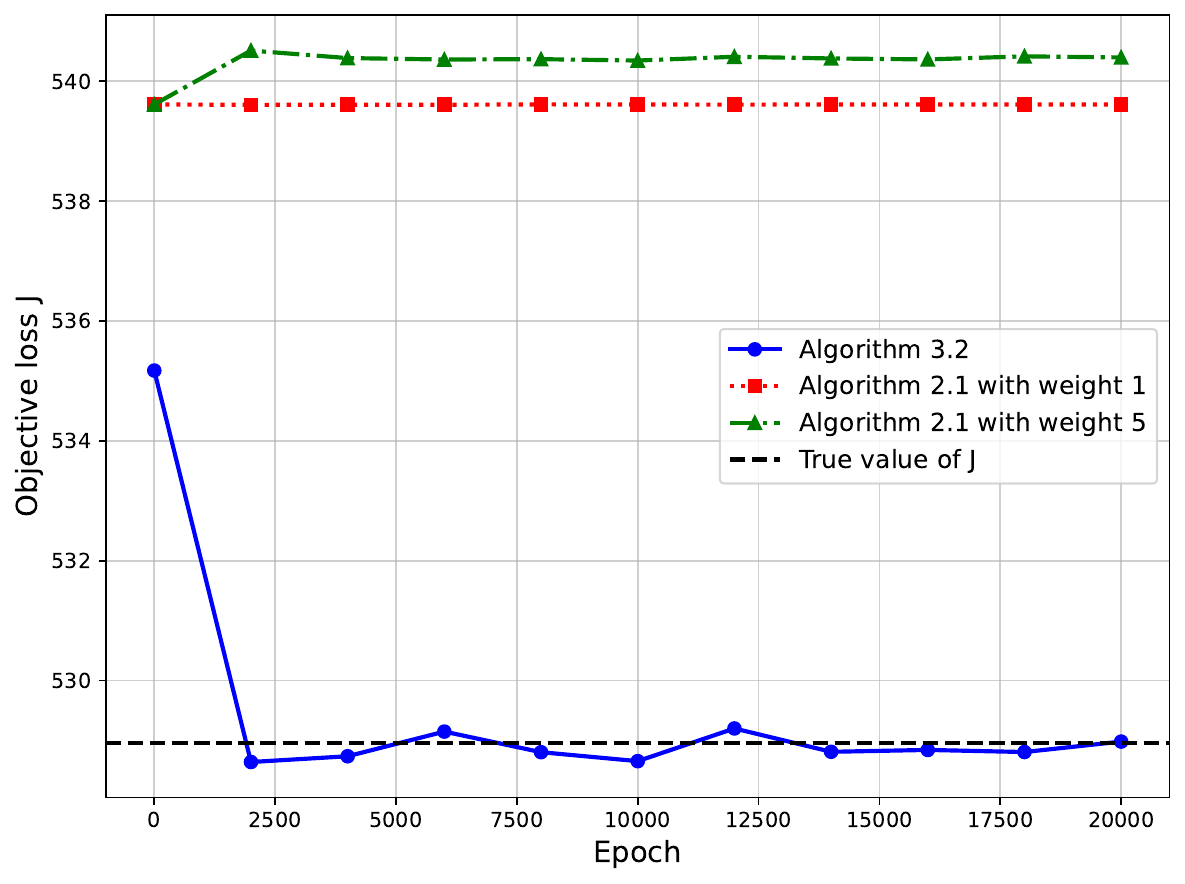}%
	}
	\caption{\small \em Numerical results of  Algorithm~\ref{alg:Single-level Deep Learning Method} for {Example 1}}
	\label{fig:example2 vanilla}
\end{figure}

\begin{table}[htpb]
	\centering
	{\small
	\begin{tabular}{@{}lcc@{}}
		\toprule
		Algorithm                        & \({\|\hat y - y^*\|_{L^2(\Omega)}}/{\|y^*\|_{L^2(\Omega)}}\) 
		& \({\|\hat u - u^*\|_{L^2(\Omega)}}/{\|u^*\|_{L^2(\Omega)}}\) \\
		\midrule
		Algorithm~\ref{alg:Bilevel Deep Learning Method}           & \(4.4658\times10^{-3}\)      
		& \(1.4949\times10^{-2}\)                  \\
		Algorithm~\ref{alg:Single-level Deep Learning Method} (\(w=1\))    & \(1.0127\times10^{0}\)      
		& \(1.0113\times10^{0}\)                  \\
		Algorithm~\ref{alg:Single-level Deep Learning Method} (\(w=5\))    & \(2.7236\times10^{-1}\)      
		& \(5.3694\times10^{0}\)                   \\
		\bottomrule
	\end{tabular}
  \caption{\small \em Relative \(L^2\)-errors of  Algorithm~\ref{alg:Single-level Deep Learning Method} and  Algorithm~\ref{alg:Bilevel Deep Learning Method} for {Example 1}\label{tab:rel_l2_errors_example1}}
}

\end{table}

Finally, we compare Algorithm \ref{alg:Bilevel Deep Learning Method} with the active set method \cite{hintermuller2008active} in terms of the computational time and the relative \( L^2 \)-errors calculated on uniform grids of size \( N \times N \) over the domain \( \Omega \) for \( N = 16, 32, 64, \dots, 1024 \).
From Table~\ref{tab:rel_l2_errors_yu}, we can see that, for $N\leq 256$, Algorithm~\ref{alg:Bilevel Deep Learning Method} achieves substantially lower relative $L^2$-errors than the active set method. Since Algorithm~\ref{alg:Bilevel Deep Learning Method} is mesh-free, it results in a consistent error that is largely independent of the grid resolution. The active set method, in contrast, is resolution-dependent; its accuracy improves with a finer mesh, allowing for high-precision solutions. Furthermore, once the NNs are trained at the resolution of $N=32$, evaluating Algorithm \ref{alg:Bilevel Deep Learning Method} at a new resolution requires only a forward pass. Classical methods like the active set method, however, require a complete re-meshing and re-computation for each resolution $N$, leading to significantly higher computational costs.

\begin{table}[htpb] 
	{\small
		\centering  
		\setlength{\tabcolsep}{2.5pt}
		\begin{tabular}{c|ccc|ccc}     
			\toprule
			\multirow{2}{*}{$N$} & \multicolumn{3}{c}{The Active Set Method~\cite{hintermuller2008active}} & \multicolumn{3}{c}{Algorithm~\ref{alg:Bilevel Deep Learning Method} } \\
			\cmidrule(lr){2-4} \cmidrule(lr){5-7}
			& $\|y-y^*\|/\|y^*\|$ & $\|u-u^*\|/\|u^*\|$ & Time(s) & $\|y-y^*\|/\|y^*\|$ & $\|u-u^*\|/\|u^*\|$ & Time(s) \\
			\midrule
			16   & $2.458\times10^{-1}$ & $2.360\times10^{-1}$ & 0.053 & $4.370\times10^{-3}$ & $1.456\times10^{-2}$ & 0.004 \\
			32   & $1.289\times10^{-1}$ & $1.286\times10^{-1}$ & 0.058 & $4.364\times10^{-3}$ & $1.490\times10^{-2}$ & 0.004 \\
			64   & $6.453\times10^{-2}$ & $6.952\times10^{-2}$ & 0.173 & $4.459\times10^{-3}$ & $1.495\times10^{-2}$ & 0.004 \\
			128  & $3.093\times10^{-2}$ & $3.517\times10^{-2}$ & 1.257 & $4.467\times10^{-3}$ & $1.495\times10^{-2}$ & 0.005 \\
			256  & $1.472\times10^{-2}$ & $1.694\times10^{-2}$ & 9.591 & $4.467\times10^{-3}$ & $1.495\times10^{-2}$ & 0.006 \\
			512  & $7.042\times10^{-3}$ & $9.067\times10^{-3}$ & 69.701 & $4.467\times10^{-3}$ & $1.495\times10^{-2}$ & 0.010 \\
			1024 & $3.166\times10^{-3}$ & $5.566\times10^{-3}$ & 178.749 & $4.467\times10^{-3}$ & $1.495\times10^{-2}$ & 0.044 \\
			\bottomrule
		\end{tabular}  
    \caption{\small \em Relative $L^2$ errors of the computed solutions using Algorithm~\ref{alg:Bilevel Deep Learning Method} and the active set method~\cite{hintermuller2008active} for Example 1. The results of the active set method are computed and evaluated with each mesh resolution $N$, while the results of Algorithm \ref{alg:Bilevel Deep Learning Method} are
			computed with fixed training resolution $N= 32$ and evaluated with each mesh resolution $N$.}
    \label{tab:rel_l2_errors_yu}
	}
\end{table}


\medskip

\noindent \textbf{Example 2 (Flat free boundary).}
This example is adapted (with a simple scaling) from Example 6.2 in  \cite{hintermuller2009mathematical}. Let \(\Omega=(0,1)^2\), \(\sigma=0.02\), \(Y_{ad}=\{y\in H_0^1(\Omega)\mid y\ge0\text{ a.e.\ in }\Omega\}\),
and \(U_{ad}=L^2(\Omega)\). Consider the non‐smooth source
$f(x_1,x_2)=y_d(x_1,x_2)=-5\bigl|x_1x_2-0.5\bigr|+1.25.$
In the absence of an exact solution, we employ the relaxed MPEC method from  \cite{hintermuller2009mathematical} on uniform grids with mesh size \( h = 1/100 \) to compute a reference pair \((y^*, u^*)\).
We employ the NNs given by \eqref{eq:state_embedding}-\eqref{eq:control_embedding} and implement Algorithm~\ref{alg:Bilevel Deep Learning Method} with
$
T = 20,000,
\gamma = 500,$ and $c_k = 5k^{0.3}$.

Figure \ref{fig:meha_3} presents the computed state \(y\), control \(u\), the reference solutions, and their pointwise errors. 
The final relative \(L^2\)-errors computed on a mesh (\(h=1/100\)) for the state and the control are respectively $6.70\times10^{-3}$ and $2.07\times10^{-2}$.

In this example, the very flat transition of the optimal state \(y^*\) from the inactive set $\left\{x \in \Omega \mid y(x)>\psi(x)\right\}$ to the active set $\left\{x \in \Omega \mid y(x)=\psi(x)\right\}$ makes active-set detection particularly challenging. In such cases, purely primal active set techniques may be less effective. By contrast, leveraging the mesh-free nature of Algorithm~\ref{alg:Bilevel Deep Learning Method} and the constraint-embedding enforced architecture \eqref{eq:state_embedding}-\eqref{eq:control_embedding}, our numerical solution strictly satisfies the constraints and exhibits robust accuracy at the free boundary.


\FloatBarrier
\begin{figure}[!htbp]
    \centering
    \subfloat[Reference state $y^*$]{\includegraphics[width=0.30\textwidth]{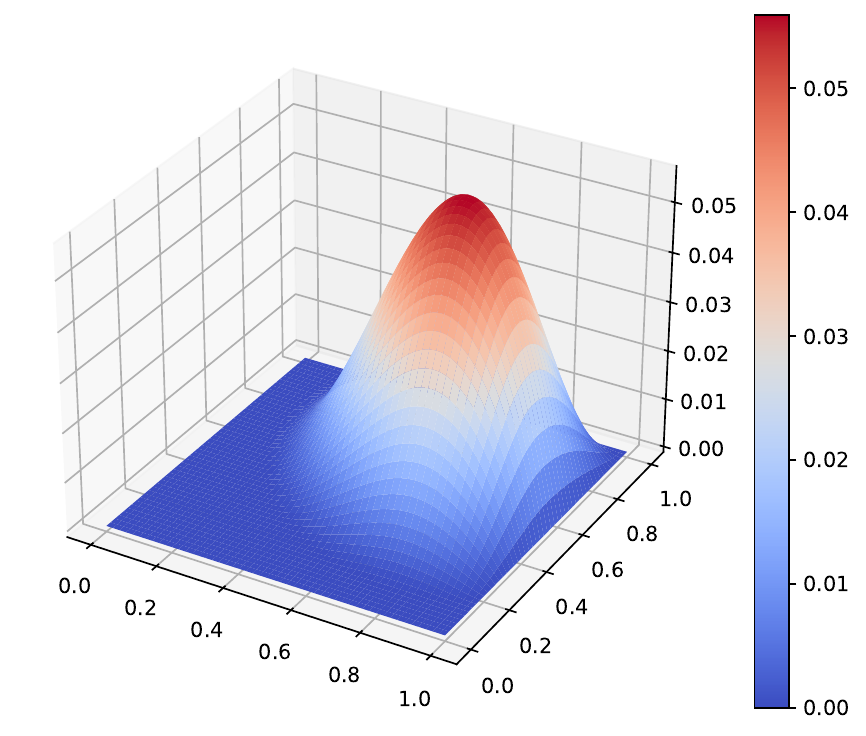}}
    \hfill
    \subfloat[Computed state $\hat y$]{\includegraphics[width=0.30\textwidth]{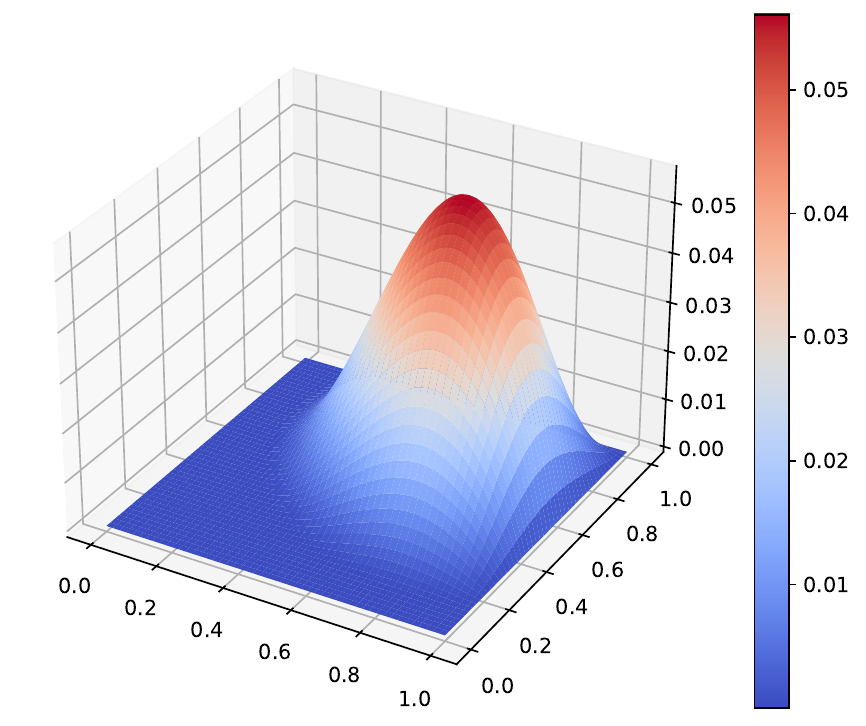}}
    \hfill
    \subfloat[Pointwise error of state]{\includegraphics[width=0.30\textwidth]{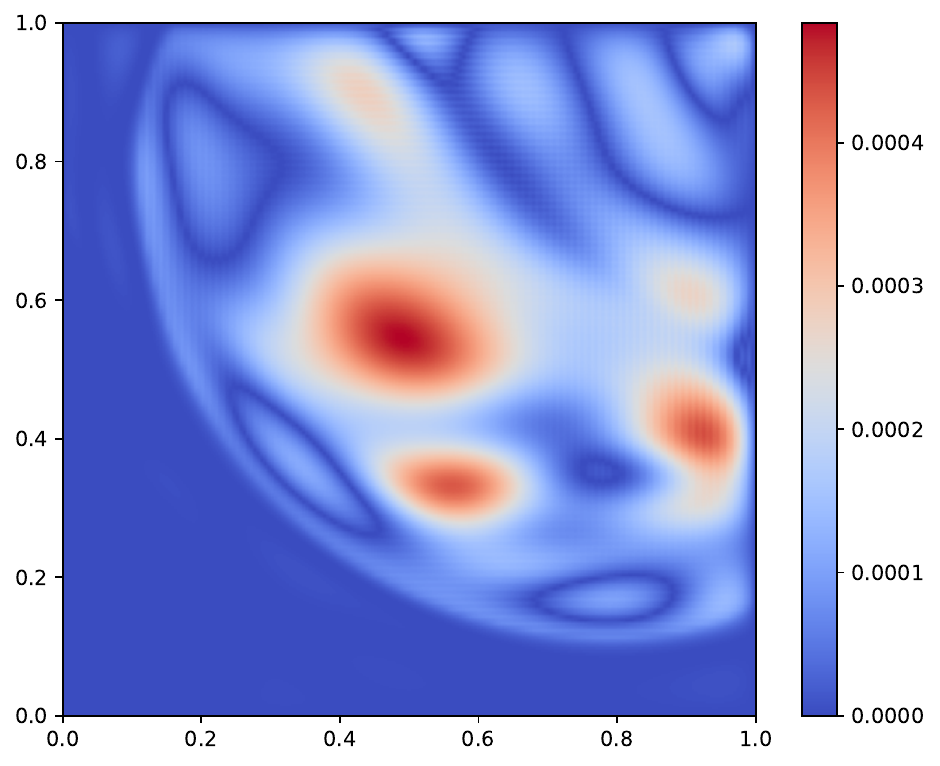}}

    \subfloat[Reference control $u^*$]{\includegraphics[width=0.30\textwidth]{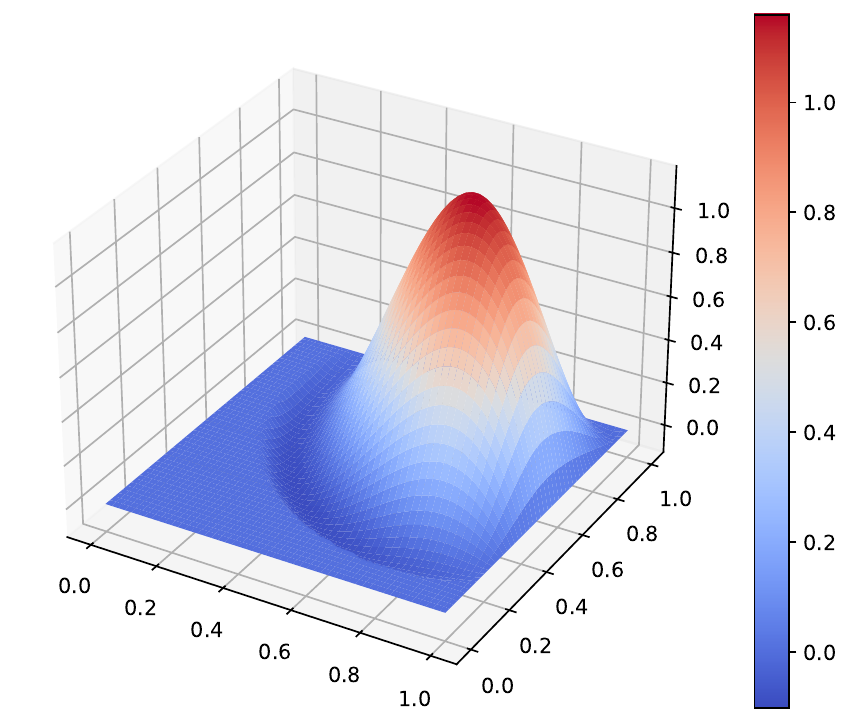}}
    \hfill
    \subfloat[Computed control $\hat u$]{\includegraphics[width=0.30\textwidth]{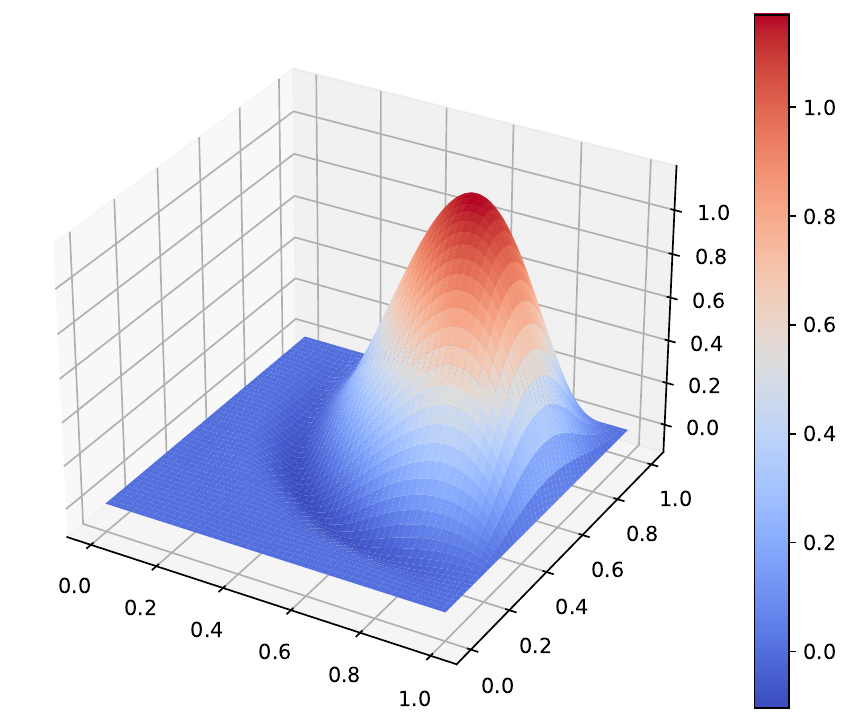}}
    \hfill
    \subfloat[Pointwise error of control]{\includegraphics[width=0.30\textwidth]{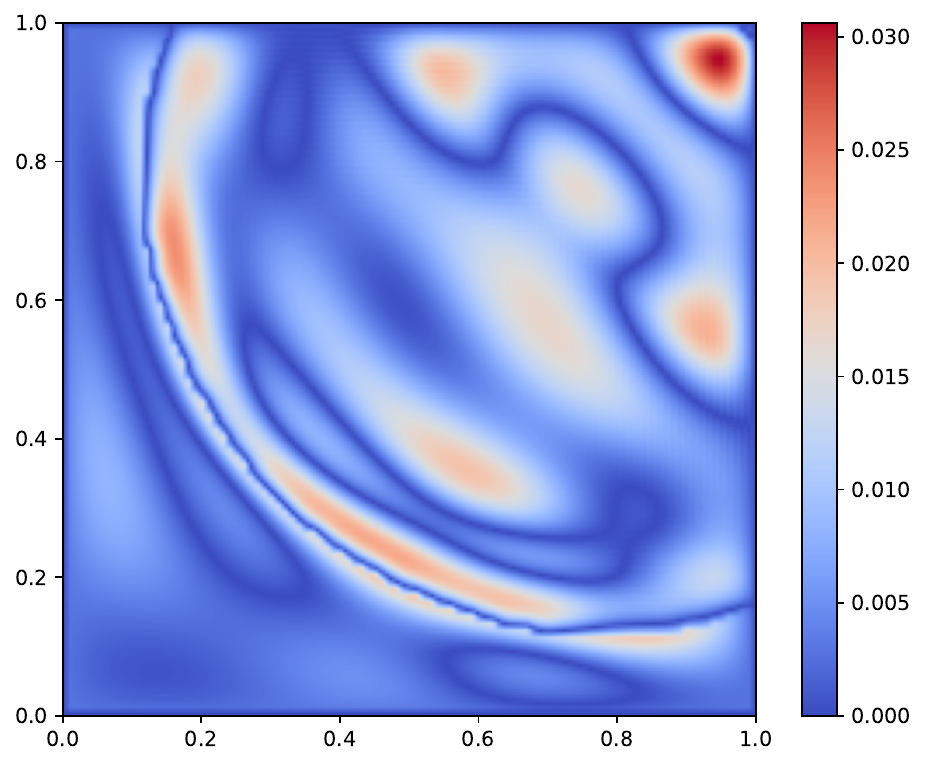}}

    \caption{\small \em Numerical results of Algorithm~\ref{alg:Bilevel Deep Learning Method} for {Example 2}}
    \label{fig:meha_3}
\end{figure}

\medskip

\noindent \textbf{Example 3 (High-dimensional space).} To demonstrate the scalability of the proposed method in high-dimensional settings, we consider a distributed optimal control problem in a 5-dimensional domain. Let $\Omega=(0,1)^5$, $\sigma=1$, $Y_{ad}=\{y\in H_0^1(\Omega)\mid y(x)\ge0\text{ a.e.\ in }\Omega\}$, and $U_{ad}=L^2(\Omega)$. We define an auxiliary function $P(x)=\prod_{i=1}^5 \sin(\pi x_i)$ and a threshold parameter $c=0.01$. The optimal state and control are constructed as
\[
y^*(x)=u^*(x)=\bigl(\max(0,P(x)-c)\bigr)^2.
\]
The multiplier is given by $\xi^*(x)=\bigl(\max(0,c-P(x))\bigr)^2$. Accordingly, the target state $y_d$ and the external forcing term $f$ are analytically derived to satisfy the stationarity conditions:
\[
y_d(x)=y^*(x)-\Delta y^*(x), \quad f(x)=-\Delta y^*(x)-y^*(x)-\xi^*(x).
\]

To overcome the curse of dimensionality inherent in classical mesh-based methods, we introduce two key modifications. First, we parameterize the state and control using Separable Physics-Informed Neural Networks (SPINNs) \cite{cho2023separable}, where the raw network output $\mathcal{N}(x;\theta)$ is formulated as:
\[
\mathcal{N}(x;\theta)=\sum_{j=1}^r \prod_{i=1}^5 f_j^{(\theta_i)}(x_i).
\]
Second, to efficiently approximate high-dimensional integrals, we replace standard pseudo-random sampling with a quasi-Monte Carlo strategy utilizing low-discrepancy Sobol sequences.

We then train the networks using Algorithm~\ref{alg:Bilevel Deep Learning Method}. Figure~\ref{fig:meha_high_dim} reports the two-stage training losses, and compares the reference solutions, computed solutions, and pointwise errors on 2D cross-sectional slices obtained by fixing three spatial coordinates. The results demonstrate that the proposed method accurately captures the optimal solutions, effectively overcoming the curse of dimensionality.

\begin{figure}[htpb]
    \centering
    \subfloat[Upper-level loss]{\includegraphics[width=0.45\textwidth]{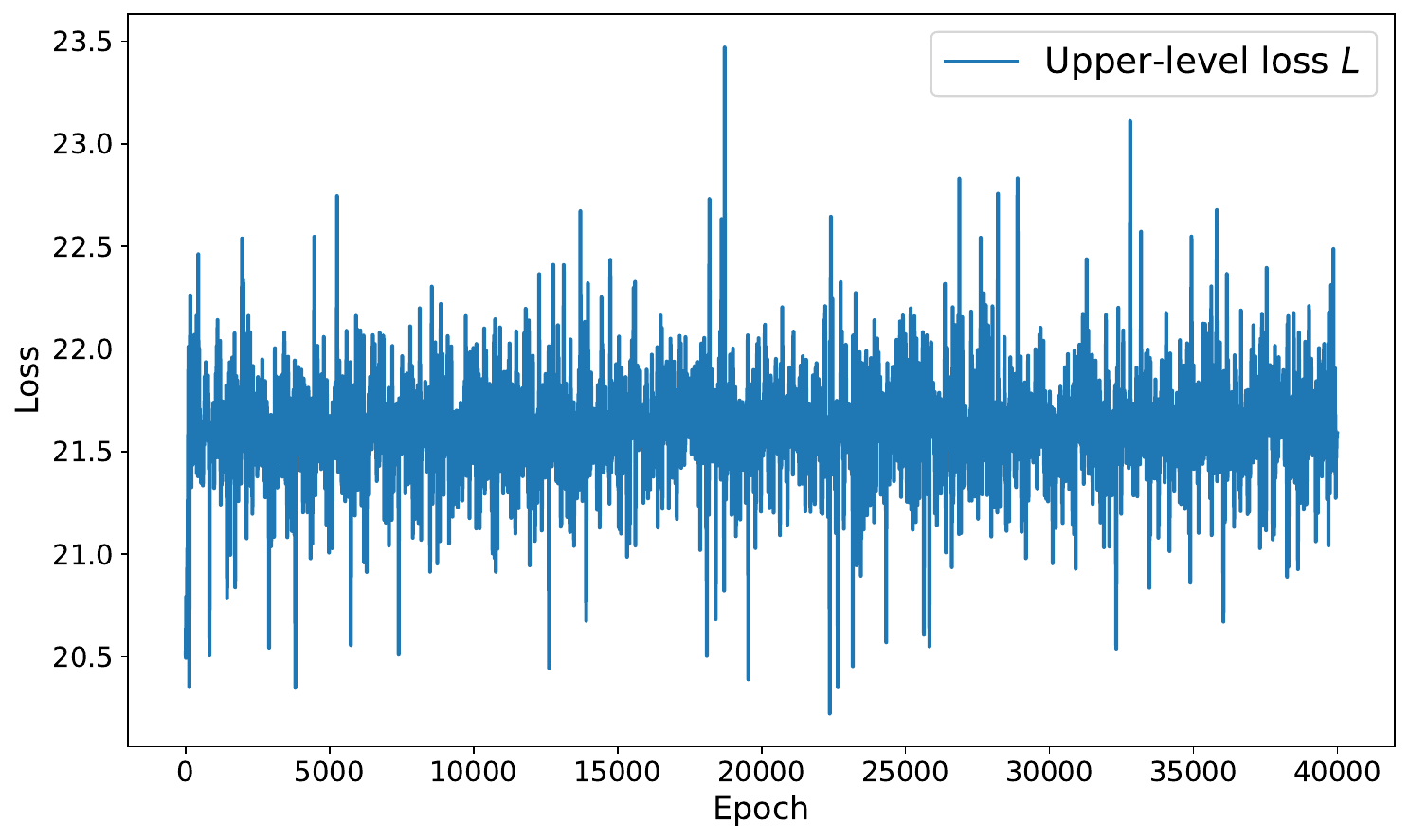}}
    \hfill
    \subfloat[Lower-level loss]{\includegraphics[width=0.45\textwidth]{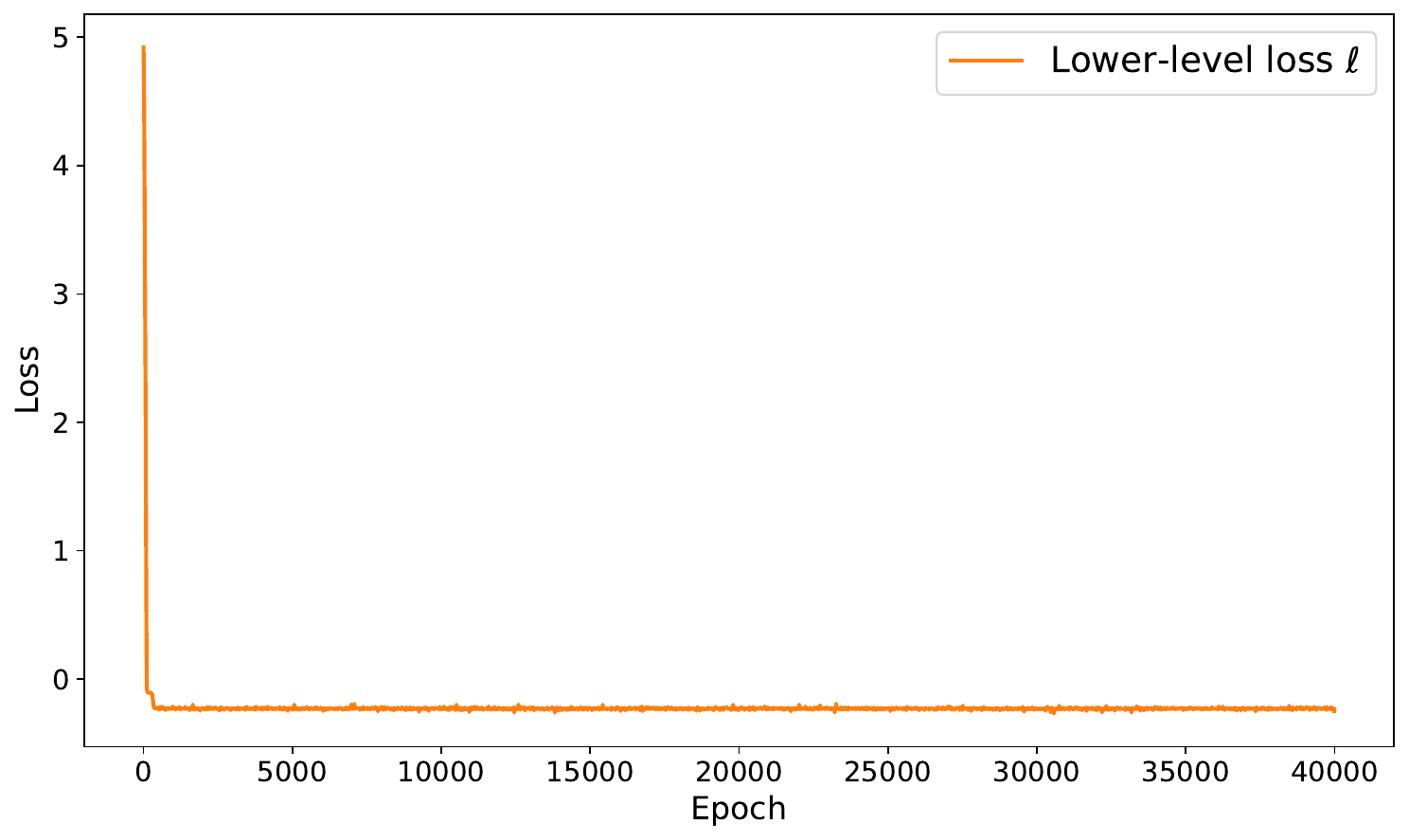}}
    
    \vspace{1em}
    
    \subfloat[Reference state $y^*$]{\includegraphics[width=0.30\textwidth]{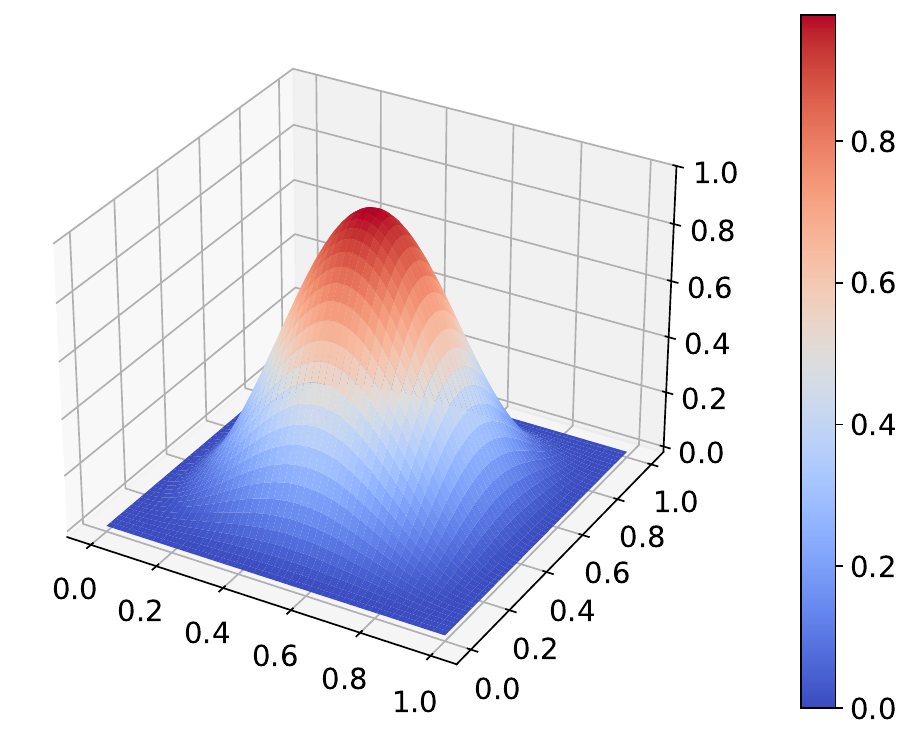}}
    \hfill
    \subfloat[Computed state $\hat y$]{\includegraphics[width=0.30\textwidth]{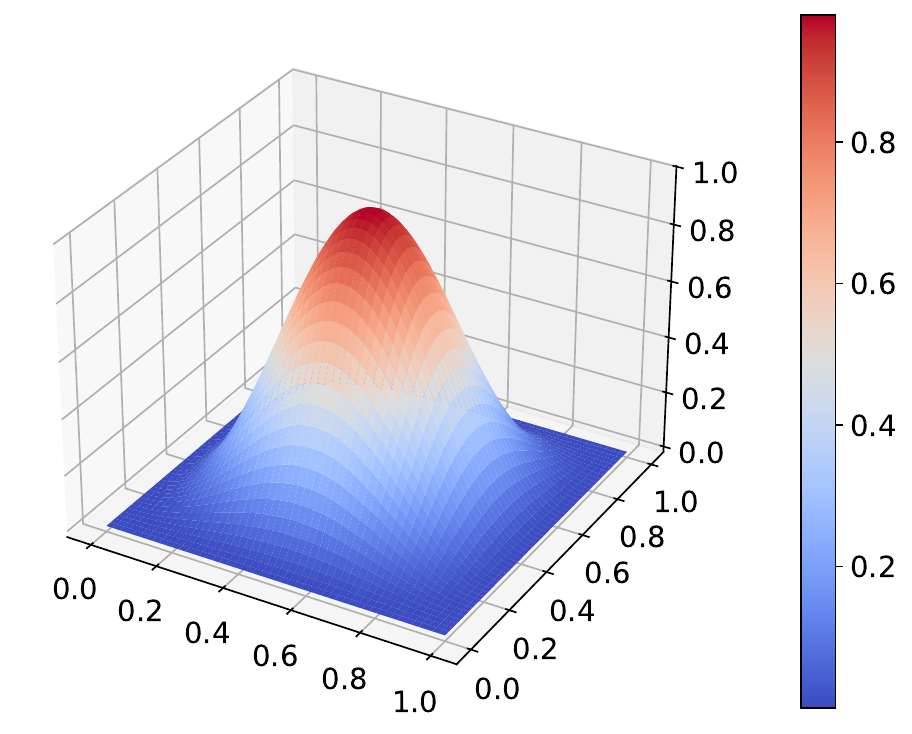}}
    \hfill
    \subfloat[Pointwise error of state]{\includegraphics[width=0.30\textwidth]{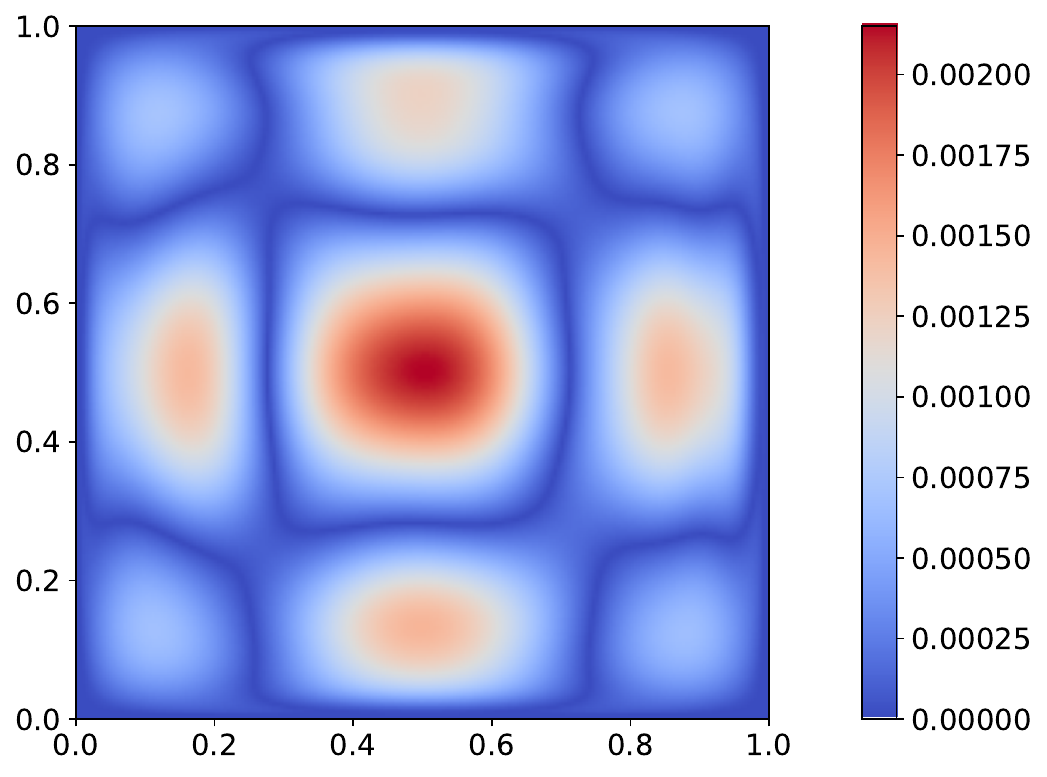}}
    
    \vspace{1em}
    
    \subfloat[Reference control $v^*$]{\includegraphics[width=0.30\textwidth]{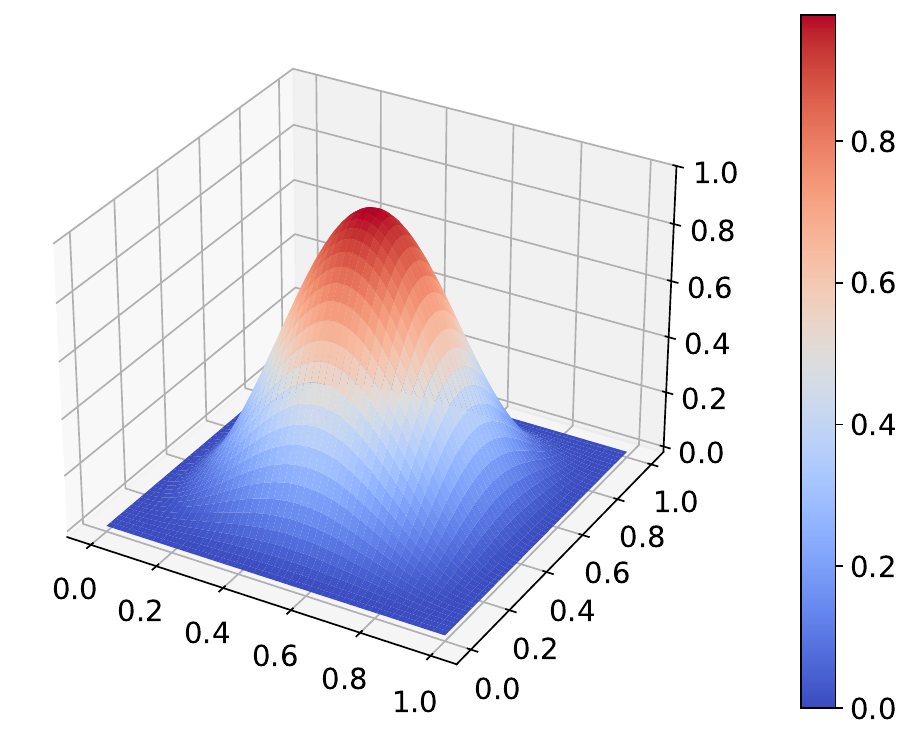}}
    \hfill
    \subfloat[Computed control $\hat v$]{\includegraphics[width=0.30\textwidth]{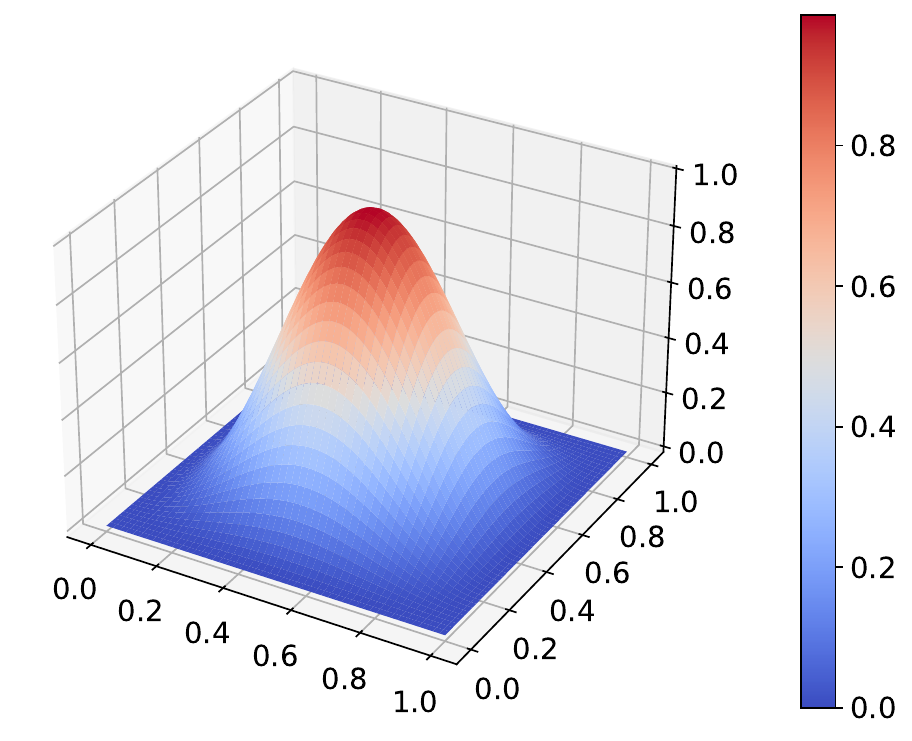}}
    \hfill
    \subfloat[Pointwise error of control]{\includegraphics[width=0.30\textwidth]{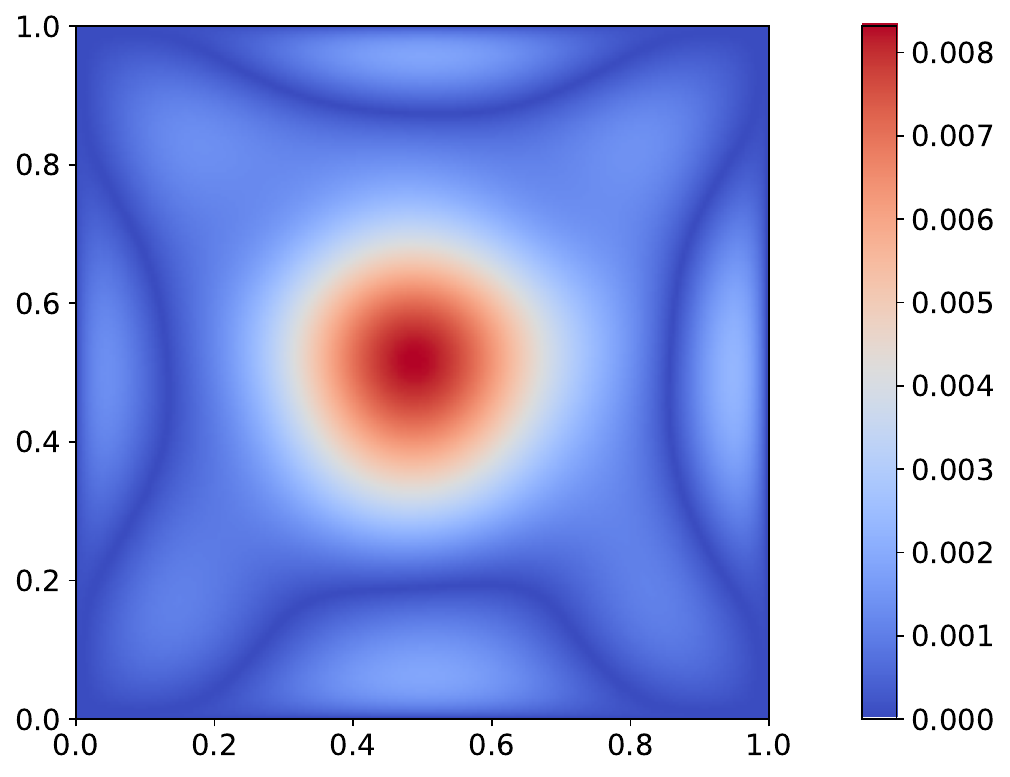}}
    
    \caption{\small \em Numerical results of Algorithm~\ref{alg:Bilevel Deep Learning Method} for {Example 3}. The state and control fields are visualized as 2D cross-sectional slices obtained by fixing three spatial coordinates.}
    \label{fig:meha_high_dim}
\end{figure}

\medskip

\noindent \textbf{Example 4 (Complex domain).}
Classical numerical methods for optimal control of obstacle problems are usually implemented with mesh-based discretization schemes (e.g., FEMs and FDMs), which require specialized meshes to conform to complex boundaries. As a result, their accuracy depends heavily on mesh quality. By contrast, Algorithm~\ref{alg:Bilevel Deep Learning Method} is mesh-free, which eliminates the need for costly mesh generation and allows for seamless handling of irregular geometries. To demonstrate this capability, we consider a problem defined on a complex domain.

Let $\rho_1(\zeta)=2.25+0.21 \sin (4 \zeta)+0.18 \cos (6 \zeta)+0.135 \cos (5 \zeta)$ and  
\[
\Omega=\left\{\left(x_1, x_2\right) \in \mathbb{R}^2 \mid x_1=r \cos (\zeta), x_2=r \sin (\zeta), 0 \leq r<\rho_1(\zeta), 0 \leq \zeta<2 \pi\right\}.
\]
Then, we consider the problem \eqref{eq:distribution} with $U_{ad}=L^2(\Omega)$ and
$
\psi(x_1,x_2)
=3\Bigl(1-\Bigl(\frac{r}{\rho_1(\zeta)}\Bigr)^2\Bigr),
$
where $
 r=\sqrt{x_1^2+x_2^2},$ and
$\zeta \in [0, 2\pi)$ is the polar angle such that $x_1 = r \cos \zeta$ and $x_2 = r \sin \zeta$.
 Set $\sigma = 1$ and 
 $f=y_d\equiv 2.$
We use the
NNs given by \eqref{eq:control-embedding-relu}-\eqref{eq:state-embedding-relu} and then implement Algorithm~\ref{alg:Bilevel Deep Learning Method} with
$T = 20,000, \gamma = 50$, and $c_k = 5k^{0.2}$.

Figure~\ref{fig:meha_complex_boundary} displays the training losses during the bilevel training process, the computed state \(\hat y\), and the computed control \(\hat u\). We can see that the control exhibits a flower-shaped pattern: under the given setup, it pushes the state downward within the contact set until it meets the obstacle, and lifts it upward outside the contact region. These results verify that our algorithm can effectively tackle problems in geometrically complex domains.

%
%
%
\
\begin{figure}[htpb]
    \centering
    	\subfloat[Upper-level loss]{%
    	\includegraphics[width=0.45\textwidth]{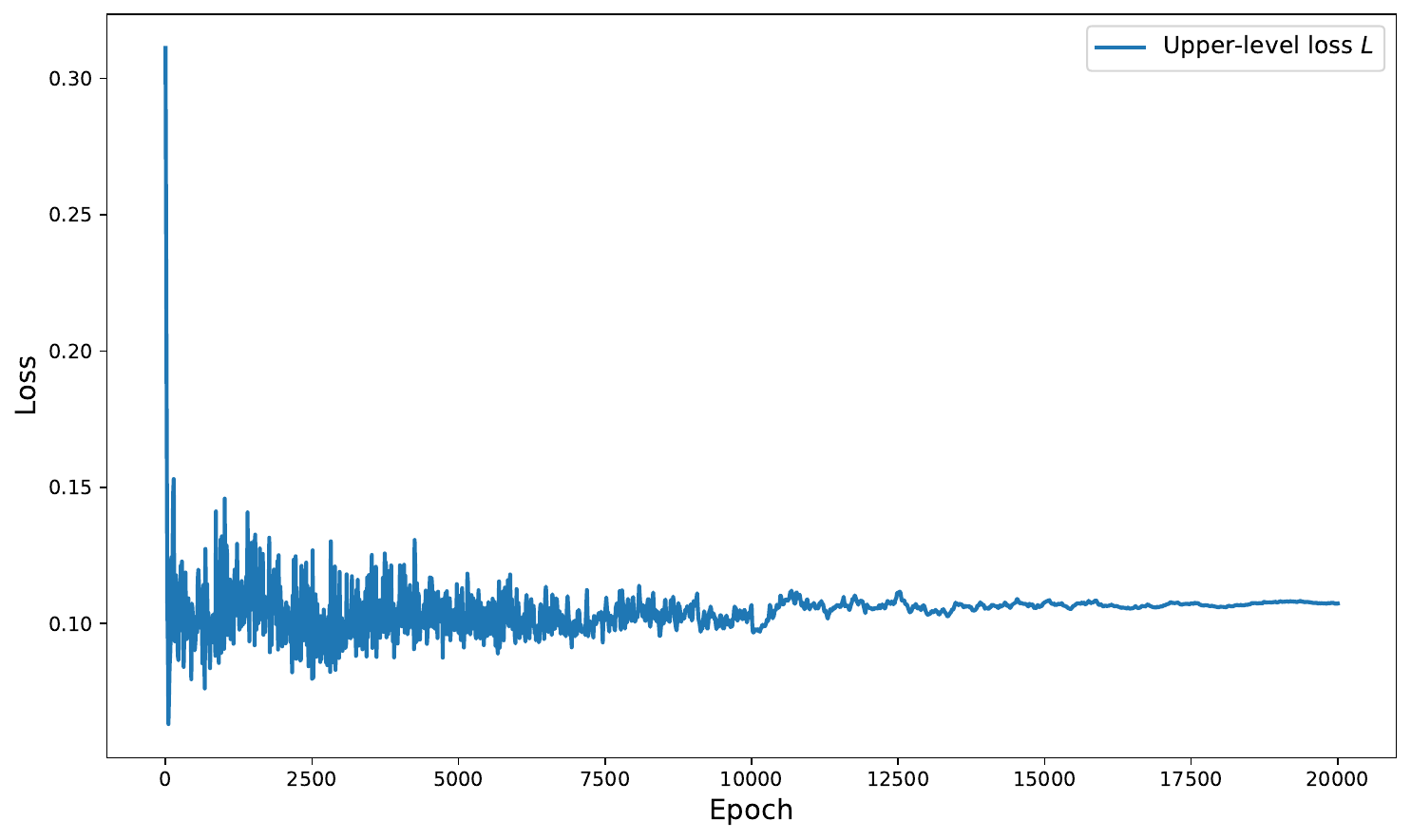}
    	
    }
    \hfill
    \subfloat[Lower-level loss]{%
    	\includegraphics[width=0.45\textwidth]{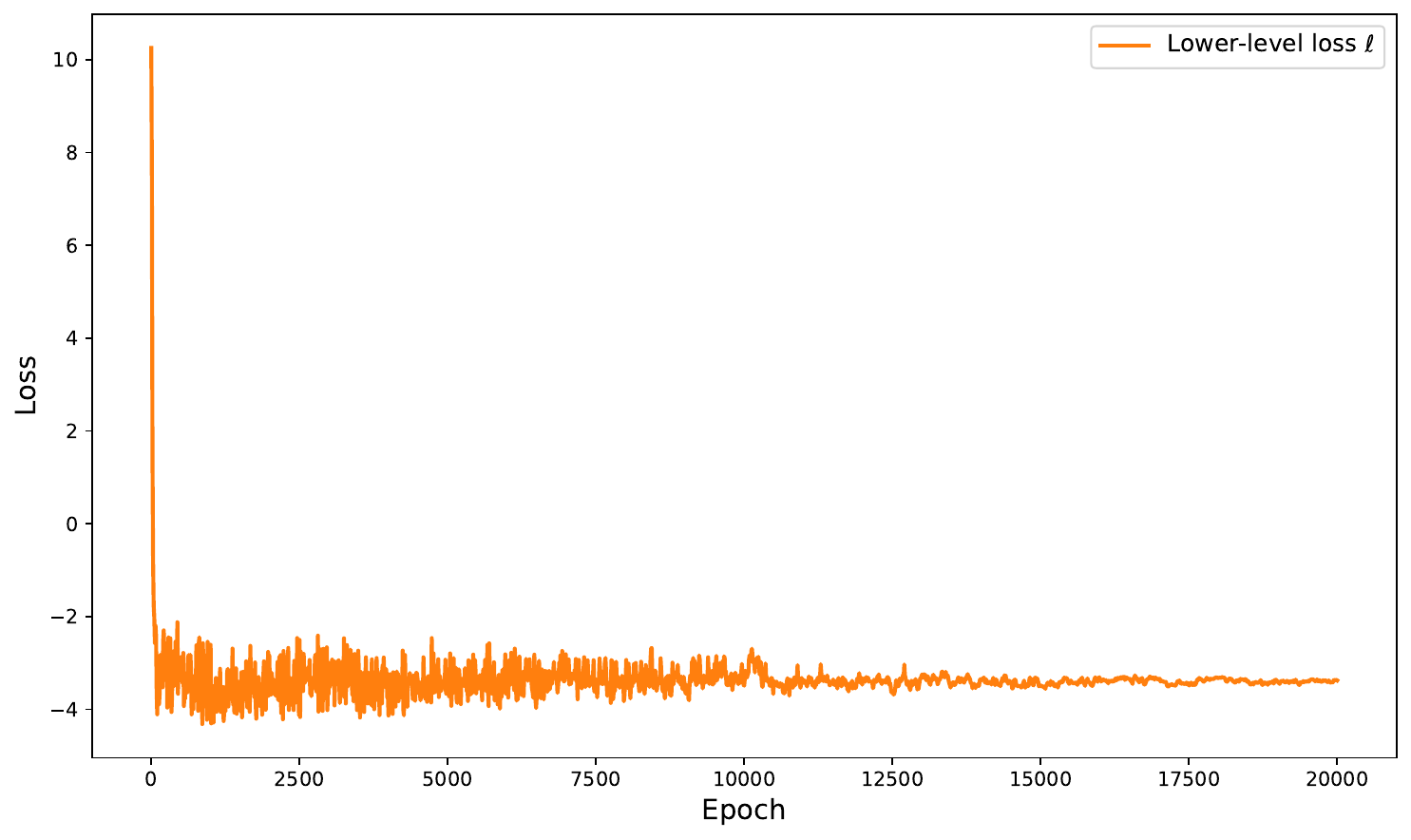}
    	
    }
    
    \subfloat[Computed state $\hat y$]{%
        \includegraphics[width=0.45\textwidth]{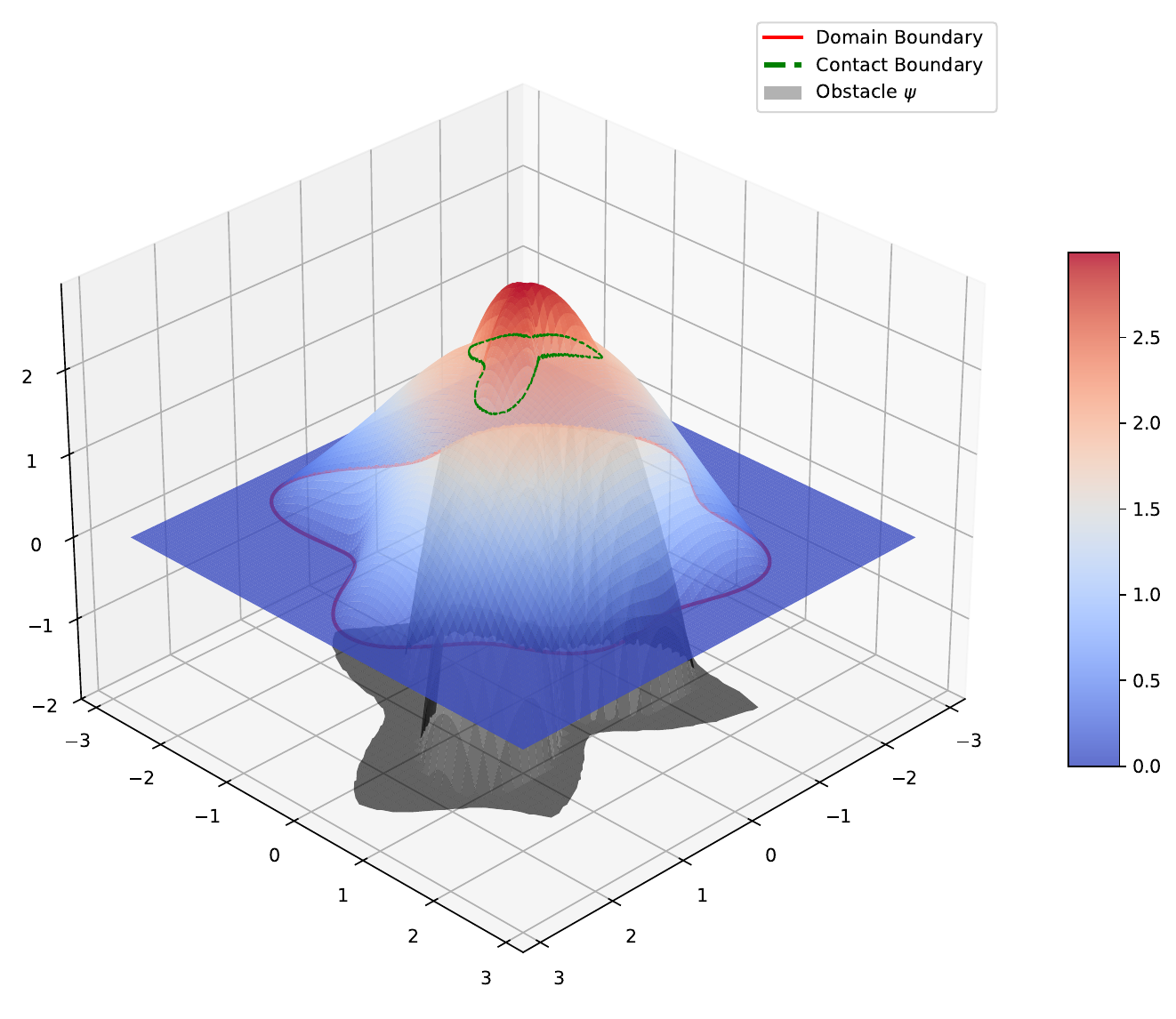}
    }
    \hfill
    \subfloat[Computed control $\hat u$]{%
        \includegraphics[width=0.45\textwidth]{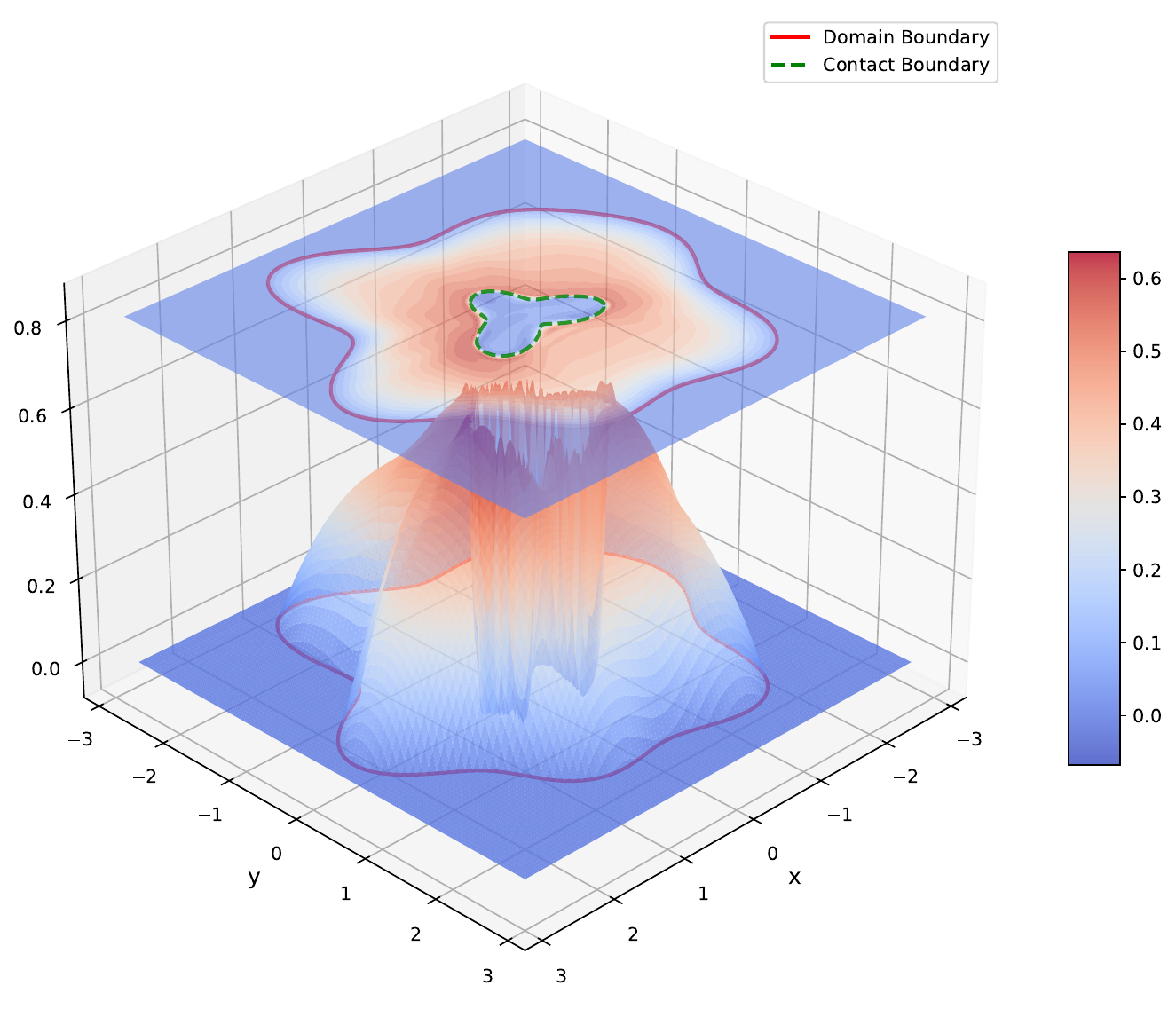}
    }

    \caption{\small \em Numerical Results of Algorithm~\ref{alg:Bilevel Deep Learning Method} for {Example 4}}
    \label{fig:meha_complex_boundary}
\end{figure}

\section{Extensions}\label{se:extension}
In this section, we demonstrate that Algorithm~\ref{alg:Bilevel Deep Learning Method} can be readily adapted to some variants of the problem~\eqref{general control}.

\subsection{Obstacle Control}
In addition to the distributed control $u$ considered in \eqref{eq:distribution}, the control can be the obstacle function \(\psi\). 
 For convenience, we consider the situation $\psi \in H_0^1(\Omega)$.  Let
\(\Omega\subset \mathbb{R}^d\) be a bounded domain, and define
\[
Y_{ad}
=\{\,y\in H_0^1(\Omega)\mid y(x)\le\psi(x)\text{ a.e.\ in }\Omega\}.
\]
Then, following \cite{ito2007optimal}, we consider the following optimal control problem with obstacle control \(\psi\in H_0^1(\Omega)\), where no control constraints are imposed on \(\psi\). 
\begin{equation} \label{eq:obstacle control}
\begin{aligned} 
\min _{y \in H_0^1(\Omega), \psi \in H_0^1(\Omega)}
J(y,\psi)
&=\frac{1}{2}\int_\Omega|y-y_d|^2\mathrm{d}x+\frac\sigma2\!\int_\Omega|\nabla\psi|^2\mathrm{d}x,\\
\quad\text{s.t.}\quad
y&=\arg\min_{y'\in Y_{ad}}\int_{\Omega}\left(\frac{1}{2}|\nabla y'|^2-fy'\right) \mathrm{d} x.\\
\end{aligned}
\end{equation}

We first approximate $\psi$ by the NN
\begin{equation} \label{embedding-extension-u}
  \hat \psi(x;\theta_\psi):= m(x)\mathcal{N}(x,\theta_\psi),  
\end{equation}
where \( m \in C^\infty(\overline{\Omega}) \) is chosen such that
    $
    m(x) = 0~\forall x \in \partial \Omega$ and $m(x) > 0~ \forall x \in \Omega.
    $
The state $y$ is approximated by the NN
\begin{equation} \label{embedding-extension-y-relu}
    \hat y(x;\theta_\psi,\theta_y):= -\operatorname{ReLU}\bigl(\hat \psi(x,\theta_\psi) - m(x)\mathcal{N}(x,\theta_y)\bigr) + \hat \psi(x,\theta_\psi).
\end{equation}
 Then, we obtain a bilevel optimization problem in the form of \eqref{eq: stochastic bi}, which can be solved by Algorithm \ref{alg:MEHA} for training the NNs.

\medskip
\noindent \textbf{Example 5 (Obstacle control).}
This example is adapted from {Example 3} in  \cite{ito2007optimal}, with \( \Omega = (0, 1)^2\) and \( \sigma = 0.5 \). The forcing term is defined as $f(x_1, x_2)=-100$ if  $x_2\in(0.25,0.65)$, and $f(x_1,x_2)=150$ elsewhere. 
The desired state is \(y_d\equiv5\). 

In this example, $f$ is discontinuous and the active set $\{x\in\Omega\mid y(x) = \psi(x) \}$ does not cover the entire domain \( \Omega \), which makes the problem more challenging. We use the active set method in \cite{ito2007optimal} to generate a reference solution on a uniform grid of mesh size \(h = 1/200\) with tolerance \(10^{-8}\). We employ the NNs given by \eqref{embedding-extension-u} and \eqref{embedding-extension-y-relu} and set
$T = 10,000, \gamma = 50,$ and $c_k = \frac{1}{5}k^{0.2}$ to implement Algorithm \ref{alg:Bilevel Deep Learning Method}. 

Figure~\ref{fig:meha_6} displays the computed state \(y\), control \(\psi\), and their corresponding reference solutions and pointwise errors. The final relative \(L^2\)-errors for the state and the control, computed on a uniform grid with a mesh size of \(h = 1/200\), are $ 4.25
\times 10^{-2}$ and $2.99 \times 10^{-2}$, respectively. 


\begin{figure}[htpb]
    \centering
    \subfloat[Reference state $y^*$]{\includegraphics[width=0.30\textwidth]{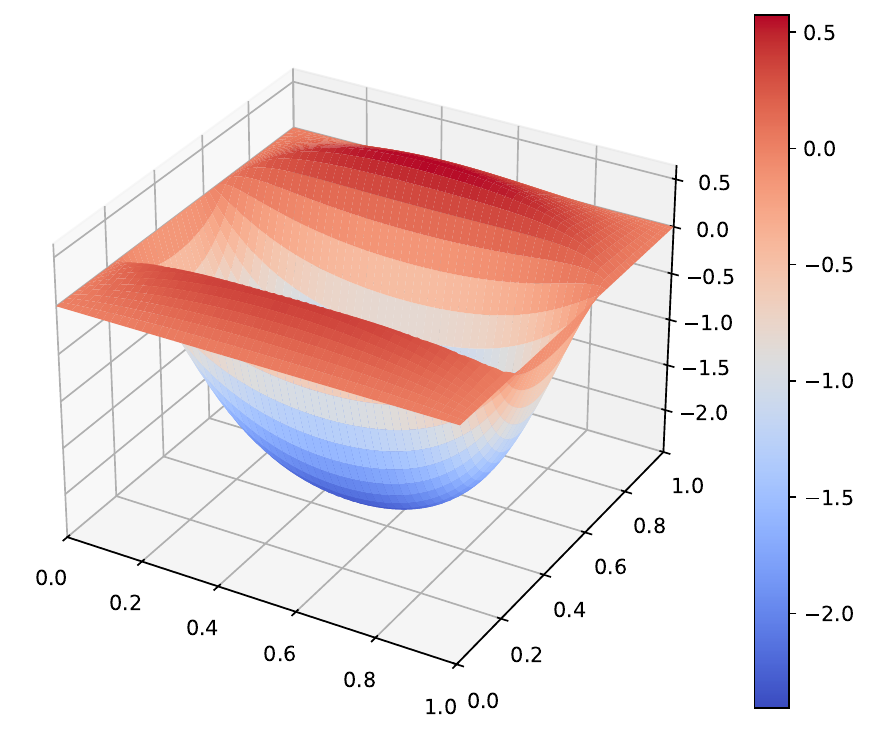}}
    \hfill
    \subfloat[Computed state $\hat y$]{\includegraphics[width=0.30\textwidth]{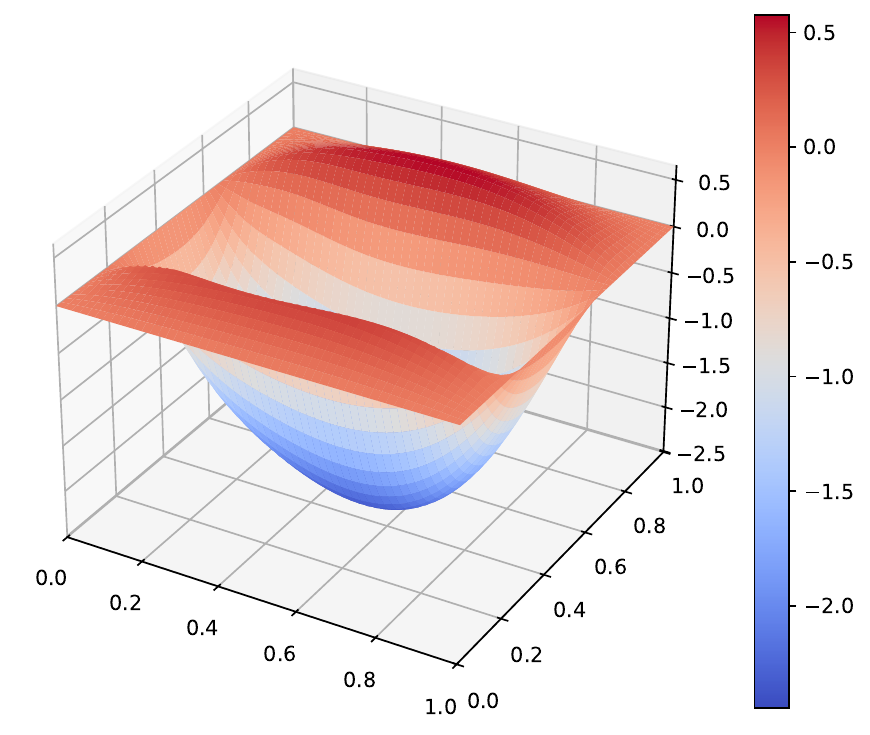}}
    \hfill
    \subfloat[Pointwise error of state]{\includegraphics[width=0.30\textwidth]{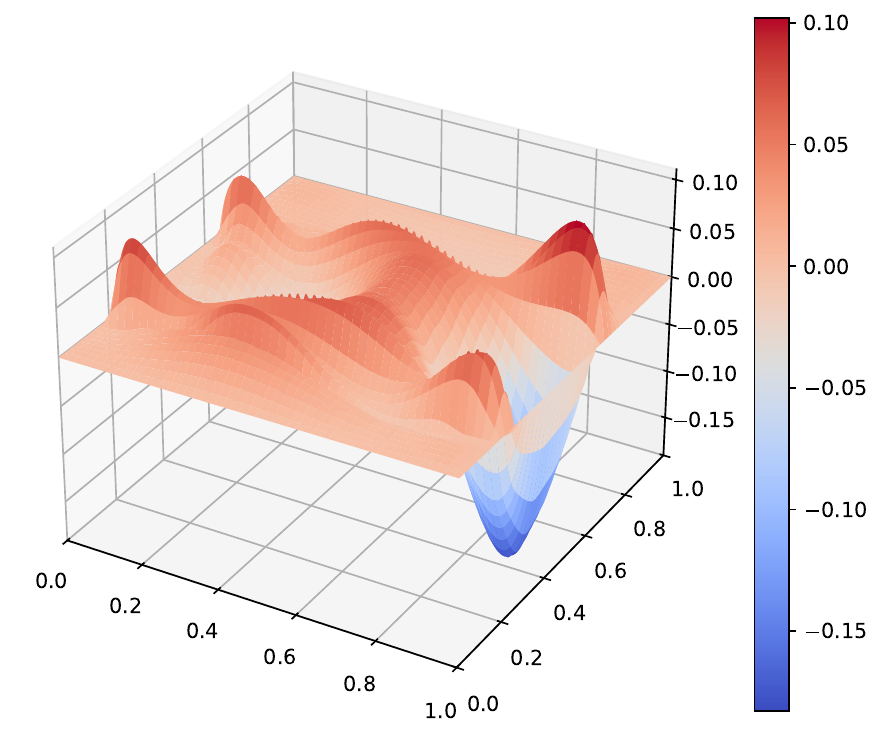}}

    \subfloat[Reference control $\psi^*$]{\includegraphics[width=0.3\textwidth]{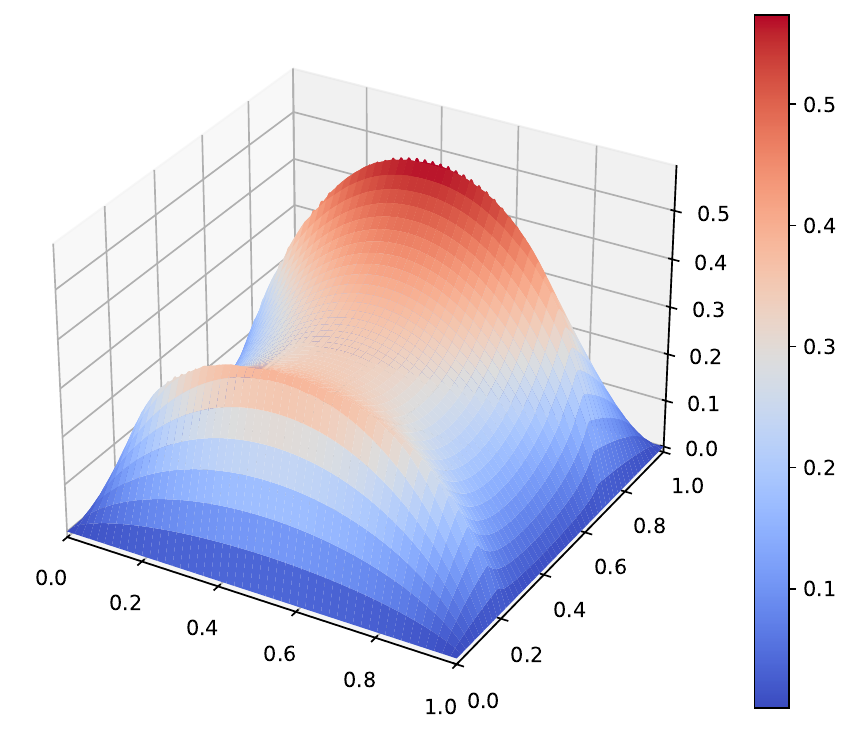}}
    \hfill
    \subfloat[Computed control $\hat \psi$]{\includegraphics[width=0.3\textwidth]{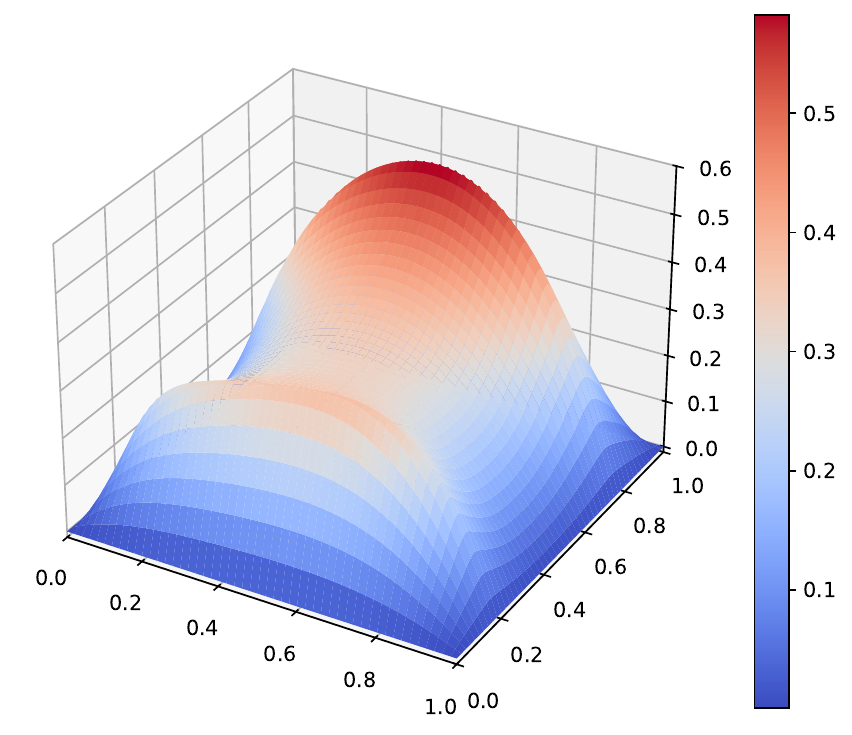}}
    \hfill
    \subfloat[Pointwise error of control]{\includegraphics[width=0.3\textwidth]{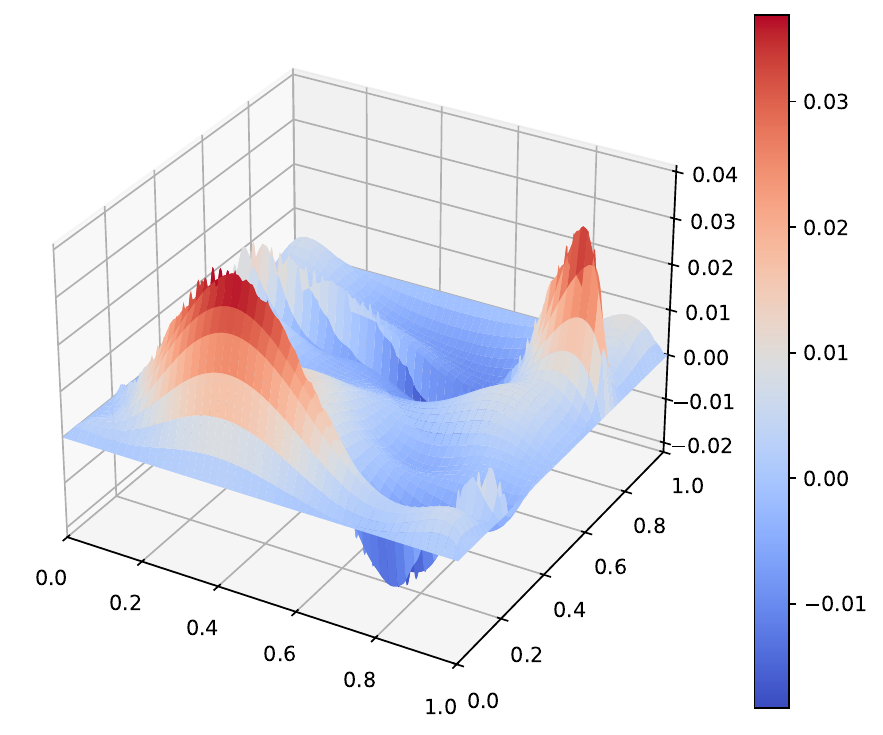}}

\caption{\small \em Numerical results of Algorithm~\ref{alg:Bilevel Deep Learning Method} for {Example 5}}
    \label{fig:meha_6}
\end{figure}

\subsection{Optimal Control of Elliptic Variational Inequalities}
In this section, we showcase that, with slight modifications, Algorithm \ref{alg:Bilevel Deep Learning Method} can be extended to optimal control of  EVIs. To fix ideas, we consider
\begin{equation}\label{eq:control_evi}
	\left\{
	\begin{aligned}
		\min_{y \in H_0^1(\Omega),\,u \in U_{ad}}&J(y, u)
		:=\frac{1}{2}\|y - y_d\|_{L^2(\Omega)}^2 + \frac{\sigma}{2}\|u\|_{L^2(\Omega)}^2 \\
		\text{s.t.}\quad 
		& \langle Ay,\, v-y\rangle \ge \langle f+u,\, v-y\rangle,
		\quad \forall v \in Y_{ad}.
	\end{aligned}
	\right.
\end{equation}
Above, \(A:H_0^1(\Omega)\to H^{-1}(\Omega)\) is a linear elliptic operator, and \(\langle\cdot,\cdot\rangle\) denotes the duality pairing between
\(H_0^1(\Omega)\) and \(H^{-1}(\Omega)\). Other notations are the same as the ones used in \eqref{eq:distribution}. When the operator \(A\) is symmetric, the problem~\eqref{eq:control_evi} can be reformulated in the form of \eqref{general control}. However, this equivalence does not hold for non-symmetric operators \(A\). Nevertheless, Algorithm~\ref{alg:Bilevel Deep Learning Method} remains applicable in such cases.  The key insight lies in that \eqref{eq:control_evi} can also be approximated by a bilevel optimization problem. 

To be concrete, it follows from \cite{gao2025prox} that the EVI in \eqref{eq:control_evi} can be expressed as the fixed-point equation:
$$
P_{Y_{ad}}\left((1-\tau A)y + \tau f + \tau u\right) - y = 0,
$$
where \(\tau > 0\) is a parameter and \(P_{Y_{ad}}\) denotes the projection operator onto \(Y_{ad}\).

By approximating the state \(y\) and the control \(u\) with NNs \(\hat{y}(x;\theta_y)\) and \(\hat{u}(x;\theta_u)\) as described in section~\ref{se: algorithm}, we derive an approximation problem in the form of \eqref{eq: stochastic bi} with the corresponding lower-level loss function given by
\begin{equation}
	e(\theta_y,\theta_u)
	:= \mathbb{E}_{x \sim \mathcal{D}} \left[ \big| P_{Y_{ad}}\left((1-\tau A)\hat{y}(x;\theta_y) + \tau f + \tau \hat{u}(x;\theta_u)\right) - \hat{y}(x;\theta_y) \big|^2 \right].
	\label{eq:lower-level-loss-vi}
\end{equation}
This formulation enables the direct application of Algorithm~\ref{alg:Bilevel Deep Learning Method}.

\begin{remark}
	Building upon the preceding discussion, Algorithm \ref{alg:Bilevel Deep Learning Method} can be generalized to optimal control problems constrained by EVIs—beyond the scope of obstacle problems—by leveraging the results presented in \cite{gao2025prox}. 
\end{remark}

\noindent \textbf{Example 6 (Non-symmetric EVI).} We consider the problem \eqref{eq:control_evi} with a non-symmetric operator \(A\). 
Let \(\Omega=(0,1)^2\), \(\sigma=0.01\), \(Y_{ad}=\{y\in H_0^1(\Omega)\mid y\ge0\text{ a.e.\ in }\Omega\}\), and \(U_{ad}=L^2(\Omega)\). 
Define
\[Ay=-\Delta y+\frac{\partial y}{\partial x_1}-\frac{\partial y}{\partial x_2}, f\left(x_1,  x_2\right)=10\left(\sin \left(2 \pi x_2\right)+x_1\right),~ y_d = x_1(1-x_1)x_2(1-x_2).\]
We set \(\tau=0.01\) and take~\eqref{eq:lower-level-loss-vi} as the lower-level loss. We then implement Algorithm~\ref{alg:Bilevel Deep Learning Method} with $T=20,000, \gamma=200,000$, $c_k=5k^{0.3}$, and \(\alpha=\beta=\eta=2\times10^{-4}\).
Because the lower-level loss \(e\) in \eqref{eq:lower-level-loss-vi} scales with \(\tau\), its magnitude is smaller than the energy-based loss used previously. In practice, to improve the optimization stability, we scale \(e(\theta_y,\theta_u)\) to \(10\,e(\theta_y,\theta_u)\). 

Figure~\ref{fig:meha_evi} illustrates the upper- and lower-level training losses during the bilevel optimization procedure, together with the computed state \(\hat y\) and control \(\hat u\). Both training losses exhibit rapid convergence. We note that the magnitude of the lower-level loss differs from that in the previous examples due to its modified definition. These results demonstrate the strong generalization capability of Algorithm~\ref{alg:Bilevel Deep Learning Method}, as well as its effectiveness and computational efficiency for the optimal control of EVIs.

%
%
%
%
%

\begin{figure}[htpb]
	\centering
	\subfloat[Upper-level loss]{\includegraphics[width=0.45\textwidth]{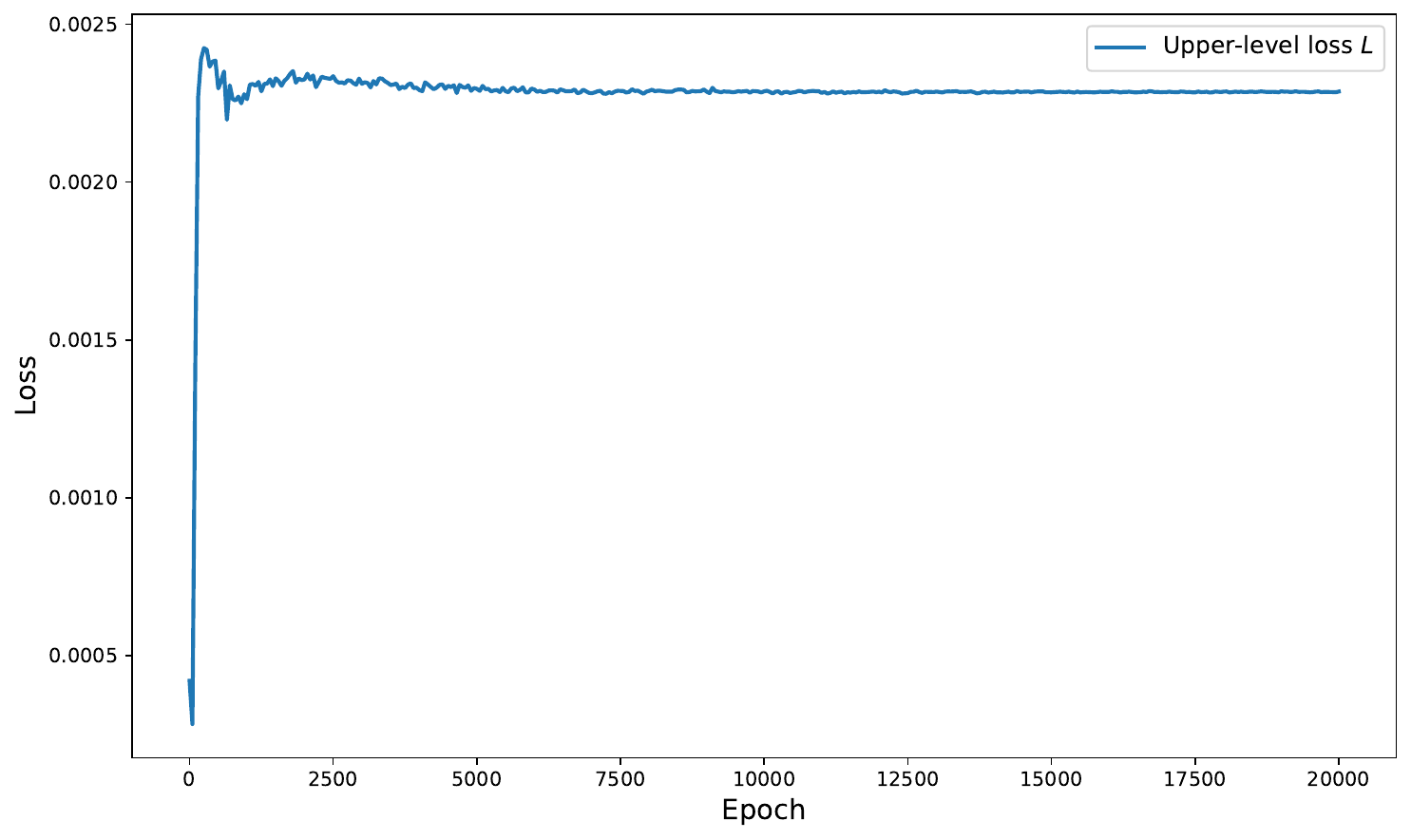}}
	\hfill
	\subfloat[Lower-level loss]{\includegraphics[width=0.45\textwidth]{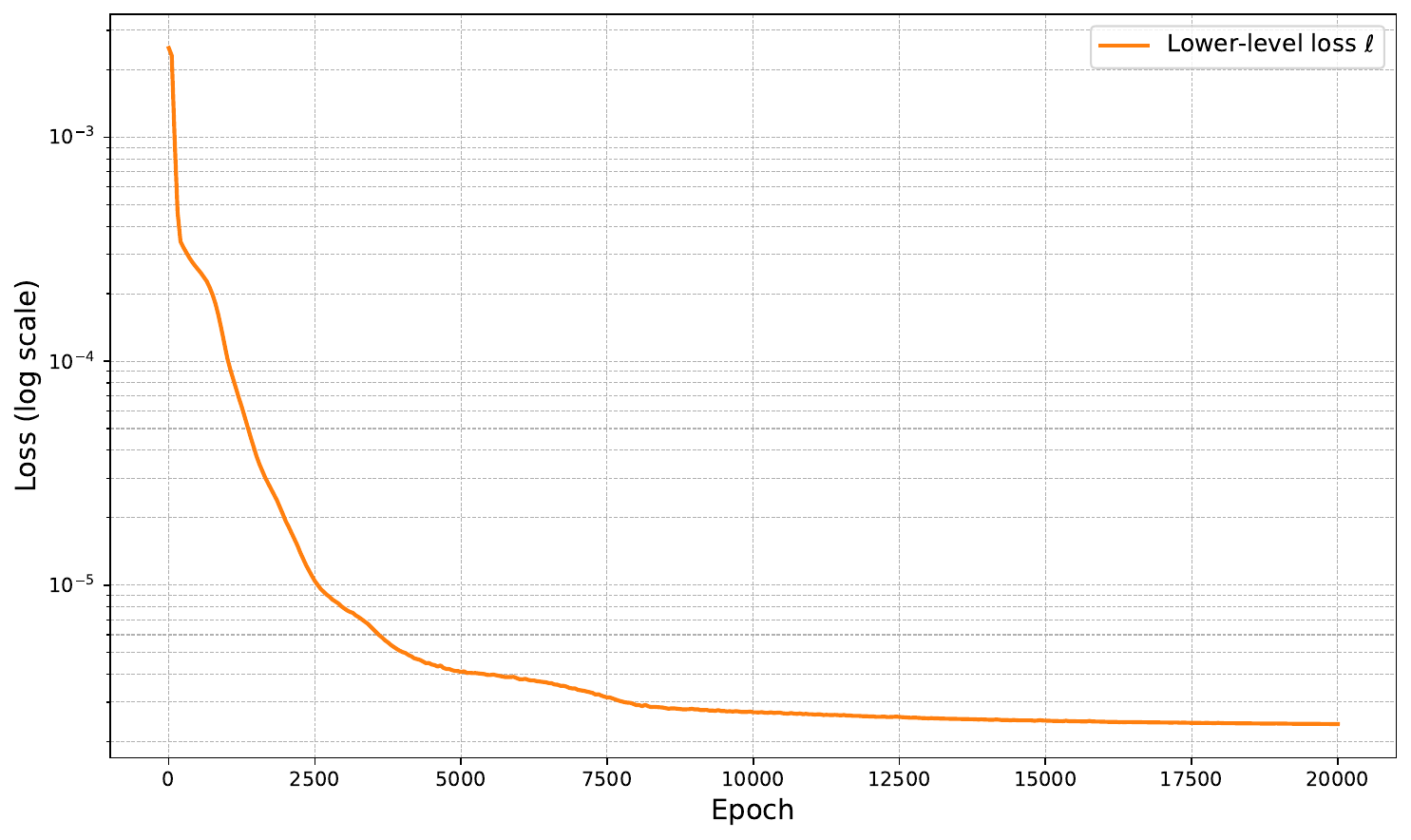}}
	
	\subfloat[Computed state $\hat y$]{\includegraphics[width=0.45\textwidth]{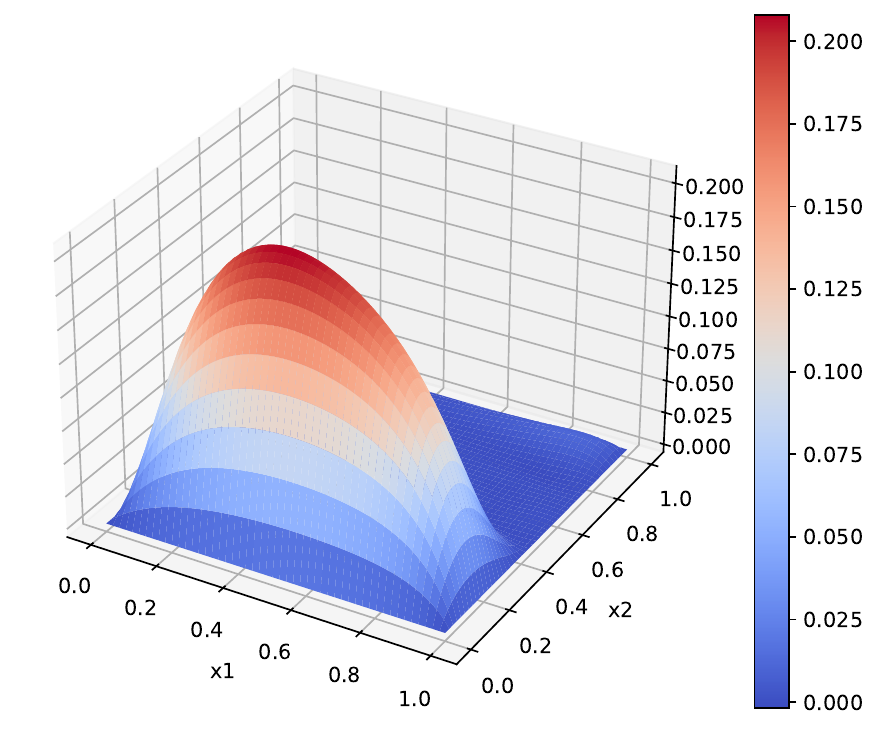}}
	\hfill
	\subfloat[Computed control $\hat u$]{\includegraphics[width=0.45\textwidth]{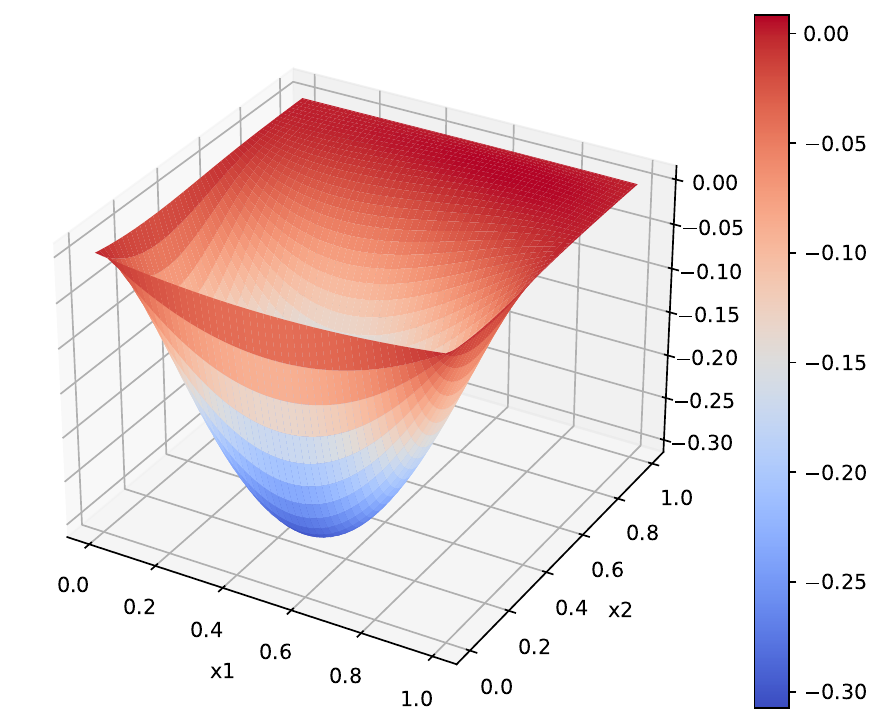}}
	
	\caption{\small \em Numerical results of Algorithm~\ref{alg:Bilevel Deep Learning Method} for {Example 6}}
	\label{fig:meha_evi}
\end{figure}

\newpage

\section{Conclusions and Perspectives}\label{se:conclusion}
In this work, we introduce a bilevel deep learning framework for solving the optimal control of obstacle problems$-$a class of nonsmooth, nonlinear, and hierarchically structured optimization problems that have not been effectively resolved by existing deep learning methods. Our method leverages specially designed neural networks (NNs) to embed constraints directly and approximate the optimal control model by a bilevel optimization problem, thereby avoiding the limitations of objective combination approaches.
The core of our framework is the proposed Single-Loop Stochastic First-Order Bilevel Algorithm (S2-FOBA), which efficiently trains the NNs by solving the bilevel problem in a single-loop fashion. The convergence of S2-FOBA is analyzed under mild assumptions. Our framework and training algorithm explicitly preserve the inherent bilevel structure of the problem. To enhance numerical stability, we further consider a two-stage training strategy that ensures the computed state is sufficiently accurate. 

Extensive numerical experiments demonstrate the effectiveness, flexibility, and robustness of our framework across various settings, including distributed and obstacle controls, regular and irregular obstacles, and complex geometric domains. The framework achieves satisfactory accuracy, outperforms existing deep learning methods, and is more efficient than classical numerical methods on benchmark problems. Overall, this work provides a computationally efficient and scalable learning-based method for optimal control problems and offers a promising direction for mesh-free bilevel optimization methods in scientific computing.

Our work leaves some important questions, which are beyond the scope of the paper and will be the
subject of future investigation. For instance
\begin{itemize}
	\item While S2-FOBA is efficient, its hyperparameter tuning, such as the choice of step sizes and penalty sequences, deserves further automation or theoretical guidance.
	\item Extending the framework to optimal control of time-dependent obstacle problems and stochastic EVIs would broaden its applicability. 
	\item Inspired by the recent success of operator learning in optimal control of PDEs \cite{luo2025efficient,song2023accelerated,song2024operator,wang2021fast}, it is of great interest to design efficient operator learning algorithms for optimal control of obstacle problems and other VIs.
\end{itemize}

\section*{Acknowledgments}

This work is supported by the NTU Start-Up Grant, the National Key R\&D Program of China (2023YFA1011400), the National Natural Science Foundation of China (12326605, 12222106, 12501429), and the Shenzhen Fundamental Research Program (20250530150024003).

\bibliographystyle{elsarticle-num}
\bibliography{references}

@article{hintermuller2011smooth,
  title={A smooth penalty approach and a nonlinear multigrid algorithm for elliptic {MPECs}},
  author={Hinterm{\"u}ller, Michael and Kopacka, Ian},
  journal={Computational Optimization and Applications},
  volume={50},
  pages={111--145},
  year={2011},
  publisher={Springer}
}

@article{hintermuller2009mathematical,
  title={Mathematical programs with complementarity constraints in function space: {C-and} strong stationarity and a path-following algorithm},
  author={Hinterm{\"u}ller, Michael and Kopacka, Ian},
  journal={SIAM Journal on Optimization},
  volume={20},
  number={2},
  pages={868--902},
  year={2009},
  publisher={SIAM}
}

@article{ito2007optimal,
  title={Optimal control of obstacle problems by ${H^1}$-obstacles},
  author={Ito, Kazufumi and Kunisch, Karl},
  journal={Applied Mathematics and Optimization},
  volume={56},
  number={1},
  pages={1--17},
  year={2007},
  publisher={Springer}
}

@article{bergounioux1997use,
  title={Use of augmented {Lagrangian} methods for the optimal control of obstacle problems},
  author={Bergounioux, Ma{\"\i}tine},
  journal={Journal of Optimization Theory and Applications},
  volume={95},
  pages={101--126},
  year={1997},
  publisher={Springer}
}

@article{mowlavi2023optimal,
  title={Optimal control of {PDEs} using physics-informed neural networks},
  author={Mowlavi, Saviz and Nabi, Saleh},
  journal={Journal of Computational Physics},
  volume={473},
  pages={111731},
  year={2023},
  publisher={Elsevier}
}

@article{hintermuller2008active,
  title={An active-set equality constrained {Newton} solver with feasibility restoration for inverse coefficient problems in elliptic variational inequalities},
  author={Hinterm{\"u}ller, Michael},
  journal={Inverse Problems},
  volume={24},
  number={3},
  pages={034017},
  year={2008},
  publisher={IOP Publishing}
}

@article{mignot1984optimal,
  title={Optimal control in some variational inequalities},
  author={Mignot, Fulbert and Puel, Jean Pierre},
  journal={SIAM Journal on Control and Optimization},
  volume={22},
  number={3},
  pages={466--476},
  year={1984},
  publisher={SIAM}
}

@article{gao2023moreau,
  title={Moreau envelope based difference-of-weakly-convex reformulation and algorithm for bilevel programs},
  author={Gao, Lucy L and Ye, Jane J and Yin, Haian and Zeng, Shangzhi and Zhang, Jin},
  journal={arXiv preprint arXiv:2306.16761},
  year={2023}
}

@article{yu2018deep,
  title={The deep {Ritz} method: {A} deep learning-based numerical algorithm for solving variational problems},
  author={E, Weinan and Yu, Bing},
  journal={Communications in Mathematics and Statistics},
  volume={6},
  number={1},
  pages={1--12},
  year={2018},
  publisher={Springer}
}

@article{ito2003semi,
  title={Semi-Smooth {Newton} Methods for Variational Inequalitiesof the First Kind},
  author={Ito, Kazufumi and Kunisch, Karl},
  journal={ESAIM: Mathematical Modelling and Numerical Analysis},
  volume={37},
  number={1},
  pages={41--62},
  year={2003},
  publisher={EDP Sciences}
}

@article{barbu1984optimal,
  title={Optimal control of variational inequalities},
  author={Barbu, Viorel},
  journal={Research Notes in Math.},
  volume={100},
  year={1984},
  publisher={Pitman}
}

@incollection{hintermuller2015ell,
  title={An $\ell^1$-penalty scheme for the optimal control of elliptic variational inequalities},
  author={Hinterm{\"u}ller, Michael and L{\"o}bhard, C and Tber, MH},
  booktitle={Numerical Analysis and Optimization: NAO-III, Muscat, Oman, January 2014},
  pages={151--190},
  year={2015},
  publisher={Springer}
}

@article{brett2015mesh,
  title={Mesh adaptivity in optimal control of elliptic variational inequalities with point-tracking of the state},
  author={Brett, Charles and Elliott, Charles M and Hinterm{\"u}ller, Michael and L{\"o}bhard, Caroline},
  journal={Interfaces and Free Boundaries},
  volume={17},
  number={1},
  pages={21--53},
  year={2015}
}

@incollection{gaevskaya2014adaptive,
  title={Adaptive finite elements for optimally controlled elliptic variational inequalities of obstacle type},
  author={Gaevskaya, Alexandra and Hinterm{\"u}ller, Michael and Hoppe, Ronald HW and L{\"o}bhard, C},
  booktitle={Optimization with PDE Constraints: ESF networking program'OPTPDE'},
  pages={95--150},
  year={2014},
  publisher={Springer}
}

@article{hintermuller2014dual,
  title={Dual-weighted goal-oriented adaptive finite elements for optimal control of elliptic variational inequalities},
  author={Hinterm{\"u}ller, Michael and Hoppe, Ronald HW and L{\"o}bhard, Caroline},
  journal={ESAIM: Control, Optimisation and Calculus of Variations},
  volume={20},
  number={2},
  pages={524--546},
  year={2014}
}

@article{meyer2015adaptive,
  title={Adaptive optimal control of the obstacle problem},
  author={Meyer, Ch and Rademacher, Andreas and Wollner, Winnifried},
  journal={SIAM Journal on Scientific Computing},
  volume={37},
  number={2},
  pages={A918--A945},
  year={2015},
  publisher={SIAM}
}

@article{cybenko1989approximation,
  title={Approximation by superpositions of a sigmoidal function},
  author={Cybenko, George},
  journal={Mathematics of Control, Signals and Systems},
  volume={2},
  number={4},
  pages={303--314},
  year={1989},
  publisher={Springer}
}

@article{hornik1991approximation,
  title={Approximation capabilities of multilayer feedforward networks},
  author={Hornik, Kurt},
  journal={Neural Networks},
  volume={4},
  number={2},
  pages={251--257},
  year={1991},
  publisher={Elsevier}
}

@article{hornik1989multilayer,
  title={Multilayer feedforward networks are universal approximators},
  author={Hornik, Kurt and Stinchcombe, Maxwell and White, Halbert},
  journal={Neural Networks},
  volume={2},
  number={5},
  pages={359--366},
  year={1989},
  publisher={Elsevier}
}

@article{raissi2019physics,
  title={Physics-informed neural networks: A deep learning framework for solving forward and inverse problems involving nonlinear partial differential equations},
  author={Raissi, Maziar and Perdikaris, Paris and Karniadakis, George E},
  journal={Journal of Computational Physics},
  volume={378},
  pages={686--707},
  year={2019},
  publisher={Elsevier}
}

@article{sirignano2018dgm,
  title={{DGM}: A deep learning algorithm for solving partial differential equations},
  author={Sirignano, Justin and Spiliopoulos, Konstantinos},
  journal={Journal of Computational Physics},
  volume={375},
  pages={1339--1364},
  year={2018},
  publisher={Elsevier}
}

@article{zhao2022two,
  title={Two neural-network-based methods for solving elliptic obstacle problems},
  author={Zhao, Xinyue Evelyn and Hao, Wenrui and Hu, Bei},
  journal={Chaos, Solitons \& Fractals},
  volume={161},
  pages={112313},
  year={2022},
  publisher={Elsevier}
}

@article{el2025physics,
  title={A physics-informed neural network framework for modeling obstacle-related equations},
  author={El Bahja, Hamid and Hauffen, Jan C and Jung, Peter and Bah, Bubacarr and Karambal, Issa},
  journal={Nonlinear Dynamics},
  volume={113},
  pages={1--12},
  year={2025},
  publisher={Springer}
}

@article{barry2022physics,
  title={Physics-informed neural networks for {PDE-constrained} optimization and control},
  author={Barry-Straume, Jostein and Sarshar, Arash and Popov, Andrey A and Sandu, Adrian},
  journal={Communications on Applied Mathematics and Computation},
  pages={1--24},
  year={2025},
  publisher={Springer}
}

@article{lu2021physics,
  title={Physics-informed neural networks with hard constraints for inverse design},
  author={Lu, Lu and Pestourie, Raphael and Yao, Wenjie and Wang, Zhicheng and Verdugo, Francesc and Johnson, Steven G},
  journal={SIAM Journal on Scientific Computing},
  volume={43},
  number={6},
  pages={B1105--B1132},
  year={2021},
  publisher={SIAM}
}

@article{song2024admm,
  title={The {ADMM-PINNs} algorithmic framework for nonsmooth {PDE-constrained} optimization: {A} deep learning approach},
  author={Song, Yongcun and Yuan, Xiaoming and Yue, Hangrui},
  journal={SIAM Journal on Scientific Computing},
  volume={46},
  number={6},
  pages={C659--C687},
  year={2024},
  publisher={SIAM}
}

@article{gao2025prox,
  title={{Prox-PINNs}: A Deep Learning Algorithmic Framework for Elliptic Variational Inequalities},
  author={Gao, Yu and Song, Yongcun and Tan, Zhiyu and Yue, Hangrui and Zeng, Shangzhi},
  journal={arXiv preprint arXiv:2505.14430},
  year={2025}
}

@book{kinderlehrer2000introduction,
  title={{An Introduction to Variational Inequalities and Their Applications}},
  author={Kinderlehrer, David and Stampacchia, Guido},
  year={2000},
  publisher={SIAM}
}

@article{jaillet1990variational,
  title={Variational inequalities and the pricing of {American} options},
  author={Jaillet, Patrick and Lamberton, Damien and Lapeyre, Bernard},
  journal={Acta Applicandae Mathematica},
  volume={21},
  pages={263--289},
  year={1990},
  publisher={Springer}
}

@article{hao2022bi,
  title={Bi-level physics-informed neural networks for {PDE} constrained optimization using {Broyden's} hypergradients},
  author={Hao, Zhongkai and Ying, Chengyang and Su, Hang and Zhu, Jun and Song, Jian and Cheng, Ze},
  journal={The 11th International Conference on Learning Representations},
  year={2023}
}

@inproceedings{liu2024moreau,
  title={Moreau Envelope for Nonconvex Bi-Level Optimization: A Single-Loop and {Hessian}-Free Solution Strategy},
  author={Liu, Risheng and Liu, Zhu and Yao, Wei and Zeng, Shangzhi and Zhang, Jin},
  booktitle={International Conference on Machine Learning},
  pages={31566--31596},
  year={2024},
  organization={PMLR}
}

@incollection{surowiec2018numerical,
	title={Numerical optimization methods for the optimal control of elliptic variational inequalities},
	author={Surowiec, Thomas M},
	booktitle={Frontiers in PDE-Constrained Optimization},
	pages={123--170},
	year={2018},
	publisher={Springer}
}

@article{meyer2013priori,
	title={A priori finite element error analysis for optimal control of the obstacle problem},
	author={Meyer, Christian and Thoma, Oliver},
	journal={SIAM Journal on Numerical Analysis},
	volume={51},
	number={1},
	pages={605--628},
	year={2013},
	publisher={SIAM}
}

@article{mignot1976controle,
	title={Contr{\^o}le dans les in{\'e}quations variationelles elliptiques},
	author={Mignot, Fulbert},
	journal={Journal of Functional Analysis},
	volume={22},
	number={2},
	pages={130--185},
	year={1976},
	publisher={Elsevier}
}

@article{schiela2013convergence,
	title={Convergence analysis of smoothing methods for optimal control of stationary variational inequalities with control constraints},
	author={Schiela, Anton and Wachsmuth, Daniel},
	journal={ESAIM: Mathematical Modelling and Numerical Analysis},
	volume={47},
	number={3},
	pages={771--787},
	year={2013},
	publisher={EDP Sciences}
}

@article{hintermuller2016bundle,
	title={A bundle-free implicit programming approach for a class of elliptic {MPECs} in function space},
	author={Hinterm{\"u}ller, Michael and Surowiec, Thomas},
	journal={Mathematical Programming},
	volume={160},
	number={1},
	pages={271--305},
	year={2016},
	publisher={Springer}
}

@article{alphonse2024neural,
	title={A neural network approach to learning solutions of a class of elliptic variational inequalities},
	author={Alphonse, Amal and Hinterm{\"u}ller, Michael and Kister, Alexander and Lun, Chin Hang and Sirotenko, Clemens},
	journal={arXiv preprint arXiv:2411.18565},
	year={2024}
}

@article{cheng2023deep,
	title={A deep neural network-based method for solving obstacle problems},
	author={Cheng, Xiaoliang and Shen, Xing and Wang, Xilu and Liang, Kewei},
	journal={Nonlinear Analysis: Real World Applications},
	volume={72},
	pages={103864},
	year={2023},
	publisher={Elsevier}
}

@article{combettes2020deep,
	title={Deep neural network structures solving variational inequalities},
	author={Combettes, Patrick L and Pesquet, Jean-Christophe},
	journal={Set-Valued and Variational Analysis},
	volume={28},
	number={3},
	pages={491--518},
	year={2020},
	publisher={Springer}
}

@incollection{darehmiraki2022deep,
	title={A deep learning approach for the obstacle problem},
	author={Darehmiraki, Majid},
	booktitle={Proceedings of Academia-Industry Consortium for Data Science: AICDS 2020},
	pages={179--188},
	year={2022},
	publisher={Springer}
}

@article{lai2025hard,
	title={The hard-constraint {PINNs} for interface optimal control problems},
	author={Lai, Ming-Chih and Song, Yongcun and Yuan, Xiaoming and Yue, Hangrui and Zeng, Tianyou},
	journal={SIAM Journal on Scientific Computing},
	volume={47},
	number={3},
	pages={C601--C629},
	year={2025},
	publisher={SIAM}
}

@book{glowinski2015variational,
	title={{Variational Methods for the Numerical Solution of Nonlinear Elliptic Problems}},
	author={Glowinski, Roland},
	year={2015},
	publisher={SIAM}
}

@book{glowinski2008lectures,
	title={{Lectures on Numerical Methods for Non-Linear Variational Problems}},
	author={Glowinski, R},
	year={2008},
	publisher={Springer Science \& Business Media}
}

@article{cryer1971method,
	title={The method of {Christopherson} for solving free boundary problems for infinite journal bearings by means of finite differences},
	author={Cryer, Colin W},
	journal={Mathematics of Computation},
	volume={25},
	number={115},
	pages={435--443},
	year={1971}
}

@article{bourgat1977numerical,
	title={Numerical analysis of flow with or without wake past a symmetric two-dimensional profile without incidence},
	author={Bourgat, JF and Duvaut, G},
	journal={International Journal for Numerical Methods in Engineering},
	volume={11},
	number={6},
	pages={975--993},
	year={1977},
	publisher={Wiley Online Library}
}

@article{brezis1976hodograph,
	title={The hodograph method in fluid-dynamics in the light of variational inequalities},
	author={Brezis, Ha{\"\i}m and Stampacchia, Guido},
	journal={Archive for Rational Mechanics and Analysis},
	volume={61},
	number={1},
	pages={1--18},
	year={1976},
	publisher={Springer New York}
}

@article{wachsmuth2014strong,
	title={Strong stationarity for optimal control of the obstacle problem with control constraints},
	author={Wachsmuth, Gerd},
	journal={SIAM Journal on Optimization},
	volume={24},
	number={4},
	pages={1914--1932},
	year={2014},
	publisher={SIAM}
}

@article{wachsmuth2016towards,
	title={Towards {M-stationarity} for optimal control of the obstacle problem with control constraints},
	author={Wachsmuth, Gerd},
	journal={SIAM Journal on Control and Optimization},
	volume={54},
	number={2},
	pages={964--986},
	year={2016},
	publisher={SIAM}
}

@article{bergounioux2004optimal,
	title={Optimal control of bilateral obstacle problems},
	author={Bergounioux, Ma{\"\i}tine and Lenhart, Suzanne},
	journal={SIAM Journal on Control and Optimization},
	volume={43},
	number={1},
	pages={240--255},
	year={2004},
	publisher={SIAM}
}

@article{cao2025adversarial,
	title={Adversarial Physics-informed Neural Networks with Hard Constraints for Optimal Control of {PDEs}},
	author={Cao, Yuandong and So, Chi Chiu and Dai, Yifan and Yung, Siu Pang and Wang, Jun-Min},
	journal={Journal of Computational Physics},
	volume = {541},
	pages={114307},
	year={2025},
	publisher={Elsevier}
}

@article{song2024operator,
	title={An Operator Learning Approach to Nonsmooth Optimal Control of Nonlinear {PDEs}},
	author={Song, Yongcun and Yuan, Xiaoming and Yue, Hangrui and Zeng, Tianyou},
	journal={arXiv preprint arXiv:2409.14417},
	year={2024}
}

@article{song2023accelerated,
	title={Accelerated primal-dual methods with enlarged step sizes and operator learning for nonsmooth optimal control problems},
	author={Song, Yongcun and Yuan, Xiaoming and Yue, Hangrui},
	journal={arXiv preprint arXiv:2307.00296},
	year={2023}
}

@article{wang2021fast,
	title={Fast {PDE-constrained} optimization via self-supervised operator learning},
	author={Wang, Sifan and Bhouri, Mohamed Aziz and Perdikaris, Paris},
	journal={arXiv preprint arXiv:2110.13297},
	year={2021}
}

@article{luo2025efficient,
	title={Efficient {PDE-constrained} optimization under high-dimensional uncertainty using derivative-informed neural operators},
	author={Luo, Dingcheng and O’Leary-Roseberry, Thomas and Chen, Peng and Ghattas, Omar},
	journal={SIAM Journal on Scientific Computing},
	volume={47},
	number={4},
	pages={C899--C931},
	year={2025},
	publisher={SIAM}
}

@article{chen2021closing,
	title={Closing the gap: Tighter analysis of alternating stochastic gradient methods for bilevel problems},
	author={Chen, Tianyi and Sun, Yuejiao and Yin, Wotao},
	journal={Advances in Neural Information Processing Systems},
	volume={34},
	pages={25294--25307},
	year={2021}
}

@article{nemirovski2009robust,
	title={Robust stochastic approximation approach to stochastic programming},
	author={Nemirovski, Arkadi and Juditsky, Anatoli and Lan, Guanghui and Shapiro, Alexander},
	journal={SIAM Journal on Optimization},
	volume={19},
	number={4},
	pages={1574--1609},
	year={2009},
	publisher={SIAM}
}

@article{de2018optimal,
	title={On the optimal control of some nonsmooth distributed parameter systems arising in mechanics},
	author={De Los Reyes, Juan-Carlos},
	journal={GAMM-Mitteilungen},
	volume={40},
	number={4},
	pages={268--286},
	year={2018},
	publisher={Wiley Online Library}
}

@book{friedman2010variational,
	title={Variational Principles and Free-Boundary Problems},
	author={Friedman, A.},
	isbn={9780486478531},
	lccn={2010024126},
	series={Dover Books on Mathematics},
	year={2010},
	publisher={Dover Publications}
}

@article{cho2023separable,
  title={Separable physics-informed neural networks},
  author={Cho, Junwoo and Nam, Seungtae and Yang, Hyunmo and Yun, Seok-Bae and Hong, Youngjoon and Park, Eunbyung},
  journal={Advances in Neural Information Processing Systems},
  volume={36},
  pages={23761--23788},
  year={2023}
}

\end{document}